\documentclass[12pt,reqno]{amsart}
\numberwithin{equation}{section}
\usepackage{amssymb}
\usepackage[sc,osf]{mathpazo} 
\usepackage{fullpage}
\usepackage{tikz}
\usetikzlibrary{patterns}
\definecolor{linkred}{rgb}{0.7,0.2,0.2}
\definecolor{linkblue}{rgb}{0,0.2,0.6}
\definecolor{linkgreen}{rgb}{0,0.6,0.2}
\usepackage[colorlinks,plainpages,backref,
    linkcolor=linkblue,
    citecolor=linkgreen,
    urlcolor=linkred]{hyperref}

\usepackage[all]{xy}

\newtheorem{theorem}{Theorem}[section]
\newtheorem{lemma}[theorem]{Lemma}
\newtheorem{prop}[theorem]{Proposition}
\newtheorem{coro}[theorem]{Corollary}

\theoremstyle{definition}
\newtheorem{definition}[theorem]{Definition}
\newtheorem{example}[theorem]{Example}
\newtheorem{remark}[theorem]{Remark}

\newcommand{\BPD}[2][1.2pc]{%
\setlength{\unitlength}{#1}
\definecolor{lightcyan}{rgb}{0.8,1,1}%
\def\FF{%
    \qbezier(0.5,0)(0.5,0.5)(1,0.5)    
}
\def\JJ{%
    \qbezier(0.5,1)(0.5,0.5)(0,0.5)
    }
\def\II{%
    \qbezier(0.5,0)(0.5,0.5)(0.5,1)}%
\def\HH{%
    \qbezier(0,0.5)(0.5,0.5)(1,0.5)}%
\def\XX{\NN\HH}
\def\NN{%
    \qbezier(0.5,0)(0.5,0.3)(0.5,0.3)
    \qbezier(0.5,1)(0.5,0.7)(0.5,0.7)}%
\def\BPDfr##1{%
\begin{picture}(1,1)%
    \linethickness{0.08\unitlength}
    ##1
    \thinlines
    \color{lightgray}%
    \put(0,0){\line(0,1){1}}%
    \put(1,0){\line(0,1){1}}%
    \put(0,0){\line(1,0){1}}%
    \put(0,1){\line(1,0){1}}%
\end{picture}}
\def\BPDfrc##1{\BPDfr{\put(0,0){\color{lightcyan}\rule{\unitlength}{\unitlength}}##1}}
\def\x{\BPDfrc{\XX}}
\def\f{\BPDfrc{\FF}}
\def\o{\BPDfrc{}}
\def\j{\BPDfrc{\JJ}}
\def\h{\BPDfrc{\HH}}
\def\i{\BPDfrc{\II}}
\def\O{\BPDfr{}}
\def\X{\BPDfr{\XX}}
\def\BX{\BPDfr{\II\HH}}
\def\F{\BPDfr{\FF}}
\def\J{\BPDfr{\JJ}}
\def\H{\BPDfr{\HH}}
\def\I{\BPDfr{\II}}
\def\B{\BPDfr{\JJ\FF}}
\def\M##1{\begin{picture}(1,1)%
    \put(0,0.2){\makebox[\unitlength]{$##1$}}
\end{picture}}
\begin{array}{@{\,}c@{\,}}
\def\dots{\scriptstyle\cdots}
{\def\arraystretch{0}
\setlength{\arraycolsep}{0pc}
\color{teal}
\begin{array}{@{}l@{}}%
#2\end{array}}
\end{array}}

\newcommand{\clan}[2][1.2pc]{%
\setlength{\unitlength}{\dimexpr#1/2}%
\def\drawclan##1{%
\expandafter%
    \drawclanbegin.##1%
    \drawclanend\drawclanendend}
\def\drawclanbegin.##1##2\drawclanendend{%
    \ifx\drawclanend##1%
        {\relax}%
    \else%
    \clandot{##1}%
    \expandafter%
        \drawclanbegin.##2\drawclanendend%
\fi}
\def\clandot##1{%
    \if+##1
    \begin{picture}(2,1)
        \linethickness{0.15\unitlength}
        \color{red}
        \qbezier(0.4,0.4)(1,0.4)(1.6,0.4)
        \qbezier(1,1.0)(1,0.4)(1,-.2)
    \end{picture}
    \else\if-##1
    \begin{picture}(2,1)
        \linethickness{0.15\unitlength}
        \color{blue}
        \qbezier(0.5,0.4)(1,0.4)(1.5,0.4)
    \end{picture}
    \else\if.##1
    \begin{picture}(2,1)
        \color{teal}
        \put(1,0.4){\circle*{0.7}}
    \end{picture}
    \else
        \def\inpA{##1}\def\inpB{\dots}%
    \ifx\inpA\inpB%
    \begin{picture}(2,1)
        \color{teal}
        \put(0.5,0.4){\circle*{0.3}}
        \put(1,0.4){\circle*{0.3}}
        \put(1.5,0.4){\circle*{0.3}}
    \end{picture}
    \else
    \def\inpA{##1}\def\inpB{}%
    \ifx\inpA\inpB%
        \relax%
    \else
        \def\inpA{##1}\def\inpB{\pm}%
    \ifx\inpA\inpB%
    \begin{picture}(2,1)
        \linethickness{0.15\unitlength}
          \color{red}
          \qbezier(0.4,0.4)(1,0.4)(1.6,0.4)
          \qbezier(1,1.0)(1,0.4)(1,-.2)
          \color{blue}
          \qbezier(0.4,-.2)(1,-.2)(1.6,-.2)
    \end{picture}
    \else
        \def\inpA{##1}\def\inpB{\mp}%
    \ifx\inpA\inpB%
    \begin{picture}(2,1)
        \linethickness{0.15\unitlength}
        \color{red}
        \qbezier(0.4,0.4)(1,0.4)(1.6,0.4)
        \qbezier(1,1.0)(1,0.4)(1,-.2)
        \color{blue}
        \qbezier(0.4,1.0)(1,1.0)(1.6,1.)
    \end{picture}
    \else\if,##1
    \begin{picture}(2,1)
        \linethickness{0.15\unitlength}
        \color{teal!50!white}
        \put(1,0.4){\circle*{0.7}}
    \end{picture}
    \else
    \begin{picture}(2,1)
        \color{white}\linethickness{4pt}
        \qbezier(1,0.4)({\numexpr(##1)+1},{\numexpr(##1)+1})({\numexpr(##1)*2+1},0.4)%
        \color{black}\linethickness{0.8pt}
        \qbezier(1,0.4)({\numexpr(##1)+1},{\numexpr(##1)+1})({\numexpr(##1)*2+1},0.4)%
        \color{teal}
        \put(1,0.4){\circle*{0.7}}
    \end{picture}
    \rule{0pc}{\dimexpr\unitlength*(##1)/2+\unitlength}%
    \fi\fi\fi\fi\fi\fi\fi\fi%
}%
{\drawclan{#2}}%
}%

\DeclareFontEncoding{LS1}{}{}
\DeclareFontSubstitution{LS1}{stix}{m}{n}
\DeclareSymbolFont{stixletters}{LS1}{stix}{m}{it}
\DeclareMathAccent{\cev}{\mathord}{stixletters}{"91}
\DeclareMathAccent{\vec}{\mathord}{stixletters}{"92}

\title{Equivariant Schubert Calculus for\\
Inverse Grassmannian Permutations}

\author{Yiming Chen}
\address[Yiming Chen, Neil J.Y. Fan, Ming Yao]{Department of Mathematics, 
Sichuan University, Chengdu, Sichuan 610065, P.R. China}
\email{ym\_chen@stu.scu.edu.cn, fan@scu.edu.cn, yaom@stu.scu.edu.cn}

\author{Neil J.Y. Fan}
\author{Rui Xiong}
\address[Rui Xiong]{Department of Mathematics and Statistics, University of Ottawa, 150 Louis-Pasteur, Ottawa, ON, K1N 6N5, Canada}
\email{rxion043@uottawa.ca}
\author{Ming Yao}

\begin{document}

\def\vv{\mathbf{v}}

\begin{abstract}
We give a Graham-positive expansion for the product of two double Schubert polynomials indexed by two inverse Grassmannian permutations. Surprisingly, the nonzero structure constants are double Schubert polynomials in two disjoint sets of equivariant variables. 
We also give a positive expansion for the product of two single Schubert polynomials indexed by a $321$-avoiding permutation (e.g., a Grassmannian permutation) and an inverse Grassmannian permutation. Unexpectedly, the nonzero structure constants are Edelman--Greene coefficients. 

\end{abstract}

\maketitle

\setcounter{tocdepth}{1}
\tableofcontents

\section{Introduction}

The central problem in equivariant Schubert calculus is to find manifestly positive combinatorial models for the coefficients $c_{u,v}^w(t)$ in the expansion
\begin{equation}\label{eq:cuvw}
\mathfrak{S}_u(x;t)\cdot \mathfrak{S}_{v}(x;t)
=\sum_{w\in S_\infty} c_{u,v}^w(t)\cdot\mathfrak{S}_{w}(x;t),
\end{equation}
where $\mathfrak{S}_w(x;t)$ are double Schubert polynomials  indexed by permutations $w\in S_\infty = \bigcup_{n\geq 0} S_n$. 
Geometrically, these coefficients $c_{u,v}^w(t)$ correspond to the structure constants of equivariant cohomology of flag varieties. 
These coefficients are known for very limited cases, just to name some famous rules: 
\begin{itemize}
   \item When one of $u,v$ is $s_{k-r+1}\cdots s_{k-1}s_k$ for some $1\leq r\leq k$, the expansion \eqref{eq:cuvw} is known as  the equivariant Pieri rule \cite{Robinson,LSY}; 
    \item When $u,v$ are both Grassmannian permutations \cite{KT03,MS99,PeYo15-1,PeYo15-2,Wheeler}, and more generally of separated descents,     these expansions admit  puzzle rules \cite{Pzz3}. 
    \item There is also a rule for two-step flag varieties due to Buch \cite{Buch}; see also Knutson and Zinn-Justin \cite{Pzz1}. 
\end{itemize}
In this paper, we  give a new family of positive combinatorial models of the expansion \eqref{eq:cuvw} under the assumption 
\begin{equation}\label{eq:invGr}
\text{$u,v$ are both inverse Grassmannian permutations}, 
\end{equation}
that is, both $u^{-1}$ and $v^{-1}$ have at most one descent.

\subsection{Non-equivariant Rule and the Symmetric Subgroup $K$}
When $u,v$ are inverse Grassmannian permutations, there is a model for the   single Schubert structure constants  $c_{u,v}^w(0)$  obtained by Pechenik and Weigandt \cite{Pandt} based on the result of Wyser \cite{Wyser}. 
Let us sketch the main idea of their proof and explain the essential obstacle of generalizing to the equivariant setting. 

Let $p,q$ be the descents (if any) of $u^{-1}$ and $v^{-1}$, respectively. 
Let $n=p+q$. 
Consider the following subgroups of $G= GL_{n}$
$$\def\mygrid{
\draw(0,2.5) to (4.5,2.5);
\draw(2.0,0) to (2.0,4.5);
\draw[thick](0.1,0) to (-.1,0) to (-.1,4.5) to (0.1,4.5);
\draw[thick](4.4,0) to (4.6,0) to (4.6,4.5) to (4.4,4.5);}
P = 
\begin{matrix}
\begin{tikzpicture}[scale=0.5]
\mygrid
\fill[pattern=north east lines, pattern color=gray] 
    ({0},{4.5}) rectangle ({2},{2.5}); 
\fill[pattern=north east lines, pattern color=gray] 
    ({0},{2.5}) rectangle ({4.5},{0}); 
\end{tikzpicture}
\end{matrix}\,,
\qquad 
Q = 
\begin{matrix}
\begin{tikzpicture}[scale=0.5]
\mygrid
\fill[pattern=north east lines, pattern color=gray] 
    ({0},{4.5}) rectangle ({4.5},{2.5}); 
\fill[pattern=north east lines, pattern color=gray] 
    ({2},{2.5}) rectangle ({4.5},{0}); 
\end{tikzpicture}
\end{matrix}\,,
\qquad 
K = 
\begin{matrix}
\begin{tikzpicture}[scale=0.5]
\mygrid
\fill[pattern=north east lines, pattern color=gray] 
    ({0},{4.5}) rectangle ({2},{2.5}); 
\fill[pattern=north east lines, pattern color=gray] 
    ({2},{2.5}) rectangle ({4.5},{0}); 
\end{tikzpicture}
\end{matrix}\,,
$$
where  the matrices are partitioned into blocks of sizes $p$ and  $q$. 
The subgroup $K$ is a symmetric subgroup, that is, it is the invariant subgroup of an involution on $G$. 
In particular, $K$ is a spherical subgroup, i.e., it acts on the flag variety $G/B$ with finitely many orbits.

We first assume $u,v\in S_{p+q}$. 
The opposite Schubert variety $\overline{B^-uB/B}$ is $P$-invariant, and 
the Schubert variety $\overline{Bw_0vB/B}$ is $Q$-invariant. 
Thus the Richardson variety $R_{u,w_0v}=\overline{B{w_0v}B/B}\cap \overline{B^-uB/B}$ is $K$-invariant. 
Since there are only finitely many $K$-orbits on $G/B$, the Richardson variety $R_{u,w_0v}$ is necessarily a $K$-orbit closure. 
By Brion \cite{Brion}, there is a general formula for the non-equivariant Schubert expansion of a spherical subgroup orbit closure. This yields an explicit formula of $c_{u,v}^w(0)$ for $w\in S_n$. 
More generally, to address the case when $u,v,w\notin S_{p+q}$, Pechenik and Weigandt \cite{Pandt} applied the (back) stability of the structure constants of single Schubert polynomials to reduce to the $u,v,w\in S_{p+q}$ case.

The equivariant Schubert expansion of $K$-orbit closures was obtained recently by the authors \cite{WE}. Unfortunately, it does not directly apply to the equivariant product in \eqref{eq:cuvw}. 
The essential difficulty  lies in the fact that: 
\begin{equation}\label{eq:diff}
\begin{matrix}
\text{the product $\mathfrak{S}_u(x;t)\cdot \mathfrak{S}_{v}(x;t)$ } 
\text{
does not represent $R_{u,w_0v}$.}
\end{matrix}
\end{equation}
To overcome this difficulty, some new ideas must be introduced. 

\subsection{Triple Schubert Calculus and the Spherical Subgroup $S$}
Motivated by Molev and Sagan \cite{MS99}, Knutson and Tao \cite{KT03} pioneered the study of triple Schubert calculus, 
i.e., the product of two Schubert polynomials with two different secondary variables
\begin{equation}\label{eq:cuvwtri}
\mathfrak{S}_u(x;t)\cdot \mathfrak{S}_{v}(x;y)
=\sum_{w\in S_\infty} c_{u,v}^w(t;y)
\cdot\mathfrak{S}_{w}(x;t).
\end{equation}
The second author, Guo and the third author \cite{FGX} established a puzzle rule of $c_{u,v}^w(t;y)$ in the separated descents case. 
In \cite{Gaox}, Gao and the third author proved the positivity conjecture of Samuel \cite{Sam24} on the coefficients $c_{u,v}^w(t;y)$ via establishing a new geometric explanation of the product. 
The key idea of \cite{Gaox} is to realize \eqref{eq:cuvwtri} as the equivariant Schubert expansion of a Richardson variety.

Lifting to triple Schubert calculus will solve the difficulty \eqref{eq:diff} above, but at the same time we will have to replace the back stability by a geometric argument. 
It turns out that it is necessary to work with the following subgroups of $GL_{n}$ for $n\gg 0$
$$\def\mygrid{\draw(0,3.5) to (4.5,3.5);
\draw(0,1.5) to (4.5,1.5);
\draw(1.0,0) to (1.0,4.5);
\draw(3.0,0) to (3.0,4.5);
\draw[thick](0.1,0) to (-.1,0) to (-.1,4.5) to (0.1,4.5);
\draw[thick](4.4,0) to (4.6,0) to (4.6,4.5) to (4.4,4.5);}
P = 
\begin{matrix}
\begin{tikzpicture}[scale=0.5]
\mygrid
\fill[pattern=north east lines, pattern color=gray] 
    ({0},{4.5}) rectangle ({1},{3.5}); 
\fill[pattern=north east lines, pattern color=gray] 
    ({0},{3.5}) rectangle ({4.5},{0}); 
\end{tikzpicture}
\end{matrix}\,,
\qquad 
Q = 
\begin{matrix}
\begin{tikzpicture}[scale=0.5]
\mygrid
\fill[pattern=north east lines, pattern color=gray] 
    ({0},{4.5}) rectangle ({3},{1.5}); 
\fill[pattern=north east lines, pattern color=gray] 
    ({3},{4.5}) rectangle ({4.5},{0}); 
\end{tikzpicture}
\end{matrix}\,,
\qquad 
S = 
\begin{matrix}
\begin{tikzpicture}[scale=0.5]
\mygrid
\fill[pattern=north east lines, pattern color=gray] 
    ({0},{4.5}) rectangle ({1},{1.5}); 
\fill[pattern=north east lines, pattern color=gray] 
    ({3},{3.5}) rectangle ({4.5},{0}); 
\fill[pattern=north east lines, pattern color=gray] 
    ({1},{3.5}) rectangle ({3.0},{1.5}); 
\end{tikzpicture}
\end{matrix}\,,$$
where  the matrices are partitioned into blocks of sizes $p,m,q$ ($m=n-p-q$). 
It is worth pointing out  that even when $u,v,w\in S_{p+q}$, we still need to work with the subgroups above for $m\gg 0$. 

It turns out that the subgroup $S$ is also a spherical subgroup, which  can be seen from the following embeddings:
\begin{align}\label{sspherical}
{\{\text{$S$-orbits of $G/B$}\}}
& = \{\text{$G$-orbits of $G/S\times G/B$}\}\nonumber\\
& \hookrightarrow
\{\text{$G$-orbits of $G/P\times G/Q\times G/B$}\}\\
& \hookrightarrow
\{\text{quiver representations of type $D$}\}. \nonumber
\end{align}
The last embedding is   well-known, and was pointed out in the study of multiple flag varieties of finite type by Magyar, Weyman and Zelevinsky \cite{MWZ}; see also \cite{Lit2,Smirnov1,HHK,BZ02,KR21,Raj}. 
We will establish  properties of $S$-orbits, including the finiteness, using the representation theory of quivers in Appendix \ref{sec:Sorbit} for completeness. 

Moreover,  we introduce a new combinatorial structure, called \emph{preclans}, a generalization of clans \cite{Matsuki}, to parametrize the $S$-orbits, and many (known and unknown) properties of $S$-orbits can be stated uniformly in terms of preclans. For nonnegative integers $p,m,q$, a $(p,m,q)$-preclan is a partial matching on $p+m+q$ nodes 
with some unmatched nodes colored by plus signs $\clan{+}$ and minus signs $\clan{-}$  such that the number of $\clan{+}$ minus the number of $\clan{-}$ equals $p-q$ and the number of uncolored nodes $\clan{,}$ equals $m$. 
For example, the following is a $(4,2,5)$-preclan:
 $$\clan{6-84,-..,+.}.$$

\subsection{Schubert Expansion and Positroid Varieties}

The non-equivariant Schubert expansion of $S$-orbit closures can be obtained by a general theorem of Brion \cite{Brion}.
From this, we can recover the results of Pechenik and Weigandt \cite{Pandt} by embedding preclans into the set of back stable clans.

To determine the equivariant Schubert expansion of $S$-orbit closures, similar to the method used in \cite[Proof of Theorem B]{WE}, it reduces to computing the equivariant localization at the identity $1\cdot B/B$. 
Denote the following subgroups of $G= GL_{n}$
$$\def\mygrid{\draw(0,3.5) to (4.5,3.5);
\draw(0,1.5) to (4.5,1.5);
\draw(1.0,0) to (1.0,4.5);
\draw(3.0,0) to (3.0,4.5);
\draw[thick](0.1,0) to (-.1,0) to (-.1,4.5) to (0.1,4.5);
\draw[thick](4.4,0) to (4.6,0) to (4.6,4.5) to (4.4,4.5);}
N = 
\begin{matrix}
\begin{tikzpicture}[scale=0.5]
\mygrid
\fill[pattern=north east lines, pattern color=gray] 
    ({0},{1.5}) rectangle ({3.0},{0}); 
\node at (0.5,4) {$1$}; 
\node at (2,2.5) {$1$}; 
\node at (3.75,0.75) {$1$}; 
\end{tikzpicture}
\end{matrix}\,,\qquad
\bar{B} = 
\begin{matrix}
\begin{tikzpicture}[scale=0.5]
\mygrid
\fill[pattern=north east lines, pattern color=gray] 
    (0,4.5) to (0,4.3) to (0.8,3.5) to (1,3.5) to (1,4.5) ;
\fill[pattern=north east lines, pattern color=gray] 
    (1,3.5) to (1,3.3) to (4.3,0) to (4.5,0) to (4.5,3.5) ;
\end{tikzpicture}
\end{matrix}\,
\quad \subseteq \quad
B = 
\begin{matrix}
\begin{tikzpicture}[scale=0.5]
\mygrid
\fill[pattern=north east lines, pattern color=gray] 
    (0,4.5) to (0,4.3) to (4.3,0) to (4.5,0) to (4.5,4.5) ;
\end{tikzpicture}
\end{matrix}\,.
$$
We  show that the subgroup $N\cong N\cdot B/B$  is a transversal slice of all $S$-orbits in $P\cdot B/B$, which is also an $S$-equivariant neighborhood of $1\cdot B/B$. 
When $m=0$, the transversal intersection  of $S$-orbits with $N\cong \mathbb{M}_p(\mathbb{C})$ can be identified with matrix Schubert varieties \cite{WE}.

Surprisingly, for general $m>0$, the intersection is related to positroid varieties over Grassmannians. 
Positroids were originally introduced by Postnikov \cite{Pos} to understand totally nonnegative Grassmannians. 
Knutson, Lam and Speyer \cite{KLS} studied 
the positroid varieties from an algebro-geometric point of view. 
In our context, the key technical tool is ``globalization''. 
Precisely, the intersection of an $S$-orbit closure with $N\cong N\cdot B/B$ is isomorphic to the intersection of a positroid variety with an opposite Schubert cell in the Grassmannian. 
In particular, the localization of an $S$-orbit closure coincides with the localization of a positroid variety. 
Under this identification, computing the localization globalizes to a known problem: finding polynomial representatives for positroid varieties.

A significant special case is when 
$m=0$, $p=q$, and $\gamma$ is the preclan with the $i$-th node matched with the $(p+w(i))$-th node for $1\leq i\leq p$ for a permutation $w\in S_p$. 
The corresponding positroid variety is the graph Schubert variety of $w$ over Grassmannian $Gr_p(\mathbb{C}^{2p})$. 
This agrees with the aforementioned results of \cite{WE}.

It is well-known that the equivariant fundamental classes of positroid varieties are represented by double affine Stanley symmetric polynomials \cite{KLS,LLS18}. 
Using the above connections, it follows that every nonzero triple coefficient in \eqref{eq:cuvwtri}
$$
c_{u,v}^w(t, y)
\text{ is a double affine Stanley symmetric polynomial}.
$$
We remark that by our choice, all double affine Stanley symmetric polynomials appearing here are essentially finite, but it would be helpful to keep their affine nature.

Obviously, the coefficients $c_{u,v}^w(t)$ in \eqref{eq:cuvw} can be obtained by setting $y=t$ in $c_{u,v}^w(t;y)$. 
However, the Graham positivity \cite{Graham01}  of double coefficients $c_{u,v}^w(t)$ does not directly follow from the triple positivity of $c_{u,v}^w(t;y)$ in \cite{Sam24,Gaox}. 
In order to get a Graham-positive formula, we  use the pipe dream model of Shimozono and Zhang \cite{SZ} for  double affine Stanley symmetric polynomials. 
We show combinatorially that  every nonzero double coefficient 
$$c_{u,v}^w(t)  
\text{ is a double Schubert polynomial in two disjoint sets of variables.
}
$$
In the remainder of the Introduction, we will package all the above steps together and state the main result in a completely combinatorial way.

\subsection{Main results}
Instead of delving into formal definitions, we prefer to present the main results  by means of  examples.
We need two specific types of preclans:
$$
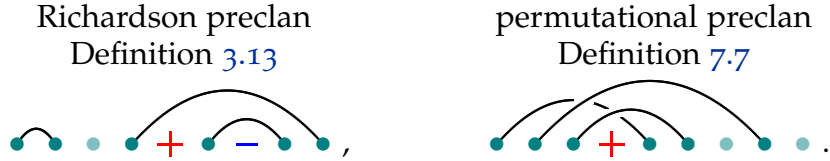
\begin{array}{c@{\qquad\qquad}c}
\text{Richardson preclan}&
\text{permutational preclan}\\
\text{Definition \ref{def:Richardosnpreclan}}& 
\text{Definition  \ref{def:perm}}
\\
\clan{1.,5+2-..},& 
\clan{463+..,.,}.\\
\end{array}$$
To state the main theorem, we still need the following notations: 
\begin{itemize}
    \item The preclan $w*\gamma$ for a permutation $w$ and a preclan $\gamma$. \hfill Definition \ref{def:monoidaction}
    \item The length $\ell(\gamma)$ of a preclan $\gamma$. \hfill Definition \ref{length}
    \item A Richardson preclan $\gamma_{v,u}$ for given inverse Grassmannians $v,u$. 
    \hfill Definition \ref{gammvu}
    \item A permutation $\eta_\gamma$ for a given permutational preclan $\gamma$. 
    \hfill Definition \ref{def:perm}
\end{itemize}

\subsection*{Main Theorem} 
Let $v,u$ be a $p$-inverse  and a $q$-inverse Grassmannian permutation, respectively. Without loss of generality, we can assume $p\geq q$. 
Then the structure constants $c_{v,u}^w(t)$ in \eqref{eq:cuvw} are given by
\begin{equation}\label{cuvt}
c_{v,u}^w(t) = \begin{cases}
\mathfrak{S}_{\eta_{\gamma}}(-t_q,\ldots,-t_1;-t_{p+1},-t_{p+2},\ldots,-t_n), \\
\qquad \text{if $\gamma:=w*\gamma_{v,u}$ is permutational and }
\ell(\gamma) = \ell(w)+\ell(\gamma_{v,u}),\\[1ex]
0, \quad \text{otherwise}.
\end{cases}
\end{equation}
It is easy to see that $c_{v,u}^w(t)$ is Graham-positive.
The triple version of \eqref{cuvt}, i.e., $c_{v,u}^w(t,y)$, takes a similar form (see Theorem \ref{th:triplecoeffi}), where the role of Schubert polynomials is played by affine Stanley symmetric polynomials. 

\begin{example}\label{eg:maineg}
Let us compute $c_{v,u}^w(t)$ for 
$v=12673485, u = 15672834,$ and $ 
w = 15782643.$
Note that $v$ is $5$-inverse Grassmannian and $u$ is $4$-inverse Grassmannian. 
The corresponding Richardson preclan $\gamma_{v,u}$ is 
$$
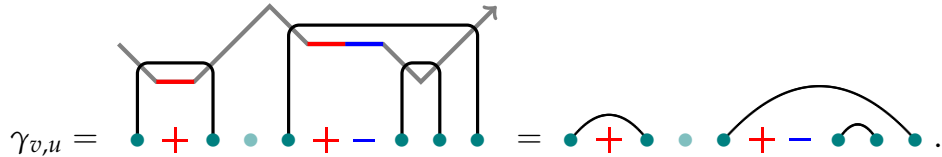
\begin{matrix}
\gamma_{v,u}=
\raisebox{-0.15cm}{\begin{tikzpicture}[scale=0.5]
\draw[ultra thick, gray,->]
    (0.5,0.5)--++(1,-1)--++(1,0)--++(1,1)--++(1,1)--++(1,-1)--++(1,0)--++(1,0)--++(1,-1)--++(1,1)--++(1,1);
\draw[ultra thick, red]
    (1.5,-.5)--(2.5,-.5);
\draw[ultra thick, red]
    (5.5,0.5)--(6.5,0.5);
\draw[ultra thick, blue]
    (6.5,0.5)--(7.5,0.5);
\draw[very thick, rounded corners]
    (1,-2)--(1,0)--(3,0)--(3,-2);
\draw[very thick, rounded corners]
    (5,-2)--(5,1)--(10,1)--(10,-2);
\draw[very thick, rounded corners]
    (8,-2)--(8,0)--(9,0)--(9,-2);
\node at (1,-2) {$\clan{.}$};
\node at (2,-2){$\clan{+}$};
\node at (3,-2) {$\clan{.}$};
\node at (4,-2){$\clan{,}$};
\node at (5,-2) {$\clan{.}$};
\node at (6,-2){$\clan{+}$};
\node at (7,-2) {$\clan{-}$};
\node at (8,-2){$\clan{.}$};
\node at (9,-2) {$\clan{.}$};
\node at (10,-2){$\clan{.}$};
\end{tikzpicture}}
=\clan{2+.,5+-1..}.
\end{matrix}
$$
Take a reduced word of $w=s_4s_5s_6s_7s_3s_2s_5s_4s_3s_6s_5s_4s_6s_7$. We can compute $\gamma=w*\gamma_{v,u}$; see Figure \ref{fig:exinIntro}.
\begin{figure}[ht]
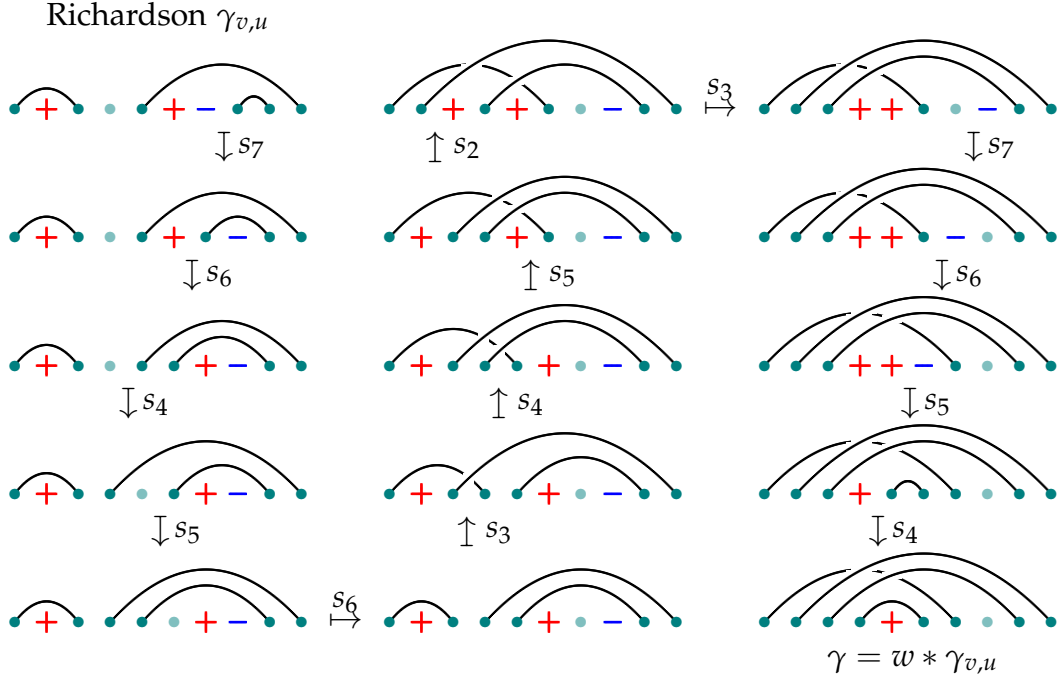

    \centering
$$
\setlength{\unitlength}{1pc}
\def\mapdown#1{\hspace{#1\unitlength}
    \makebox[0pc]{\rotatebox[origin=c]{-90}{$\mapsto$}}\,\,
    \makebox[1.2pc]{$s_{#1}$}}
\def\mapup#1{\hspace{#1\unitlength}
    \makebox[0pc]{\rotatebox[origin=c]{90}{$\mapsto$}}\,\,
    \makebox[1.2pc]{$s_{#1}$}}
\begin{array}{l@{\,\,}c@{\,\,}l@{\,\,}c@{\,\,}l}
\hfill\makebox[0pc]{Richardson $\gamma_{v,u}$}\hfill\\
\clan[\unitlength]{2+.,5+-1..}&&
    \clan[\unitlength]{58+5+.,-..}
    &\stackrel{\displaystyle s_3}{\mapsto}&
        \clan[\unitlength]{586++.,-..}\\
\mapdown{7}&&\mapup{2}&&\mapdown{7}\\
\clan[\unitlength]{2+.,5+2-..}&&
    \clan[\unitlength]{5+75+.,-..}&&
        \clan[\unitlength]{586++.-,..}\\
\mapdown{6}&&\mapup{5}&&\mapdown{6}\\
\clan[\unitlength]{2+.,53+-..}&&
    \clan[\unitlength]{4+75.+,-..}&&
        \clan[\unitlength]{686++-.,..}\\
\mapdown{4}&&\mapup{4}&&\mapdown{5}\\
\clan[\unitlength]{2+.6,3+-..}&&
    \clan[\unitlength]{3+7.4+,-..}&&
        \clan[\unitlength]{686+1..,..}\\
\mapdown{5}&&\mapup{3}&&\mapdown{4}\\
\clan[\unitlength]{2+.64,+-..} 
&\stackrel{\displaystyle s_6}{\mapsto}& 
    \clan[\unitlength]{2+.64+,-..}&&
        \clan[\unitlength]{6862+..,..}\\
        &&&&\hfill\makebox[0pc]{ $\gamma=w*\gamma_{v,u}$}\hfill
\end{array}$$
    \caption{Computation of $w*\gamma_{v,u}$}
    \label{fig:exinIntro}
\end{figure}
The preclan $\gamma=w\ast\gamma_{v,u}$ is permutational, $\ell(\gamma)=\ell(w)+\ell(\gamma_{v,u})$, and $\eta_\gamma=14523$.
The pipe dreams of $\eta_\gamma$ are shown as follows.
$$\begin{matrix}
 \BPD[1pc]{\M{}\M{6}\M{7}\M{8}\M{9}\M{10}\\
 \M{4}\B\B\B\B\J\\
 \M{3}\X\X\B\J\\
 \M{2}\X\X\J\\
 \M{1}\B\J}
\BPD[1pc]{\M{}\M{6}\M{7}\M{8}\M{9}\M{10}\\
 \M{4}\B\B\X\B\J\\
 \M{3}\X\B\B\J\\
 \M{2}\X\X\J\\
 \M{1}\B\J}
  \BPD[1pc]{\M{}\M{6}\M{7}\M{8}\M{9}\M{10}\\
 \M{4}\B\B\X\B\J\\
 \M{3}\X\B\X\J\\
 \M{2}\X\B\J\\
 \M{1}\B\J}
   \BPD[1pc]{\M{}\M{6}\M{7}\M{8}\M{9}\M{10}\\
 \M{4}\B\X\X\B\J\\
 \M{3}\B\B\B\J\\
 \M{2}\X\X\J\\
 \M{1}\B\J}
\BPD[1pc]{\M{}\M{6}\M{7}\M{8}\M{9}\M{10}\\
 \M{4}\B\X\X\B\J\\
 \M{3}\B\B\X\J\\
 \M{2}\X\B\J\\
 \M{1}\B\J}
    \BPD[1pc]{\M{}\M{6}\M{7}\M{8}\M{9}\M{10}\\
 \M{4}\B\X\X\B\J\\
 \M{3}\B\X\X\J\\
 \M{2}\B\B\J\\
 \M{1}\B\J}
\end{matrix}
$$
Here we relabeled the row and column indices. 
Then we have 
\begin{align*}
c_{v,u}^{w}(t)&=\mathfrak{S}_{14523}(-t_4,\ldots,-t_1;-t_6,\ldots,-t_{10})\\
&=
\begin{matrix}
    (t_6-t_3)(t_7-t_3)\\
    (t_6-t_2)(t_7-t_2)
\end{matrix}+
\begin{matrix}
    (t_6-t_3)(t_8-t_4)\\
    (t_6-t_2)(t_7-t_2)
\end{matrix}+
\begin{matrix}
    (t_6-t_3)(t_8-t_4)\\
    (t_6-t_2)(t_8-t_3)
\end{matrix}\\
&\quad +\begin{matrix}
    (t_7-t_4)(t_8-t_4)\\
    (t_6-t_2)(t_7-t_2)
\end{matrix}+
\begin{matrix}
(t_7-t_4)(t_8-t_4)\\
(t_6-t_2)(t_8-t_3)
\end{matrix}+
\begin{matrix}
    (t_7-t_4)(t_8-t_4)\\
    (t_7-t_3)(t_8-t_3)
\end{matrix}. 
\end{align*}
\end{example}

\subsection{Applications} 
We can obtain a positive combinatorial model for the coefficients $c_{u,v}^w(0)$ when 
\begin{equation}\label{eq:321andinvGr}
\text{$u$ is $321$-avoiding and $v$ is an inverse Grassmannian permutation}.
\end{equation}
Note that $321$-avoiding permutations include both Grassmannian permutations and inverse Grassmannian permutations. 
The specific case when $u$ is \emph{subjacent} (that is, $u=s_p\tilde{u}$ for some $p$-inverse Grassmannian permutation $\tilde{u}$) was considered by Pechenik and Weigandt \cite{Pandt}.
 
Recall that the Stanley symmetric function $F_{w}(x)$ was introduced by Stanley~\cite{Sta} to enumerate reduced words of $w\in S_n$. The coefficients $EG^{w}_{\mu}$ in the Schur expansion
\[
F_{w}(x)=\sum_{\mu} EG^{w}_{\mu}\cdot s_{\mu}(x)
\]
are known as the Edelman--Greene coefficients. In~\cite{EG}, Edelman and Greene showed that $EG^{w}_{\mu}$ is the number of reduced word tableaux of $w$ of shape $\mu$, thereby proving Stanley's conjecture that the Edelman--Greene coefficients are nonnegative.

When $u$ is 321-avoiding, we can write $u=w_\mu\tilde{u}$, where $w_\mu$ is a Grassmannian permutation for a partition $\mu$ and $\tilde{u}$ is an inverse Grassmannian permutation, and $\ell(u)+\ell(w_\mu)=\ell(\tilde{u})$.  To state our result, we need one more notation --- 
a permutation $g_\gamma$ for a preclan with $v_\gamma=e$, see Definition \ref{ggamma}.

\subsection*{Theorem}
Let $v$ be a $p$-inverse Grassmannian permutation and $u$ be a $321$-avoiding permutation. 
We have 
$$c_{v,u}^w(0)=\begin{cases}
EG^{g_{\gamma}}_{\mu}, & 
\text{if $\gamma=w*\gamma_{v,\tilde{u}}$ 
satisfies $\ell(\gamma)=\ell(w)+\ell(\gamma_{v,\tilde{u}})$ and 
$v_\gamma=e$}, \\[1ex]
0,& \text{otherwise}. 
\end{cases}$$

We believe that there exists a Graham-positive formula for $c_{u,v}^w(t)$, where $u$ is 321-avoiding and $v$ is an inverse Grassmannian.
But due to the length of this paper, we decide to leave it for future work. 

When $\mu$ is the partition of a single box, the Edelman--Greene coefficient $EG^{w}_{\mu}$ is nonzero only when $w$ is a simple reflection, and $EG^{w}_{\mu}=1$. In this case,  we recover the result of Pechenik and Weigandt \cite{Pandt}.

The following example shows that these coefficients are not always $0$-$1$. 

\begin{example}
Consider 
$v=1256347, u=1357246, w=24681357.$
 Note that $u$ is a 4-Grassmannian permutation, and in particular it is $321$-avoiding, and $v$ is $4$-inverse Grassmannian. 
We can write $u=w_{(2,1)}\tilde{u}$, where 
$
w_{(2,1)}=1246357,
\tilde{u}= 1567234.$
Note that $\tilde{u}$ is $4$-inverse Grassmannian and $w_{(2,1)}$ is $4$-Grassmannian. 
Take a reduced word of $w=s_1s_3s_2s_5s_4s_3s_7s_6s_5s_4$. We have
$$\gamma_{v,\tilde{u}}=
\raisebox{-0.15cm}
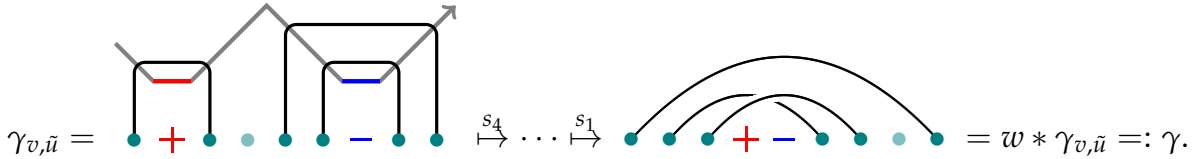
{\begin{tikzpicture}[scale=0.5]
\draw[ultra thick, gray,->]
    (0.5,0.5)--++(1,-1)--++(1,0)--++(1,1)--++(1,1)--++(1,-1)--++(1,-1)--++(1,0)--++(1,1)--++(1,1);
\draw[ultra thick, red]
    (1.5,-.5)--(2.5,-.5);
\draw[ultra thick, blue]
    (6.5,-.5)--(7.5,-.5);
\draw[very thick, rounded corners]
    (1,-2)--(1,0)--(3,0)--(3,-2);
\draw[very thick, rounded corners]
    (5,-2)--(5,1)--(9,1)--(9,-2);
\draw[very thick, rounded corners]
    (6,-2)--(6,0)--(8,0)--(8,-2);
\node at (1,-2) {$\clan{.}$};
\node at (2,-2){$\clan{+}$};
\node at (3,-2) {$\clan{.}$};
\node at (4,-2){$\clan{,}$};
\node at (5,-2) {$\clan{.}$};
\node at (6,-2){$\clan{.}$};
\node at (7,-2) {$\clan{-}$};
\node at (8,-2){$\clan{.}$};
\node at (9,-2) {$\clan{.}$};
\end{tikzpicture}}
\stackrel{s_4}\mapsto \cdots 
\stackrel{s_1}\mapsto 
\clan{844+-..,.}
=w*\gamma_{v,\tilde{u}}=:\gamma.$$
Since the first $p$ nodes of $\gamma$ are either $\clan{+}$ or left-ends, we have $v_\gamma=e$, and $\ell(\gamma)=\ell(w)+\ell(\gamma_{v,\tilde{u}})$.
The resulting permutation $g_\gamma=123547698=s_{8}s_6s_4$, and the Edelman--Greene coefficient is 
$EG^{g_\gamma}_{(2,1)}=2$, i.e., 
$c_{v,u}^w(0)=2$. 
In fact, there are two reduced word tableaux of shape $(2,1)$ for $g_\gamma$:  
$\begin{array}{|c|c|}
\hline 
4 & 6\\\hline 8\\\cline{1-1}
\end{array},\ 
\begin{array}{|c|c|}
\hline 
4 & 8\\\hline 6\\\cline{1-1}
\end{array}.$

\end{example}

\subsection*{Structure of the paper}
This paper is organized as follows. In Section \ref{Preliminary}, we recall the necessary notations on flag varieties, Schubert polynomials, spherical subgroups, positroid varieties, etc. In Section \ref{S-orbits and Preclans}, we introduce preclans and establish some basic properties of preclans. In Section \ref{Polynomial Representatives}, we determine the polynomial representative of the fundamental class of an $S$-orbit closure. In Section \ref{Localization and Globalization}, we intersect an $S$-orbit closure with the transversal slice $N\cdot B/B$, and then globalize the intersection into the Grassmannian. Section \ref{Positroid Varieties} is devoted to identifying the intersection in Section \ref{Localization and Globalization} with a positroid variety. Finally, in Section \ref{Schubert Expansion}, we show our main results and give an application. In  Appendix \ref{sec:Sorbit}, we utilize the theory of quiver representations to establish some results on $S$-orbits.

\subsection*{Acknowledgment}
We are grateful to Zachary Hamaker, Oliver Pechenik, and Anna Weigandt for bringing this problem to our attention. We also thank Anders Buch, Daoji Huang, Allen Knutson, Thomas Lam,  Leonardo C. Mihalcea, Xiaolong Pan, Mark Shimozono, and Alexander Yong for helpful discussions. This work is supported by the National Key Research and Development Program of China (No. 2025YFA1017702) and the National Natural Science Foundation of China (No. 12471314).

\section{Preliminaries}\label{Preliminary}

In the entire paper, we will fix nonnegative integers 
$$p,m,q\geq 0,\qquad n= p+m+q\geq 0,$$
and denote $G=GL_n$ and its subgroups 
$$
T = \left[\begin{matrix}
* \\ & \ddots \\ && *
\end{matrix}\right],\qquad 
B = \left[\begin{matrix}
*&\cdots&* \\ & \ddots&\vdots \\ && *
\end{matrix}\right],\qquad 
B^- = \left[\begin{matrix}
*& \\ \vdots& \ddots& \\ {}*&\cdots& *
\end{matrix}\right]
.$$
The Weyl group of $G$ is the symmetric group $S_n$, and we denote 
    the simple reflections by $s_1,\ldots,s_{n-1}$,
    the length function by $\ell:S_n\to \mathbb{Z}_{\geq 0}$, and 
    the Bruhat order by $\leq$. 

We will use the following subgroups
$$
\def\O{\color{lightgray}0}
S=P\cap Q=\left[\begin{matrix}
GL_p&\O&\O \\{}*&GL_m&*\\\O&\O&GL_q
\end{matrix}\right],\qquad 
P = 
\left[\begin{matrix}
GL_p&\,\,\O\quad\O\,\, \\
\begin{matrix}
*\\{}*
\end{matrix}&GL_{m+q}
\end{matrix}\right],\qquad 
Q = \left[\begin{matrix}
GL_{p+m}&\begin{matrix}
*\\{}*\end{matrix}
\\\,\,\O\quad \O\,\,&GL_q
\end{matrix}\right],
$$
$$\def\O{\color{lightgray}0}
N = 
\left[\begin{matrix}
1_p&\O&\O \\ \O& 1_m&\O\\ {}*&*& 1_q
\end{matrix}\right],\qquad 
\bar{B}=B\cap P=\left[\begin{matrix}
B_p &\,\,\O\quad\O\,\,\\
\begin{matrix}
\O\\{}\O
\end{matrix}& B_{m+q}
\end{matrix}\right],$$
where $B_k$ is the subgroup of upper triangular matrices of $GL_k$. 
When $m=0$, we prefer to denote the subgroup $S$ by $K\cong GL_p\times GL_q$.

\subsection{Flag varieties and Schubert varieties}

The \emph{flag variety} is 
$$G/B\cong \{0=V_0\subset V_1\subset \cdots \subset V_n=\mathbb{C}^n:\dim V_i=i\}. $$
We can decompose 
$$G/B=\bigsqcup_{w\in S_n}BwB/B=\bigsqcup_{w\in S_n}B^-wB/B,$$
where $w$ is viewed as the corresponding permutation matrix. 
The Schubert variety of $w\in S_n$ is the closure $\overline{BwB/B}$ and the opposite Schubert variety of $w\in S_n$ is the closure $\overline{B^-wB/B}$. 

Let $w_0=n\cdots 21\in S_n$ be the longest permutation. The double Schubert polynomials $$\mathfrak{S}_w(x;t)\in 
\mathbb{Q}[x_i,t_j]_{i,j}=\mathbb{Q}[x_1,x_2,\ldots,t_1,t_2,\ldots].$$
for $w\in S_n$ can be characterized by 
$$\mathfrak{S}_{w_0}(x;t)=\prod_{i+j\leq n}
(x_i-t_j),\qquad
\mathfrak{S}_w(x;t)
=\partial_k\mathfrak{S}_{ws_k}(x;t)
\text{ if }
\ell(ws_k)=\ell(w)+1,$$
where $\partial_k$ is the divided difference operator on polynomials  
$$\partial_kf =\frac{f-f|_{x_k\leftrightarrow x_{k+1}}}{x_k-x_{k+1}}.$$

\begin{prop}[\cite{AF23}]
For $w\in S_n$, 
the double Schubert polynomial $\mathfrak{S}_w(x;t)$ represents the equivariant fundamental class of the opposite Schubert variety $[\overline{B^-wB/B}]_T\in H_T^*(G/B)$. 
\end{prop}

For a permutation $w\in S_n$, a {\it pipe dream} (also called an RC-graph \cite{BB}) of $w$ is a tiling of an $n\times n$ grid by $\BPD{\X}$'s and $\BPD{\B}$'s such that the 
$j$-th  pipe enters from the top, flows downward, and leftward and  exits from the left side, ending up at row $w^{-1}(j)$ for $1\le j\le n$. Then the double Schubert polynomial $\mathfrak{S}_{w}(x;t)$ is  the weighted sum over $\mathcal{PD}(w):=\{\text{pipe dreams of }w\text{ with no pipes cross more than once}\}$
$$\mathfrak{S}_{w}(x;t)=\sum_{D\in\mathcal{PD}(w)}\prod_{(i,j)\in D}(x_i-t_j),$$
where $i$ is the row index and $j$ the column index of  a $\BPD{\X}$ tile. 
For example, let $w=1342\in S_4$, there are three pipe dreams of $w$:
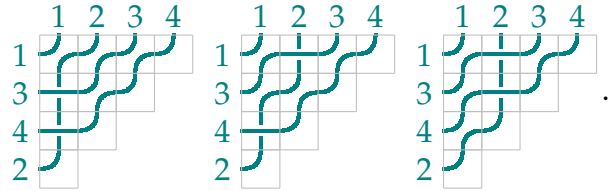
$$\BPD{\M{}\M{1}\M{2}\M{3}\M{4}\\\M{1}\B\B\B\J\\\M{3}\X\B\J\\\M{4}\X\J\\\M{2}\J} \BPD{\M{}\M{1}\M{2}\M{3}\M{4}\\\M{1}\B\X\B\J\\\M{3}\B\B\J\\\M{4}\X\J\\\M{2}\J} 
\BPD{\M{}\M{1}\M{2}\M{3}\M{4}\\\M{1}\B\X\B\J\\\M{3}\B\X\J\\\M{4}\B\J\\\M{2}\J}.$$
Thus $\mathfrak{S}_{1342}(x;t)=(x_2-t_1)(x_3-t_1)+(x_1-t_2)(x_3-t_1)+(x_1-t_2)(x_2-t_2).$

\subsection{Spherical subgroups}\label{spherical}
A subgroup $S\subseteq G$ is called \emph{spherical} if the number of $S$-orbits on flag variety $G/B$ is finite. 
Recall the Demazure monoid structure on the Weyl group  of $G$ is characterized by 
$$s_k*s_k=s_k,\qquad u*v=uv\text{ if $\ell(uv)=\ell(u)+\ell(v)$}.$$
By \cite[Section 3]{RS92}, the monoid acts on the set of $S$-orbits as follows. 
Let $P_k=B\cup Bs_kB$ be the minimal parabolic subgroup for $1\leq k\leq n-1$. 
We have a natural projection 
$$\pi_k: G/B\longrightarrow G/P_k. $$
Let $\mathbb{O}$ be an $S$-orbit of $G/B$. The monoid action is given by
$$s_k*\mathbb{O} = \text{the dense $S$-orbit in $\pi_k^{-1}(\pi_k(\mathbb{O}))$}.$$
Pick any $y\in \pi_k(\mathbb{O})$, the fiber $F=\pi_k^{-1}(y)\cong \mathbb{P}^1$. 
By Mars and Springer \cite{Mars}, there are seven types 
\begin{itemize}
    \item[(I)]  
    $\pi_k^{-1}(\pi_k(\mathbb{O})) = \mathbb{O}$, i.e., each fiber $F$ is contained in $\mathbb{O}$, in which case $s_k*\mathbb{O}=\mathbb{O}$; 
    \item[(IIa)] 
    $\pi_k^{-1}(\pi_k(\mathbb{O})) = \mathbb{O}\cup \mathbb{O}'$ with each fiber $F$ containing exactly one point of $\mathbb{O}$, in which case 
    $s_k*\mathbb{O}=\mathbb{O}'$; 
    
    \item[(IIb)]  
    $\pi_k^{-1}(\pi_k(\mathbb{O})) = \mathbb{O}\cup \mathbb{O}'$ with each fiber $F$ containing exactly one point of $\mathbb{O}'$, in which case 
    $s_k*\mathbb{O}=\mathbb{O}$; 

    \item[(IIIa)]   
    $\pi_k^{-1}(\pi_k(\mathbb{O})) =\mathbb{O}\cup \mathbb{O}'\cup \mathbb{O}''$ with each fiber $F$ containing exactly one point of $\mathbb{O}$ and one point of $\mathbb{O}''$, in which case $s_k*\mathbb{O}=\mathbb{O}'$; 
    
    \item[(IIIb)]   
    $\pi_k^{-1}(\pi_k(\mathbb{O})) =\mathbb{O}\cup \mathbb{O}'\cup \mathbb{O}''$ with each fiber $F$ containing exactly one point of $\mathbb{O}'$ and one point of $\mathbb{O}''$, in which case 
    $s_k*\mathbb{O}=\mathbb{O}$; 

    \item[(IVa)]   
    $\pi_k^{-1}(\pi_k(\mathbb{O})) =\mathbb{O}\cup \mathbb{O}'$ with each fiber $F$ containing exactly two points of $\mathbb{O}$, in which case $s_k*\mathbb{O}=\mathbb{O}'$; 

    \item[(IVb)]   
    $\pi_k^{-1}(\pi_k(\mathbb{O})) =\mathbb{O}\cup \mathbb{O}'$ with each fiber $F$ containing exactly two points of $\mathbb{O}'$, in which case $s_k*\mathbb{O}=\mathbb{O}$. 
\end{itemize}
Let $d$ be the number of $S$-orbits $\mathbb{O}'$ such that $s_k*\mathbb{O}'=s_k*\mathbb{O}$.
We have 
$$\begin{array}{c|c|c|c}\hline
\vphantom{\dfrac12}
& \qquad d=1\qquad &\qquad d=2 \qquad
& \qquad d=3\qquad \\\hline
\vphantom{\dfrac12}
s_k*\mathbb{O}\neq \mathbb{O}
& \text{\rm never} & \rm (IIa) \text{ or } (IVa) & \rm (IIIa)\\\hline
\vphantom{\dfrac12}
s_k*\mathbb{O}= \mathbb{O} & \rm (I) & \rm (IIb) \text{ or } (IVb) & \rm (IIIb)\\\hline
\end{array}$$

\subsection{Positroid varieties}
The affine symmetric group $\widetilde{S}_n$ is the group of bijections $f:\mathbb{Z}\to \mathbb{Z}$ such that $f(i+n)=f(i)+n$. 
Let 
$$\widetilde{S}_{n}^{k}
=\left\{f\in \widetilde{S}_n:
\sum_{i=1}^n (f(i)-i) = kn
\right\}\subseteq \widetilde{S}_n. $$

For an affine permutation $f$, define its inversion set by
$$
\operatorname{inv}(f):=\{(i,j)\in \mathbb Z^2 : i<j \text{ and } f(i)>f(j)\},
$$
where $(i,j)$ is identified with $(i+rn,j+rn)$ for every $r\in \mathbb Z$.
The length of $f$, denoted by $\ell(f)$, is the number of equivalence classes
in $\operatorname{inv}(f)$ under this identification.

For any $0\leq k\leq n$, the set of $(k,n)$ bounded affine permutations is 
$$\operatorname{Bound}(k,n)
=\{f\in\widetilde{S}_{n}^{k}:i\leq f(i)\leq n+i\}\subseteq \widetilde{S}_n^k. 
$$
Let $Gr_k(\mathbb{C}^n)$ denote the Grassmannian of $k$-dimensional subspaces in $\mathbb{C}^n$.
For each $V\in Gr_k(\mathbb{C}^n)$, we choose a full rank $k\times n$  matrix 
$$\left[
\begin{matrix}
\mid & \mid & \cdots & \mid\\
\vv_1 & \vv_2 & \cdots & \vv_n\\
\mid & \mid & \cdots & \mid\\
\end{matrix}\right]_{k\times n}$$
such that $V$ is its row  span. 
We can define a bounded affine permutation $f_V\in \operatorname{Bound}(k,n)$ by 
$$f_V(i)= \min\{j>i:\vv_i\in \\
\operatorname{span}(\vv_{i+1},\ldots,\vv_j)\},$$
where we extend $\vv_i$ to $i\in \mathbb{Z}$ by $\vv_{i}=\vv_{i\bmod n}$. 
When $\vv_i=0$, the minimum is understood as $i$ itself. 
The \emph{positroid variety} is the closure
$$\Pi_f = \overline{\{V\in Gr_k(\mathbb{C}^n): 
f_V=f\}}.$$

The   Stanley symmetric functions were introduced by Stanley \cite{Sta} as the forward limit of Schubert polynomials.
The (double) affine Stanley symmetric polynomials (functions) were introduced by Lam and Shimozono \cite{Lam,LS13}. 
Instead of introducing the algebraic definition, we prefer to present the pipe dream model of $\widetilde{F}_f(x;t)$ due to Shimozono--Zhang \cite{SZ} when restricted to $f\in \operatorname{Bound}(k,n)$. 
Explicitly, a pipe dream of $f$ is an $n$-periodic tiling of $\{1,\ldots,k\}\times \mathbb{Z}$ with tiles  $\BPD{\X}$ and $\BPD{\B}$, such that the pipe $i$ starts from the $i$-th position at the bottom, ends at the $f(i)$-th position at the top with no two pipes cross more than once, and  the tile at $(i,j)$ in a $n$-period is the same as $(i,j+n)$.
The weight of a $\BPD{\X}$  tile  at position $(i,j)$ is $x_i-t_j$. The weight of a pipe dream is the product of the weights of the  $\BPD{\X}$ tiles.
Then $\widetilde{F}_f(x;t)$ is the sum of  weights of all pipe dreams of $f$.  For example, let $f=[2547]\in \mathrm{Bound}(2,4)$, there are four pipe dreams of $f$: %
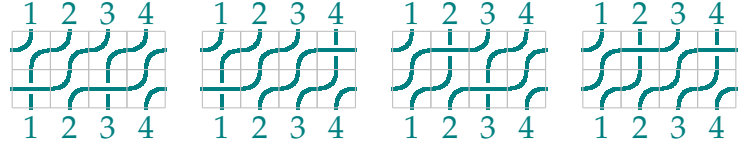
$$
\begin{array}{cccc}
\BPD{
\M{1}\M{2}\M{3}\M{4}\\
\B\B\B\B\\
\X\B\X\B\\
\M{1}\M{2}\M{3}\M{4}
}
&
\BPD{
\M{1}\M{2}\M{3}\M{4}\\
\B\B\B\X\\
\X\B\B\B\\
\M{1}\M{2}\M{3}\M{4}
}
&
\BPD{
\M{1}\M{2}\M{3}\M{4}\\
\B\X\B\B\\
\B\B\X\B\\
\M{1}\M{2}\M{3}\M{4}
}
&
\BPD{
\M{1}\M{2}\M{3}\M{4}\\
\B\X\B\X\\
\B\B\B\B\\
\M{1}\M{2}\M{3}\M{4}
}
\end{array}.
$$
Thus 
\begin{align*}
    \widetilde{F}_{[2547]}(x;t)&=(x_2-t_1)(x_2-t_3)+(x_1-t_2)(x_2-t_3)\\
    &\quad+(x_1-t_4)(x_2-t_1)+(x_1-t_2)(x_1-t_4).
\end{align*}

\begin{prop}[\cite{KLS,LLS18}]
For $f\in \operatorname{Bound}(k,n)$, the double affine Stanley symmetric function $\widetilde{F}_f(x;t)$ represents the equivariant fundamental class of the positroid variety $[\Pi_f]_T\in H_T^*(Gr_k(\mathbb{C}^n))$. 
\end{prop}

\section{\texorpdfstring{$S$}{S}-orbits and Preclans}\label{S-orbits and Preclans}

 In this section, we introduce preclans and use them to parametrize the $S$-orbits on $G/B$.
We first recall the notion of clans, and then define preclans and their representatives in $G/B$, and prove that the $S$-orbits are indexed by $(p,m,q)$-preclans.
We then describe the weak order on preclans and show that it agrees with the weak order on $S$-orbits.
Finally, we define Richardson preclans and prove that they parametrize Richardson varieties arising from pairs of inverse Grassmannian permutations.

\subsection{Clans}
A \emph{clan} is a partial matching with all unmatched nodes colored by $\clan{+}$ or $\clan{-}$. 
We say a clan $\gamma$ is a $(p,q)$-clan if 
\begin{itemize}
    \item the number of $\clan{+}$'s and matchings is $p$; 
    \item the number of $\clan{-}$'s and matchings is $q$.
\end{itemize}
For example, the following diagram represents a $(4,5)$-clan
\begin{equation}\label{eq:claneg}
\gamma = \clan{5-63-..+.}.
\end{equation}
For a clan $\gamma$, its  length is defined as
$$\ell(\gamma) = \sum_{(i,j)\text{-matching}} (j-i) 
- \texttt{\#}\{\text{crossings in $\gamma$} \}.$$
For example, for the clan in \eqref{eq:claneg}, we have 
$\ell(\gamma)=(6-1)+(9-3)+(7-4)-2=12.$

A clan is said to be 
\begin{itemize}
    \item \emph{matchless}, if it has no matching. For example, there are $6$ matchless $(2,2)$-clans
    $$
    \begin{matrix}
    \clan{++--},\,
    \clan{+-+-},\,
    \clan{+--+},\,
    \clan{-++-},\,
    \clan{-+-+},\,
    \clan{--++}.
    \end{matrix}
    $$
    \item a \emph{rainbow}, if it is of the form 
$$
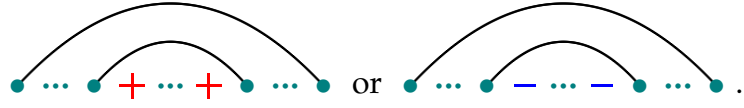
\clan{8\dots4+\dots+.\dots.}
\text{ or }
\clan{8\dots4-\dots-.\dots.}.$$
\end{itemize}
The main reason for studying clans is the following. 
\begin{prop}[Matsuki \cite{Matsuki}]
The $K$-orbits of $G/B$ are finite and parametrized by $(p,q)$-clans. 
\end{prop}

\subsection{Preclans}\label{sec:preclans}
A \emph{preclan} is a clan with some unmatched nodes uncolored. 
We say a preclan is a $(p,m,q)$-preclan if 
\begin{itemize}
    \item the number of $\clan{+}$'s and matchings is $p$; 
    \item the number of $\clan{-}$'s and matchings is $q$; 
    \item the number of uncolored nodes $\clan{,}$'s is $m$. 
\end{itemize}
We will always denote $n=p+m+q$. 
For example, the following diagram is a $(4,2,5)$-preclan 
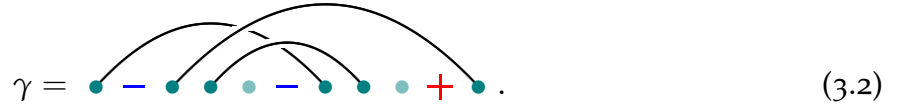
\begin{equation}\label{eq:preclaneg}
\gamma=\clan{6-84,-..,+.}.
\end{equation}
For a $(p,m,q)$-preclan $\gamma$, we associate a permutation $\sigma_{\gamma}$ to $\gamma$ as follows.  First assign $1,\ldots,p$ to $\clan{+}$'s and left-ends of $\gamma$ from left to right, then assign $p+1,\ldots,p+m$ to uncolored $\clan{,}$'s of $\gamma$ from left to right, and finally assign $p+m+1,\ldots,p+m+q$ to $\clan{-}$'s and right-ends of $\gamma$ from left to right. Then  $\sigma_{\gamma}$ is obtained by reading the assigned labels of $\gamma$ from left to right.

For example, let $\gamma$ be the preclan in \eqref{eq:preclaneg}. We have 
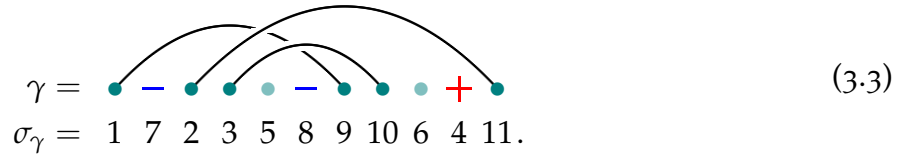
\begin{align}\label{u_v_gamma}
\def\m#1{\makebox[1.2pc]{$#1$}}
\begin{aligned}
\gamma&=\,\clan{6-84,-..,+.}\\
\sigma_\gamma &= \,\m{1}\m{7}\m{2}\m{3}\m{5}\m{8}\m{9}\m{10}\m{6}\m{4}\m{11}.
\end{aligned}
\end{align}

\begin{definition}\label{def:dotgamma}
For a $(p,m,q)$-preclan $\gamma$, define  $\dot{\gamma}=[\dot{\gamma}_1\,\,\dot{\gamma}_2\,\, \ldots\,\, \dot{\gamma}_n]\in GL_n$, where 
$\dot{\gamma}_1,\ldots,\dot{\gamma}_n$ are  column vectors defined  in the following way.  For $1\le i\le n$, if the $i$-th node of $\gamma$ is the left-end of an $(i,j)$-matching, then 
   $\dot{\gamma}_i = \mathbf{e}_{\sigma_{\gamma}(i)}+\mathbf{e}_{\sigma_{\gamma}(j)}$. Otherwise,
let 
   $\dot{\gamma}_i = \mathbf{e}_{\sigma_{\gamma}(i)}$.
\end{definition}
For example, for the preclan in \eqref{u_v_gamma}, we have 
$$
\def\O{\color{lightgray}0}
\def\I{1}
\dot{\gamma}
=\left[\begin{array}{ccccccccccc}
\I&\O&\O&\O&\O&\O&\O&\O&\O&\O&\O\\
\O&\O&\I&\O&\O&\O&\O&\O&\O&\O&\O\\
\O&\O&\O&\I&\O&\O&\O&\O&\O&\O&\O\\
\O&\O&\O&\O&\O&\O&\O&\O&\O&\I&\O\\
\O&\O&\O&\O&\I&\O&\O&\O&\O&\O&\O\\
\O&\O&\O&\O&\O&\O&\O&\O&\I&\O&\O\\
\O&\I&\O&\O&\O&\O&\O&\O&\O&\O&\O\\
\O&\O&\O&\O&\O&\I&\O&\O&\O&\O&\O\\
\I&\O&\O&\O&\O&\O&\I&\O&\O&\O&\O\\
\O&\O&\O&\I&\O&\O&\O&\I&\O&\O&\O\\
\O&\O&\I&\O&\O&\O&\O&\O&\O&\O&\I\\
\end{array}\right]\in GL_{11}.$$

\begin{theorem}
The $S$-orbits of $G/B$ are finite and parametrized by $(p,m,q)$-preclans: 
$$G/B=\bigsqcup_{\gamma} S\dot{\gamma}B/B.$$
\end{theorem}
\begin{proof}
Let us sketch the key steps of the proof. 
\begin{itemize}
    \item We first show that the $S$-orbits are bijective to certain isomorphism classes of quiver representations of type $D$. 
    This step is essentially due to Magyar, Weyman and Zelevinsky \cite{MWZ}. 
    \item Using the classification of indecomposable representations of type $D$, we show that the isomorphism classes are naturally indexed by $(p,m,q)$-preclans. 
    \item Finally, in Proposition \ref{prop:an-ele-S}, we show that $S\dot{\gamma}B/B$ is the corresponding orbit of $\gamma$. 
\end{itemize}
The complete proof is given in Section \ref{sec:finiteofS}.  
\end{proof}

 \subsection{Weak order}
For a preclan $\gamma$, 
we define $s_k*\gamma$ to be the preclan obtained from $\gamma$ by the following 9 local moves on the $k$-th node and the $(k+1)$-st node. Otherwise, let $s_k*\gamma=\gamma$.

\def\mpto{\rotatebox[origin=c]{-90}{$\mapsto$}}
\begin{gather}\label{eq:clanweak1}
\begin{matrix}
\begin{matrix}
\clan{\dots+-\dots}\\
\clan{\dots-+\dots}
\end{matrix}&\,\,&
\clan{\dots\pm2\dots.}&\,\,&
\clan{2\dots.\pm\dots}\\[-1ex]
\mpto&&\qquad\quad\mpto\hfill&&\hfill\mpto\quad\qquad\\
\clan{\dots1.\dots}&&
\clan{\dots3\pm\dots.}&&
\clan{3\dots\pm.\dots}
\end{matrix}\\[1ex]\label{eq:clanweak2}
\begin{matrix}
\clan{2\dots.2\dots.}&
\clan{34\dots.\dots.}&
\clan{4\dots3\dots..}\\[-1ex]
\mpto&\quad\mpto\hfill&\hfill\mpto\quad\\
\clan{3\dots3.\dots.}&
\clan{52\dots.\dots.}&
\clan{5\dots2\dots..}
\end{matrix}\\[1ex]\label{eq:clanweak3}
\quad
\begin{matrix}
\clan{,\pm}&\qquad&
\clan{,2\dots.}& \quad&
\clan{3\dots,.}
\\
\mpto&& \mpto \qquad\quad && \qquad\mpto\\
\clan{\pm,}&& \clan{3,\dots.} && 
\clan{2\dots.,}
\end{matrix}
\end{gather}

The weak order on all $(p,m,q)$-preclans is generated by the relations $\gamma\leq s_k*\gamma$.

\begin{example}
When $p=q=m=1$, the Hasse diagram of the weak order is displayed in the following.

$$
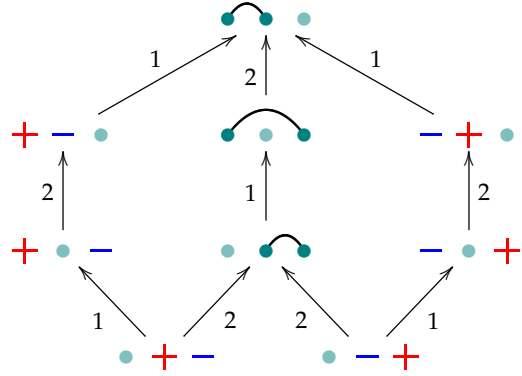
\xymatrix@C=-1pc{
&& {\clan{1.,}}\\
{\clan{+-,}}\ar[urr]^1 &&
{\clan{2,.}}\ar[u]^2 && 
{\clan{-+,}}\ar[ull]_1\\
{\clan{+,-}}\ar[u]^2&&
{\clan{,1.}}\ar[u]^1&& 
{\clan{-,+}}\ar[u]_2\\
&{\clan{,+-}}\ar[ur]_2\ar[ul]^1 &&
{\clan{,-+}}\ar[ul]^2\ar[ur]_1
}$$
\end{example}

\begin{theorem}
\label{th:weakordApp}
For a $(p,m,q)$-preclan $\gamma$, let  $\mathbb{O}_\gamma=S\dot{\gamma}B/B$ denote the corresponding $S$-orbit. 
We have 
\begin{equation}
s_k*\mathbb{O}_{\gamma}=\mathbb{O}_{s_k*\gamma}. 
\end{equation}
Moreover, types (IVa) and (IVb) in Subsection \ref{spherical} will never appear. 
\end{theorem}
\begin{proof}
Again, we only give a sketch of the proof here; see Section \ref{sec:weakordApp}.  
\begin{itemize}
    \item We can realize $G/B\to G/P_k$ as a contraction of quivers, obtained by taking composition of arrows. 
    \item By ignoring the indecomposable summands not contributing to the fiber, we are led to finitely many cases. The assertion follows from direct computation. \qedhere
\end{itemize}
\end{proof}

\begin{definition}\label{def:monoidaction}
For a permutation $w\in S_n$ and a $(p,m,q)$-preclan $\gamma$, we define 
$$w*\gamma = s_{i_1}*\cdots *s_{i_\ell}*\gamma$$
for any reduced decomposition $w=s_{i_1}\cdots s_{i_\ell}$. 
By Theorem \ref{th:weakordApp} and general theory of spherical subgroup orbits, $w*\gamma$ does not depend on the choices of  reduced decompositions of $w$. 
\end{definition}

\begin{definition}\label{length}    
For a preclan $\gamma$, let $\{a_1<\cdots<a_m\}$ be the set of indices of uncolored nodes of $\gamma$. Define the length of $\gamma$ to be 
\begin{equation}\label{eq:lenghts}
\ell(\gamma):= \ell(\bar{\gamma})+\sum_{i=1}^m(a_i-i),
\end{equation}
where $\bar{\gamma}$ is the clan obtained from $\gamma$ by ignoring all the uncolored $\clan{,}$'s. 
\end{definition}
For example, for the preclan $\gamma$ in \eqref{eq:preclaneg}, 
the clan $\bar{\gamma}$ is \eqref{eq:claneg} with length $12$ 
and the indices of uncolored nodes are $5,9$; we thus have 
$$\ell(\gamma)=12+(5-1)+(9-2)=23.$$
For given $(p,m,q)$, one can check that 
\begin{itemize}
    \item 
each $(p,m,q)$-preclan with minimal length is of length $0$ and it is given by 
$$
\text{$m$ many uncolored nodes $\clan{,}$'s followed by a matchless clan}; 
$$
    \item 
there is a unique $(p,m,q)$-preclan with maximal length given by 
$$\text{a rainbow clan followed by $m$ many uncolored nodes $\clan{,}$'s}.
$$
\end{itemize}
The case $m=0$ follows from \cite{WY}, and the general case follows from the definition of length function in \eqref{eq:lenghts}.

\begin{lemma}\label{lem:preclanisconn}
For a $(p,m,q)$-preclan $\gamma$, it is of maximal (resp., minimal) length if and only if it is maximal (resp., minimal) under the weak order.

\end{lemma}
\begin{proof}
The case when $m=0$ (i.e. $(p,q)$-clans) is implied in Wyser and Yong \cite[Section 2.1]{WY}; see also Richardson and Springer \cite[Theorem 4.6]{RS92}. 

Assume that $\gamma$ is maximal, i.e., $s_k*\gamma = \gamma$ for all $1\leq k\leq n-1$. 
It follows from  \eqref{eq:clanweak3} that the last $m$ nodes are all uncolored nodes $\clan{,}$'s. That is, $\gamma$ is given by 
$$
\text{a clan followed by $m$ many uncolored nodes $\clan{,}$'s}.$$
The lemma follows from the case $m=0$. 

Assume $\gamma$ is minimal, i.e. $s_k*\gamma' \neq \gamma$ for all $1\leq k\leq n-1$ and $\gamma'$. 
It follows from  \eqref{eq:clanweak3} that the first $m$ nodes are all uncolored nodes $\clan{,}$'s. That is, $\gamma$ is given by 
$$
\text{$m$ many uncolored nodes $\clan{,}$'s followed by a clan}.$$
Again, the lemma follows from the case $m=0$. 
\end{proof}

\begin{remark}
As pointed out by \cite[Section 6.6]{Mars}, the ``minimal'' part of Lemma \ref{lem:preclanisconn} is manifestly true when $m=0$, since $(G,S)$ forms a symmetric pair. However, this property is not always true for general spherical subgroups. 
\end{remark}

\begin{theorem}\label{thm:dim=length+}
For any $(p,m,q)$-preclan $\gamma$, we have 
\begin{equation}\label{eq:length}
\dim S\dot{\gamma}B/B=\ell(\gamma)+\binom{p}{2}+\binom{q}{2}+\binom{m}{2}.
\end{equation}
\end{theorem}

\begin{proof}
It is easy to check that  $\ell(s_k*\gamma)=\ell(\gamma)+1$ if $s_k*\gamma\neq \gamma$. 
Since $\dim(s_k*\mathbb{O})=\dim\mathbb{O}+1$ if $s_k*\mathbb{O}\neq \mathbb{O}$, by Lemma \ref{lem:preclanisconn}, it suffices to show that \eqref{eq:length} holds when $\gamma$ is of maximal length. 
On one hand, $S\dot{\gamma}B/B$ is dense in the flag variety, thus it has dimension $\binom{n}{2}$. 
On the other hand, 
$$\begin{aligned}
\ell(\gamma) & = \ell(\text{the rainbow $(p,q)$-clan})+m(p+q) = pq+m(p+q).
\end{aligned}$$
The theorem follows from the obvious identity 
$\binom{n}{2}=\binom{p}{2}+\binom{q}{2}+\binom{m}{2}+pq+m(p+q)$. 
\end{proof}

\subsection{Richardson preclans}
For a $(p,m,q)$-preclan $\gamma$, we can define a $p$-inverse Grassmannian permutation $v_{\gamma}$ and a $q$-inverse Grassmannian $u_\gamma$ as follows.
\begin{itemize}
    \item 
To define $v_\gamma$,  first label the left-ends and $\clan{+}$'s of $\gamma$ from left to right with $1,2,\ldots,p$, and then label the other nodes by $p+1,\ldots,n$ from left to right. Then $v_\gamma$ is obtained by reading the labels of $\gamma$ from left to right.

\item Similarly, to define $u_\gamma$,  first label the left-ends and $\clan{-}$'s of $\gamma$ from left to right with $1,2,\ldots,q$, and then label the other nodes by $q+1,\ldots,n$ from left to right. Then $u_\gamma$ is obtained by reading the labels of $\gamma$ from left to right.
\end{itemize}
For example, let $\gamma$ be the preclan in \eqref{eq:preclaneg}. We have 
\begin{equation}\label{uvgamma}
\def\m#1{\makebox[1.2pc]{$#1$}}
\begin{aligned}
\gamma&=\,\clan{6-84,-..,+.}\\
v_\gamma &= \,\m{1}\m{5}\m{2}\m{3}\m{6}\m{7}\m{8}\m{9}\m{10}\m{4}\m{11}\\ 
u_\gamma &= \,\m{1}\m{2}\m{3}\m{4}\m{6}\m{5}\m{7}\m{8}\m{9}\m{10}\m{11}.\\
\end{aligned}
\end{equation}

Recall the subgroups $P,Q,S$ of $GL_n$ introduced in the Introduction. 

\begin{theorem}\label{thm:gammainPQorb}
For any $(p,m,q)$-preclan $\gamma$, we have 
\begin{equation}
S\dot{\gamma}B/B\subseteq Pv_\gamma B/B\cap Qw_0u_\gamma B/B.
\end{equation}
\end{theorem}
\begin{proof}Let us sketch the key steps.
\begin{itemize}
    \item The $P$-orbits are bijective to certain isomorphism classes of quiver representations of type $A$, and they are naturally indexed by $p$-subsets of $[n]$, equivalently, by $p$-inverse Grassmannian permutations, see Proposition \ref{prop:Porbis}.
    
\item We can realize the map  $S\dot{\gamma}B/B\mapsto P\dot{\gamma} B/B$ sending an $S$-orbit to its $P$-orbit as the operator of deleting a vertex of the corresponding quiver. By decomposing the representations of the quiver after deleting a vertex into indecomposables, we can show that the map is given by $S\dot{\gamma}B/B\mapsto Pv_\gamma B/B$, see Proposition \ref{prop:StoP}.

    \item Similar derivation  for $Q$, see Proposition \ref{prop:StoQ}. 
\end{itemize}
The complete proof is given in Section \ref{sec:RichApp}. 
\end{proof}

\begin{theorem}\label{thm:equivuandv}
Let $v\in S_n$ be a $p$-inverse Grassmannian permutation, $u\in S_n$ be a $q$-inverse Grassmannian permutation. 
The following statements are equivalent. 
\begin{itemize}
    \item[(1)] there exists a $(p,m,q)$-preclan $\gamma$ such that $v_\gamma=v$ and $u_\gamma=u$; 
    \item[(2)] the intersection $PvB/B\cap Qw_0uB/B$ is nonempty; 
    \item[(3)] the Richardson variety $\overline{B^-vB/B}\cap \overline{Bw_0uB/B}$ is nonempty; 
    \item[(4)] the open Richardson variety $B^-vB/B\cap Bw_0uB/B$ is nonempty; 
    \item[(5)] $v\leq w_0u$ in the Bruhat order. 
\end{itemize}
\end{theorem}

\begin{proof}
We prove these equivalences in the following steps.

\medbreak
\paragraph{(1) $\Leftrightarrow$ (2)}
Note that $PvB/B\cap Qw_0uB/B$ can be decomposed into $S$-orbits, since $S=P\cap Q$. This equivalence follows from Theorem \ref{thm:gammainPQorb}.

\medbreak
\paragraph{(4) $\Rightarrow$ (2) $\Rightarrow$ (3)}
Note that 
$$\begin{aligned}
B^-vB/B& \subseteq 
PvB/B \subseteq
\overline{PvB/B}=\overline{B^-vB/B},\\
Bw_0uB/B & \subseteq
Qw_0uB/B \subseteq
\overline{Qw_0uB/B}
= \overline{Bw_0uB/B}.
\end{aligned}
$$
We have 
$$
\begin{aligned}
B^-vB/B\cap Bw_0uB/B
& \subseteq 
PvB/B\cap Qw_0uB/B\\
& \subseteq 
\overline{PvB/B}\cap 
\overline{Qw_0uB/B}\\
&= \overline{B^-vB/B}\cap \overline{Bw_0uB/B}.
\end{aligned}$$

\medbreak
\paragraph{(3) $\Leftrightarrow$ (4) $\Leftrightarrow$ (5)}
This is a standard fact of Richardson variety. 
\end{proof}

\begin{definition}\label{def:Richardosnpreclan}
A preclan $\gamma$ is called a \emph{Richardson} preclan, if 
\begin{itemize}
    \item the matchings of $\gamma$ are pairwise non-crossing; 
    \item uncolored nodes $\clan{,}$'s of $\gamma$ are not covered by any matchings. 
\end{itemize}
\end{definition}

\begin{lemma}\label{richardmax}
A Richardson preclan $\gamma$ is the preclan of maximal length with fixed $u_\gamma$ and $v_\gamma$.  
\end{lemma}

\begin{proof}
Actually, one can check the length defined in \eqref{eq:lenghts} can be rewritten as
$$
\begin{aligned}
\ell(\gamma) & = 
\sum\{\text{indices of right-ends and uncolored $\clan{,}$}\}
- 
\sum\{\text{indices of left-ends}\}\\
&\quad - \texttt{\#}\{\text{crossings in $\gamma$}\}
- \texttt{\#}\{\text{covering pairs in $\gamma$}\}
-\binom{m+1}{2},
\end{aligned}
$$
where a covering pair is a pair of an uncolored node $\clan{,}$ with   a matching covering it. 
The lemma follows from the 
following two length increasing operators preserving $u_\gamma$ and $v_\gamma$: 
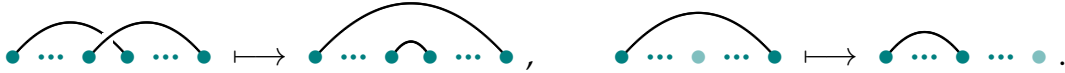
$$
\clan{3\dots3.\dots.}
\longmapsto 
\clan{5\dots1.\dots.},
\qquad 
\clan{4\dots,\dots.}
\longmapsto 
\clan{2\dots.\dots,}.
$$
Both operators will give a preclan with fewer crossings or covering pairs, thus increase length. 
For example, for the preclan in \eqref{eq:preclaneg}, applying these two operators, we will arrive at
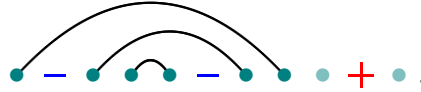
\begin{equation*}
\clan{7-41.-..,+,}.\qedhere 
\end{equation*}
\end{proof}

By Theorem \ref{thm:equivuandv} and Lemma \ref{richardmax}, there exists a unique  Richardson preclan $\gamma_{v,u}$, whenever $v \le w_0 u$ 
(which always holds when $m \gg 0$).
Let $v\in S_\infty$ be a $p$-inverse Grassmannian permutation,   and $u\in S_\infty$ be  a $q$-inverse Grassmannian permutation.  We construct $\gamma_{v,u}$ 
explicitly as below.
Let us first construct a 2-colored Motzkin path \cite{BLPP,Chen} $T_{v,u}$ as follows. For $i\ge 1$, the  $i$-th step $T_i$ of $T_{v,u}$ is determined according to the following four cases.
$$\begin{array}{c|c|c}
     \vphantom{\dfrac12}
     & i\in v^{-1}([p]) & i\notin v^{-1}([p]) \\\hline
     \vphantom{\dfrac12}
i\in u^{-1}([q]) & 
\begin{matrix}
\begin{tikzpicture}[scale=0.5]
\draw[ultra thick, gray, ->] (0,1) -- (1,0);
\end{tikzpicture}
\end{matrix} 
& 
\begin{matrix}
\begin{tikzpicture}[scale=0.5]
\draw[ultra thick, blue, ->] (0,0.5) -- (1,0.5);
\end{tikzpicture}
\end{matrix} 
\\\hline
     \vphantom{\dfrac12}
i\notin u^{-1}([q]) & 
\begin{matrix}
\begin{tikzpicture}[scale=0.5]
\draw[ultra thick, red, ->] (0,0.5) -- (1,0.5);
\end{tikzpicture}
\end{matrix} 
 & 
\begin{matrix}
\begin{tikzpicture}[scale=0.5]
\draw[ultra thick, gray, ->] (0,0) -- (1,1);
\end{tikzpicture}
\end{matrix} 
\end{array}$$

\begin{definition}\label{gammvu}
Taking $m\gg 0$,  construct a  $(p,m,q)$-Richardson preclan $\gamma_{v,u}$ based on $T_{v,u}$ as follows. 
If the $i$-th step of $T_{v,u}$ is $\searrow$, then locate the minimal $j>i$ such that the $j$-th step of $T_{v,u}$ is $\nearrow$ and  of the same height with the $i$-th step,  match $i$ with $j$. Replace each unmatched $\nearrow$ step by a node $\clan{,}$. 
Color the horizontal steps by $\clan{+}$ or $\clan{-}$ as indicated in $T_{v,u}$. 
\end{definition}

Since we take $m\gg 0 $, the construction of $\gamma_{v,u}$ is always feasible, i.e., each $\searrow$ step can find a $\nearrow$ on its right. Obviously, we have   $v_{\gamma_{v,u}} =v$ and $u_{\gamma_{v,u}}=u$.
For example, consider $v$ and $u$ in Example \ref{eg:maineg}. 
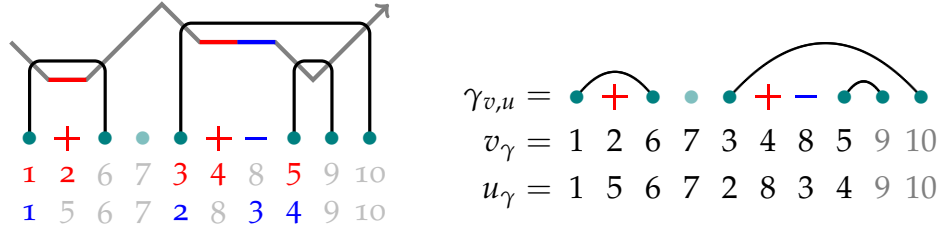
$$
\begin{matrix}
\begin{tikzpicture}[scale=0.5]
\draw[ultra thick, gray, ->]
    (0.5,0.5)--++(1,-1)--++(1,0)--++(1,1)--++(1,1)--++(1,-1)--++(1,0)--++(1,0)--++(1,-1)--++(1,1)--++(1,1);
\draw[ultra thick, red]
    (1.5,-.5)--(2.5,-.5);
\draw[ultra thick, red]
    (5.5,0.5)--(6.5,0.5);
\draw[ultra thick, blue]
    (6.5,0.5)--(7.5,0.5);
\draw[very thick, rounded corners]
    (1,-2)--(1,0)--(3,0)--(3,-2);
\draw[very thick, rounded corners]
    (5,-2)--(5,1)--(10,1)--(10,-2);
\draw[very thick, rounded corners]
    (8,-2)--(8,0)--(9,0)--(9,-2);
\node at (1,-2) {$\clan{.}$};
\node at (2,-2){$\clan{+}$};
\node at (3,-2) {$\clan{.}$};
\node at (4,-2){$\clan{,}$};
\node at (5,-2) {$\clan{.}$};
\node at (6,-2){$\clan{+}$};
\node at (7,-2) {$\clan{-}$};
\node at (8,-2){$\clan{.}$};
\node at (9,-2) {$\clan{.}$};
\node at (10,-2){$\clan{.}$};
\node at (1,-3) {\color{red}1};
\node at (2,-3) {\color{red}2};
\node at (3,-3) {\color{lightgray}6};
\node at (4,-3) {\color{lightgray}7};
\node at (5,-3) {\color{red}3};
\node at (6,-3) {\color{red}4};
\node at (7,-3) {\color{lightgray}8};
\node at (8,-3) {\color{red}5};
\node at (9,-3) {\color{lightgray}9};
\node at (10,-3) {\color{lightgray}10};
\node at (1,-4) {\color{blue}1};
\node at (2,-4) {\color{lightgray}5};
\node at (3,-4) {\color{lightgray}6};
\node at (4,-4) {\color{lightgray}7};
\node at (5,-4) {\color{blue}2};
\node at (6,-4) {\color{lightgray}8};
\node at (7,-4) {\color{blue}3};
\node at (8,-4) {\color{blue}4};
\node at (9,-4) {\color{lightgray}9};
\node at (10,-4) {\color{lightgray}10};
\end{tikzpicture}
\end{matrix}\qquad 
\def\m#1{\makebox[1.2pc]{$#1$}}
\begin{aligned}
\gamma_{v,u}=&\,\clan{2+.,5+-1..}\\
v_\gamma=& \,\m{1}\m{2}\m{6}\m{7}\m{3}\m{4}\m{8}\m{5}\m{\color{gray}9}\m{\color{gray}10}\\
u_\gamma=&\,\m{1}\m{5}\m{6}\m{7}\m{2}\m{8}\m{3}\m{4}\m{\color{gray}9}\m{\color{gray}10}
\end{aligned}
$$

\begin{theorem}\label{prop:RichApp}
For any $v,u$ satisfying any one of the conditions in Theorem \ref{thm:equivuandv}, $\gamma=\gamma_{v,u}$ is the unique  Richardson preclan  such that $v_{\gamma}=v$ and $u_{\gamma} =u$. 
Moreover, we have 
$$\overline{S\dot{\gamma}B/B}=\overline{B^-v_\gamma B/B}\cap \overline{Bw_0u_\gamma B/B}. $$
\end{theorem}
\begin{proof}
From the proof of Theorem \ref{thm:equivuandv}, 
the closure of the intersection $PvB/B\cap Qw_0uB/B$ is a Richardson variety, and thus $PvB/B\cap Qw_0uB/B$ is irreducible. 

Since the intersection $PvB/B\cap Qw_0uB/B$ is $S$-invariant and irreducible, there is a unique $S$-orbit dense inside it, i.e., the orbit of maximal dimension. 
The theorem  follows from Theorem \ref{thm:dim=length+},  Theorem \ref{thm:gammainPQorb}, and Lemma \ref{richardmax}.
\end{proof}

\section{Polynomial Representatives}\label{Polynomial Representatives}

Recall that 
\begin{equation}\label{eq:HTGBcongrho}
H_T^*(G/B)\cong  \bigoplus_{\lambda\leq (n-1,\ldots,1,0)} \mathbb{Q}[t_1,\ldots,t_n]x^{\lambda}.
\end{equation}
For each $(p,m,q)$-preclan $\gamma$, we define a polynomial
$$\Upsilon_{\gamma}(x;t) \in  \text{the right-hand side of \eqref{eq:HTGBcongrho}}$$
to represent the fundamental class $[\overline{S\dot{\gamma}B/B}]_T$.

\begin{theorem}\label{thm:dkonUpsilon}
For a $(p,m,q)$-preclan $\gamma$ and any $1\leq k\leq n-1$, we have
$$\partial_k 
\Upsilon_\gamma(x;t)
=\begin{cases}
\Upsilon_{s_k*\gamma}(x;t), & \ell(s_k*\gamma)=\ell(\gamma)+1,\\
0,& \text{otherwise}. 
\end{cases}$$
\end{theorem}
\begin{proof}
Let $\pi_k:G/B\to G/P_k$ be the natural projection. By Brion \cite{Brion}, for any spherical subgroup orbit $\mathbb{O}$, we have 
$$\pi_{k}^*\pi_{k*} [\overline{\mathbb{O}}]
= \begin{cases}
[\overline{s_k*\mathbb{O}}], 
& \text{type (I), (IIa), (IIIa)},\\
2[\overline{s_k*\mathbb{O}}], 
& \text{type (IVa)},\\
0, & \text{otherwise}. 
\end{cases}$$
The argument for the non-equivariant case also works for the equivariant case. 
Note that the right-hand side of \eqref{eq:HTGBcongrho} is preserved under $\partial_k$. 
When passing through \eqref{eq:HTGBcongrho}, the divided difference operator $\partial_k$ corresponds to $\pi_k^*\pi_{k*}$.  
Thus the Theorem follows from Theorem \ref{th:weakordApp}. 
\end{proof}

For any $f(x;t)\in \mathbb{Q}[x_i;t_j]_{i,j}$, we define 
$$f(t;t)=f|_{x_1=t_1,x_2=t_2,\ldots}\in \mathbb{Q}[t_j]_j.$$

\begin{prop}\label{prop:eqSchubertexp}
For a $(p,m,q)$-preclan $\gamma$, we have 
$$
\Upsilon_\gamma(x;t)=\sum_{w\in S_n}
\Upsilon_{w*\gamma}(t;t)\cdot 
\mathfrak{S}_w(x;t),
$$
where the sum is over all $w\in S_n$ such that $\ell(w)+\ell(\gamma)=\ell(w*\gamma)$.  
\end{prop}
\begin{proof}
Note that we have an isomorphism of vector spaces
$$
\bigoplus_{\lambda\leq (n-1,\ldots,1,0)} \mathbb{Q}[t_1,\ldots,t_n]\cdot
x^{\lambda}
= 
\bigoplus_{w\in S_n}
\mathbb{Q}[t_1,\ldots,t_n]\cdot 
\mathfrak{S}_w(x;t).$$
As a result, the double Schubert expansion of $\Upsilon_\gamma(x;t)$ only involves those $w\in S_n$. The Proposition follows immediately from Theorem \ref{thm:dkonUpsilon} and the fact 
$$\mathfrak{S}_w(t;t)=\begin{cases}
1, & w=e,\\
0, & \text{otherwise}.
\end{cases}$$
The same argument was used by the authors in \cite[Section 4.2]{WE}.
\end{proof}

\begin{prop}\label{prop:noneqSchubertexp}
For a $(p,m,q)$-preclan $\gamma$, we have 
$$\Upsilon_\gamma(x;0)=\sum_{w\in S_n} \mathfrak{S}_w(x;0),$$
where the sum is over all $w\in S_n$ such that $w*\gamma = \gamma_0$ and $\ell(w)+\ell(\gamma)=\ell(\gamma_0)$ where $\gamma_0$ is the rainbow $(p,q)$-clan followed by $m$  uncolored nodes $\clan{,}$'s.
\end{prop}
\begin{proof}
Note that the polynomial $\Upsilon_\gamma(x;t)$ is 
homogeneous of degree the codimension of $S\dot{\gamma}B/B$. 
We have 
$\Upsilon_\gamma(0;0)=0$ except when $\gamma=\gamma_0$. 
When $\gamma=\gamma_0$, the orbit closure $\overline{S\dot{\gamma}B/B}=G/B$, thus $\Upsilon_\gamma(x;t)=1$. 
The assertion follows from Proposition \ref{prop:eqSchubertexp}. 
\end{proof}

\begin{coro}\label{coro:RichUpsilon0}
For a Richardson $(p,m,q)$-preclan $\gamma$, we have 
\begin{align}\label{eq:upsilo}
\Upsilon_\gamma(x;0) =
\mathfrak{S}_{v_\gamma}(x;0)
\mathfrak{S}_{u_\gamma}(x;0).
\end{align}
\end{coro}
\begin{proof}
By Theorem \ref{prop:RichApp}, the right-hand side of \eqref{eq:upsilo} represents $[\overline{S\dot{\gamma}B/B}]$.  Since by Proposition \ref{prop:noneqSchubertexp}, $\Upsilon_\gamma(x;0) $ involves only $w\in S_n$, it is enough to show 
\begin{align}\label{eq:schubuv}
\mathfrak{S}_{v_\gamma}(x;0)
\mathfrak{S}_{u_\gamma}(x;0)\in \bigoplus_{\lambda\leq (n-1,\ldots,0)}\mathbb{Q}\cdot x^\lambda
=\bigoplus_{w\in S_n}\mathbb{Q}\cdot \mathfrak{S}_w(x;0).
\end{align}
Let us first pick $N\gg 0$ such that \eqref{eq:schubuv} is true with $S_n$ replaced by $S_N$.  

Let $\tilde{\gamma}$ be the preclan obtained by adding $N-n$ many uncolored nodes $\clan{,}$'s at the end of $\gamma$. 
Note that adding uncolored nodes $\clan{,}$'s at the end of $\gamma$, the permutations $v_\gamma$ and $u_\gamma$ are unchanged. 
By Theorem \ref{prop:RichApp}, we also have 
$$\Upsilon_{\tilde{\gamma}}(x;0) =
\mathfrak{S}_{v_\gamma}(x;0)\cdot
\mathfrak{S}_{u_\gamma}(x;0).$$
Note that a priori, the Schubert expansion of $\Upsilon_{\tilde{\gamma}}(x;0)$ may involve $w\in S_N$. By Proposition \ref{prop:noneqSchubertexp}, it suffices to show that if $\ell(w*\tilde{\gamma})=\ell(\tilde{\gamma})+\ell(w)$, then we must have $
w\in S_n.$

By the definition of weak order \eqref{eq:clanweak1}, \eqref{eq:clanweak2}, \eqref{eq:clanweak3}, we have $\ell(w*\tilde{\gamma})=\ell(\tilde{\gamma})$ for
$w=s_{k}$ for $n\leq k\leq N-1$. 
The general case follows from the observation that adding uncolored nodes $\clan{,}$'s commutes with $\gamma\mapsto s_k*\gamma$ for $1\leq k\leq n-1$. 
\end{proof}

\begin{coro}\label{coro:RichUpsilon}
For a Richardson $(p,m,q)$-preclan $\gamma$, we have 
\begin{align}\label{eq:doubleupsilo}
\Upsilon_\gamma(x;t) =
\mathfrak{S}_{v_\gamma}(x;t)\cdot
\mathfrak{S}_{u_\gamma}(x;\cev{t}),
\end{align}
where $\cev{t}=(t_n,\ldots,t_1)$.
\end{coro}
\begin{proof}
By Theorem \ref{prop:RichApp}, the right-hand side of \eqref{eq:doubleupsilo} represents $[\overline{S\dot{\gamma}B/B}]_T$. 
It suffices to show 
$$
\mathfrak{S}_{v_\gamma}(x;t)
\mathfrak{S}_{u_\gamma}(x;\cev{t})\in \text{the right-hand side of \eqref{eq:HTGBcongrho}}.$$
Since Schubert polynomial is contained in $\mathbb{N}[x_i-t_j]_{i,j}$, so is $\mathfrak{S}_{v_\gamma}(x;t)
\mathfrak{S}_{u_\gamma}(x;\cev{t})$. 
For any $f(x;t)\in \mathbb{N}[x_i-t_j]_{i,j}$, we have 
$$f(x;t)\in f(x;0)+
\sum_{\lambda} \mathbb{Q}[t_1,\ldots,t_n]x^\lambda$$
where the sum over $\lambda\leq \mu$ componentwise for some $\mu$ such that the coefficient of $x^\mu$ in $f(x;0)$ is nonzero. 
Thus it suffices to show 
$$\mathfrak{S}_{v_\gamma}(x;0)\cdot
\mathfrak{S}_{u_\gamma}(x;0)\in \bigoplus_{\lambda\leq (n-1,\ldots,0)}\mathbb{Q}x^\lambda.$$
This follows from Corollary \ref{coro:RichUpsilon0} above. 
\end{proof}

\begin{remark}
We have the following remarks. 
\begin{itemize}
    \item The method of proving Corollary \ref{coro:RichUpsilon0} does not generalize to Corollary \ref{coro:RichUpsilon} as $\cev{t}$ might change when replacing $n$ by $N$. 

    \item Corollary \ref{coro:RichUpsilon} and Theorem \ref{thm:dkonUpsilon} characterize $\Upsilon_\gamma(x;t)$. 
    \begin{itemize}
        \item[(1)] Note that a preclan $\gamma$ of length $0$, i.e. $m$ many uncolored nodes $\clan{,}$ followed by a matchless clan, is a Richardson preclan, and thus $\Upsilon_\gamma(x;t)$ is given by Corollary \ref{coro:RichUpsilon}. 
        \item[(2)] For general preclan $\gamma$ of positive length, by Lemma \ref{lem:preclanisconn}, we can write it as $\gamma=s_k*\gamma'$ for a preclan $\gamma'$ of smaller length and $1\leq k\leq n-1$. We can determine $\Upsilon_\gamma(x;t)$ using Theorem \ref{thm:dkonUpsilon} and induction. 
    \end{itemize}
    Thus when $m=0$, these polynomials agree with the one defined by Wyser and Yong \cite{WY}. 

    \item Proposition \ref{prop:noneqSchubertexp} and Corollary \ref{coro:RichUpsilon0} imply the result of Pechenik and Weigandt \cite{Pandt}. 
    Let us briefly explain the equivalence. 
For each $(p,m,q)$-preclan, we can complete it into a $(p+m,q+m)$-clan by adding $m$ new nodes before the first node, and matching the $i$-th new node (from right) with the $i$-th uncolored node. 
For example, for the preclan in \eqref{eq:preclaneg}, we get 
$$
\clan{{10}56-84.-...+.}.
$$
One can check this defines an embedding of the set of $(p,m,q)$-preclans to the set of $(p+m,q+m)$-clans and it intertwines the actions $s_i*-$ and $s_{i+m}*-$. 
We can further embed it into the set of back stable clans of Pechenik and Weigandt \cite{Pandt}.

\end{itemize}
\end{remark}

\begin{example}\label{eg:Upsilon}
Let us compute $\Upsilon_{\gamma}(x,t)$ for all $(1,1,1)$-preclans.
\def\term#1#2#3#4{
#1 \qquad 
\makebox[2pc]{$#2$}\qquad 
\makebox[2pc]{$#3$}\qquad 
\makebox[0.4\linewidth][l]{$#4$}\qquad 
}
$$\term{\text{preclan $\gamma$}}{v_\gamma}{u_\gamma}{\Upsilon_\gamma(x,t)}$$
$$\term{\clan{,+-}}{213}{231}{(x_1-t_1)\cdot (x_1-t_3)(x_2-t_3)}$$
$$\term{\clan{,-+}}{231}{213}{(x_1-t_1)(x_2-t_1)\cdot (x_1-t_3)}$$
$$\term{\clan{,1.}}{213}{213}{(x_1-t_1)\cdot (x_1-t_3)}$$
$$\term{\clan{+,-}}{123}{231}{(x_1-t_3)(x_2-t_3)}$$
$$\term{\clan{-,+}}{231}{123}{(x_1-t_1)(x_2-t_1)}$$
$$\term{\clan{+-,}}{123}{213}{(x_1-t_3)}$$
$$\term{\clan{-+,}}{213}{123}{(x_1-t_1)}$$
$$\term{\clan{2,.}}{123}{123}{x_1+x_2-t_1-t_3}$$
$$\term{\clan{1.,}}{123}{123}{1}$$
\end{example}

\section{Localization and Globalization}\label{Localization and Globalization}

The main purpose of this section is to relate $\Upsilon_\gamma(t;t)$ with the Grassmannian $Gr_q(\mathbb{C}^n)$.
The maps used in the argument of this section can be summarized in the following diagram 
\begin{equation}\label{eq:bigdiagram}
\begin{matrix}\xymatrix{
&&&&\\
G/B & P\cdot B/B\ar[l]_-f & N\cdot B/B \ar[l]_-g & 1\cdot B\ar[l]_-h
\ar`u[ul]`[lll][lll]\\
 && N\ar[u]_-\wr\ar[d]^-\wr & \{1_n\}\ar[l]\ar[u]\ar[d]\\
& Q\backslash G & Q\backslash Q\cdot N \ar[l]_-\phi & Q\cdot 1\ar[l]_-\psi
\ar`d[d]`[dll][ll]\\
&&&&} 
\end{matrix}
\end{equation}
For the last row of \eqref{eq:bigdiagram}, the classical convention of $T$-action is 
$$t\cdot Qg=Qgt,\quad \text{ for $t\in T,\ Qg\in Q\backslash G$}.$$
In particular, we have the following sign issue. 

\begin{lemma}\label{lemma:diag1}
The induced map 
$$\mathbb{Q}[t_1,\ldots,t_n]\cong H_T^*(1\cdot B/B)
\stackrel{\sim}\longrightarrow 
H_T^*(Q\backslash Q\cdot 1)\cong \mathbb{Q}[t_1,\ldots,t_n]$$
is given by sending $t_i$ to $-t_i$. 
\end{lemma}

\subsection{Localization}
Geometrically, $\Upsilon_\gamma(t;t)$ is the image of $[\overline{S\dot{\gamma}B/B}]_T$ under the localization map 
\begin{equation}\label{eq:locatid}
-|_e:H_T^*(G/B)\longrightarrow H_T^*(1\cdot B/B)\cong \mathbb{Q}[t_1,\ldots,t_n]. 
\end{equation}
Consider the upper row of \eqref{eq:bigdiagram}. 
We can decompose \eqref{eq:locatid} into maps 
$$H_T^*(G/B)
\stackrel{f^*}\longrightarrow 
H_T^*(P\cdot B/B)
\stackrel{g^*}\longrightarrow 
H_T^*(N\cdot B/B)
\stackrel{h^*}\longrightarrow 
H_T^*(1\cdot B/B). $$

\begin{lemma}\label{lemma:diag2}
For any $(p,m,q)$-preclan $\gamma$, 
$$g^*f^*[\overline{S\dot{\gamma}B/B}]_T
= 
\begin{cases}
[\overline{S\dot{\gamma}B/B\cap N\cdot B/B}]_T, 
& v_\gamma = e,\\
0, & \text{otherwise},
\end{cases}
$$
where the closure of the right-hand side is taken inside $N\cdot B/B$. 
In the first case $v_\gamma=e$, the intersection $S\dot{\gamma}B/B\cap N\cdot B/B$ is always nonempty. 
\end{lemma}
\begin{proof}
Let us first study the behavior of $f^*$. 
Note that $P\cdot B/B$ is an $S$-invariant open neighborhood of $1\cdot B/B$. We thus have 
$$f^*[\overline{S\dot{\gamma}B/B}]_T
= [\overline{S\dot{\gamma}B/B}\cap P\cdot B/B]_T.$$
Note that 
$$\overline{S\dot{\gamma}B/B}\cap P\cdot B/B
=\overline{S\dot{\gamma}B/B\cap P\cdot B/B}
=\begin{cases}
\overline{S\dot{\gamma}B/B}, & S\dot{\gamma}B/B\subseteq P\cdot B/B,\\
\varnothing, & \text{otherwise}. 
\end{cases}
$$
where the closures of the right-hand two terms are all taken inside $P\cdot B/B$. 
By Theorem \ref{thm:gammainPQorb}, 
the first case happens if and only if $v_\gamma = e$. 

Then let us study the behavior of $g^*$. 
We can identify $P\cdot B/B\cong P/\bar{B}$. 
By Richardson \cite{Richardson92}, the $N$-orbit $N\cdot B/B$ intersects every $S$-orbit $S\dot{\gamma}B/B$ transversally, and the intersection is irreducible. 
We thus have 
$$g^*[\overline{S\dot{\gamma}B/B}\cap P\cdot B/B]_T
=[\overline{S\dot{\gamma}B/B}\cap N\cdot B/B]_T,$$
when $v_\gamma=e$. 

It is well-known that transversal intersection preserves the closure relation. 
Precisely, in our previous work \cite[Lemma 3.2]{WE}, we proved that for any $K$-orbit,  taking closure
commutes with the transversal intersection with any closed $N$-orbits. 
But it is also true for all $N$-orbits, since they are locally closed \cite[Section 8.3]{Humph} and restriction to an open subset also commutes with taking closure. We thus have 
$$\overline{S\dot{\gamma}B/B}\cap N\cdot B/B = \overline{S\dot{\gamma}B/B\cap N\cdot B/B},$$
where the closure of the right-hand side is taken inside $N\cdot B/B$. 

To see the intersection is nonempty, we need two further steps. 
Let us first deal with the case $p=0$. 
Recall and introduce 
$$\def\O{\color{lightgray}0}
S=\left[\begin{matrix}
GL_m & * \\ \O & GL_q
\end{matrix}\right],\quad 
N=\left[\begin{matrix}
1_m & \O \\ * & 1_q
\end{matrix}\right],\quad 
M:=\left[\begin{matrix}
GL_m & \O \\ * & GL_q
\end{matrix}\right],\quad 
L:=\left[\begin{matrix}
GL_m & \O \\ \O & GL_q
\end{matrix}\right].$$
Note that $v_{\gamma}=e$ for any $(0,m,q)$-preclan. 
Note that 
$SgB/B\cap MB/B\neq \varnothing$ for any $g\in G$, since $SgB/B=QgB/B$ is a union of Schubert cells and $MB/B$ contains the dense opposite Schubert cells. 
Since $M=NL$ and $S$ is $L$-invariant, we have 
$SgB/B\cap NB/B\neq \varnothing$. 
That is, $SgB\cap N\neq \varnothing$ for all $g\in G=GL_{m+q}$. 

For general $p$, we take a one-parameter subgroup $\rho:\mathbb{C}^\times \to T$ such that all point $g B\in B^-wB/B$, the limit
$\lim_{t\to \infty}\rho(t)gB=wB$. 
Assume $v_\gamma=e$, then by  Definition \ref{def:dotgamma}, we have $\dot{\gamma}\in P$. 
Since $P\cdot B/B=\bigsqcup_{w\in S_p\times S_{m+q}}B^-wB/B$, the limit
$\lim_{t\to\infty}\rho(t)\dot{\gamma}B$ exists and the limit is $wB\in P\cdot B/B$ for some $w=w_1\times w_2\in S_p\times S_{m+q}$. 
Note that the permutation matrix $w_1\times e$ is an element of $S$. We thus have $(e\times w_2)\cdot B\in \overline{S\dot{\gamma}B/B}$.
Let us introduce 
$$\def\O{\color{lightgray}0}
S'=\left[\begin{matrix}
1_p &\O &\O\\
\O & GL_m & * \\ \O & \O & GL_q
\end{matrix}\right]\subseteq S,\quad 
N'=\left[\begin{matrix}
1_p& \O& \O\\
\O & 1_m & \O \\\O & * & 1_q
\end{matrix}\right]\subseteq N ,\quad 
B'=
\left[\begin{matrix}
1_p&\O & \O\\
\O & B_m & * \\\O & \O & B_q
\end{matrix}\right]\subseteq B.
$$
By the case $p=0$ above, we can find $s\in S'$ and $b\in B'$ such that $s(e\times w_2)b\in N'$. 
In particular, 
$s(e\times w_2)\in 
\overline{S\dot{\gamma}B/B}\cap N\cdot B/B = \overline{S\dot{\gamma}B/B\cap N\cdot B/B}\neq \varnothing$ and thus $S\dot{\gamma}B/B\cap N\cdot B/B\neq \varnothing$. 
\end{proof}

\subsection{Globalization}
Next, we will globalize the intersection $S\dot{\gamma}B/B\cap N\cdot B/B$ considered above to Grassmannian via the vertical map of \eqref{eq:bigdiagram}: 
$$
N\cdot B/B
\stackrel{\sim}\longleftarrow N\stackrel{\sim}\longrightarrow Q\backslash Q\cdot N.$$

\begin{lemma}\label{lemma:diag3}
For a $(p,m,q)$-preclan $\gamma$ with $v_\gamma=e$, we have 
$$
S\dot{\gamma}B/B\cap N\cdot B/B
\cong Q\backslash Q\dot{\gamma}\bar{B}\cap Q\backslash Q\cdot N
$$
under the vertical map of \eqref{eq:bigdiagram}. 
\end{lemma}
\begin{proof}
Note that under the isomorphism
$N\cong N\cdot B/B$, 
$$N\cap S\dot{\gamma}B\cong 
S\dot{\gamma}B/B\cap N\cdot B/B$$
and under the isomorphism $N\cong Q\backslash Q\cdot N$, 
$$N\cap Q\dot{\gamma}\bar{B}\cong 
Q\backslash Q\dot{\gamma}\bar{B}\cap Q\backslash Q\cdot N.$$
It suffices to show 
$$
N\cap S\dot{\gamma}B = N\cap S\dot{\gamma}\bar{B} = N\cap Q\dot{\gamma}\bar{B}. $$
By Definition \ref{def:dotgamma}, it is easy to see that if $v_\gamma=e$, then $\dot{\gamma}\in P$. 

For any element $x\in N\cap S\dot{\gamma}B$ with $x=s\dot{\gamma}b$ for $s\in S$ and $b\in B$, we have 
$$b=\dot{\gamma}^{-1}s^{-1}x\in 
P\cdot S\cdot N\subseteq P. $$
Thus $b\in \bar{B}=B\cap P$. 
This proves 
$N\cap S\dot{\gamma}B = N\cap S\dot{\gamma}\bar{B}$. 

For any element $x\in N\cap Q\dot{\gamma}\bar{B}$ with $x=q\dot{\gamma}b$ for $q\in Q$ and $b\in\bar{B}$, we have 
$$q=xb^{-1}\dot{\gamma}^{-1}\in N\cdot \bar{B}\cdot P\subseteq P. $$
Thus $q\in S=Q\cap P$. 
This proves $N\cap S\dot{\gamma}\bar{B} = N\cap Q\dot{\gamma}\bar{B}$. 
\end{proof}

Consider the lower row of \eqref{eq:bigdiagram}. 
We have an isomorphism 
$$Q\backslash G
\stackrel{\sim}\longrightarrow
\operatorname{Gr}_q(\mathbb{C}^n),\qquad Qg\longmapsto \operatorname{span}(\text{last $q$ rows of $g$}).$$
Moreover, $Q\backslash Q\cdot N$ coincides with the dense opposite Schubert cell over $\operatorname{Gr}_q(\mathbb{C}^n)$. 
We thus have 
$$\overline{Q\backslash Q\dot{\gamma}\bar{B}}\cap Q\backslash Q\cdot N
= \overline{Q\backslash Q\dot{\gamma}\bar{B}\cap Q\backslash Q\cdot N},$$
where the closure of the right-hand side is taken inside $Q\backslash Q\cdot N$. 
The following Lemma is clear. 

\begin{lemma}\label{lemma:diag4}
For any $(p,m,q)$-preclan $\gamma$ with $v_\gamma=e$, 
$$\phi^*[\overline{Q\backslash Q\dot{\gamma}\bar{B}}]_T
= [\overline{Q\backslash Q\dot{\gamma}\bar{B}\cap Q\backslash Q\cdot N}]_T,$$
where the closure of the right-hand side is taken inside $Q\backslash Q\cdot N$. 
\end{lemma}

Combining all 
Lemmas \ref{lemma:diag1}, \ref{lemma:diag2}, 
\ref{lemma:diag3} and 
\ref{lemma:diag4}, we get the following theorem. 

\begin{theorem}\label{thm:globalization}
For any $(p,m,q)$-preclan $\gamma$, 
$$\Upsilon_\gamma(t;t)
|_{t_i\mapsto -t_i}
= \begin{cases}
[\overline{Q\backslash Q\dot{\gamma}\bar{B}}]_T|_{Q\cdot 1}, & v_\gamma = e,\\
0, & \text{otherwise}.
\end{cases}
$$
\end{theorem}

\section{Positroid Varieties}\label{Positroid Varieties}

The main purpose of this section is to identify $\overline{Q\backslash Q\dot{\gamma}\bar{B}}$ with a positroid variety. 

\begin{definition}\label{fgamma}    
For a $(p,m,q)$-preclan $\gamma$ with $v_\gamma=e$, we associate  an affine permutation $f_\gamma\in \widetilde{S}_n$ to it as follows. 
\begin{itemize}
    \item[(1)] We first define $f_\gamma(*)$ for $1\leq *\leq p$. 
    Note that $v_\gamma=e$ implies that 
    the first $p$ nodes are either $\clan{+}$'s or left-ends. 
    \begin{itemize}
        \item[$\bullet$] If the $i$-th node is $\clan{+}$, then let 
    $$f_\gamma(p+1-i) = p+1-i+\texttt{\#}\{\text{matchings covering the $i$-th node}\}.$$
    
        \item[$\bullet$] If the $i$-th node is the left-end of an $(i,j)$-matching, then let
    $$f_\gamma(p+1-i) = j.$$
    
    \end{itemize}
    \item[(2)] Then let us define $f_\gamma(*)$ for $p<*\leq p+m$. 
    Let    $$f_\gamma(p+i) = \text{the index of the $i$-th uncolored $\clan{,}$}.$$

    \item[(3)] Finally, we define $f_\gamma(*)$ for $p+m<*\leq p+m+q=n$. 
    Write $q=q_1+q_2$ with $q_1$ the number of matchings, and $q_2$ the number of $\clan{-}$'s.
    \begin{itemize}
        \item[$\bullet$] For $1\leq i\leq q_1$, we define 
    $$f_\gamma(p+m+i)=n+i.$$
        \item[$\bullet$] For $1\leq i\leq q_2$, we define
    $$f_\gamma(p+m+q_1+i)=n+\text{the index of the $i$-th $\clan{-}$}.$$
    \end{itemize}    
\end{itemize}

\end{definition}

It is easy to see that the last $q$ values of $f_\gamma$ are increasing.
For example, consider the following $(4,3,4)$-preclan 
\begin{equation}\label{eq:Egvgamma=e}
\def\m#1{\makebox[1.2pc]{$#1$}}
\begin{aligned}
\gamma =&\,  \clan{586+,.-,..,}\\
i = & \,\m{4}\m{3}\m{2}\m{1}\m{5}\m{6}\m{7}\m{8}\m{9}\m{10}\m{11}\\
f_\gamma(i) =& \,\m{6}\m{10}\m{9}\m{4}\m{5}\m{8}\m{11}\m{12}\m{13}\m{14}\m{18}.
\end{aligned}
\end{equation}
We have 
\begin{itemize}
    \item $f_\gamma(1) = 4$, since the $4$-th node is $\clan{+}$, and there are $3$ matchings covering it, we have $f_\gamma(p+1-4)=p+1-4+3=4$. 
    
    \item $f_\gamma(2)=9$, since there is a $(3,9)$-matching, we find $f_\gamma(p+1-3)=9$; 
    similarly for the $(2,10)$-matching and $(1,6)$-matching, we have $f_{\gamma}(3)=10$ and $f_\gamma(4)=6$; 
    \item $f_\gamma(5) = 5,\,f_\gamma(6)=8,\,f_\gamma(7)=11$, since the indices of unmatched $\clan{,}$ are $5,8,11$; 
    \item $f_\gamma(8)=11+1,\,f_\gamma(9)=11+2$ and $f_\gamma(10)=11+3$, since there are three matchings; 
    \item $f_\gamma(11) = 11+7$, since the index of $\clan{-}$ is $7$. 
\end{itemize}

There is an alternative description of the first $p$ values of  $f_{\gamma}$.
Assume that there are $p_1$ $\clan{+}$'s with positions  $i_1<\cdots<i_{p_1}$. Then for any $1\le r\le p_1$, we have
$$i_{r}=r+\#\{\text{matchings covering the $r$-th }\clan{+}\},$$ 
and 
\begin{equation}\label{eq:Lemmafgammaisw-d}
f_{\gamma}(p+1-i_r)=p+1-r.
\end{equation}
In particular, $f_{\gamma}(p+1-i_{p_1})=p+1-p_1=q_1+1$, where $q_1$ is the number of matchings.

\begin{lemma}\label{ggamma}
Let $\gamma$ be a $(p,m,q)$-preclan $\gamma$ with $v_\gamma=e$. Then  $f_\gamma\in \operatorname{Bound}(q,n)$ is well defined and 
\begin{align}
g_{\gamma}:=f_\gamma(1-q)f_\gamma(2-q) \cdots f_\gamma(n-q)\in S_n
\end{align}
is an ordinary permutation.
\end{lemma}
\begin{proof}
By Definition \ref{fgamma}, it is easy to check that $i \le f_{\gamma}(i)\le i+n$ for any $1\le i\le n$.
For $1\leq i\leq p+m$, by Definition \ref{fgamma}  and \eqref{eq:Lemmafgammaisw-d}, we have  
$$f_\gamma(i)\in[q_1+1,p]\cup\{\text{indices of $\clan{,}$ and right-ends}\}.$$
For $p+m+1\leq i\leq n$, by Definition \ref{fgamma} (3), we have 
$$f_\gamma(i)\in[n+1,n+q_1]\cup \{n+\text{indices of $\clan{-}$'s}\}.$$
Thus $f_{\gamma}(i)-f_{\gamma}(j)\not\equiv 0 \mod n$, for any $1\le i<j\le n$.
Also note that, only the last $q$ values of $f_\gamma$ can exceed $n$, so $g_{\gamma}$ belongs to $ S_{n}$.
\end{proof}

Let us denote the permutation matrix 
$$u_0 = \left[
    \begin{array}{c|c}
    \begin{array}{@{}c@{\,\,}c@{\,\,}c@{}}
    && 1\\[-0.75ex]
    &\reflectbox{$\ddots$}\\[-0.75ex] 1
\end{array}\\\hline& 1_{m+q}
\end{array}\right]\in GL_n,$$
corresponding to the permutation $p(p-1)\cdots 1(p+1)\cdots n\in S_n$. 

\begin{theorem} \label{thm:Bbarandpositroid}
Suppose that $\gamma$ is a $(p,m,q)$-preclan 
with $v_\gamma=e$.  
Under the identification $Q\backslash G\cong Gr_q(\mathbb{C}^n)$, we have 
$$\overline{Q\backslash Q\dot{\gamma}\bar{B}} \cdot u_0 = \Pi_{f_\gamma}.$$
\end{theorem}

\begin{remark}
More generally, one can show that after twisting by $u_0$, every $\bar{B}$-orbit closure over $Gr_q(\mathbb{C}^n)$  is a positroid variety. 
In particular, the $\bar{B}$-orbits over $Gr_q(\mathbb{C}^n)$ are finite. 
The above case corresponds to the case when the orbit is inside $Q\backslash Q\cdot P$.  
\end{remark}

The remainder of this section is devoted to the proof of Theorem \ref{thm:Bbarandpositroid}. 
Note that both sides are irreducible, so it suffices to show the inclusion ``$\subseteq$'' and compare the dimensions. 
That is, the proof has two steps. 
\begin{itemize}
    \item[(1)] 
The first step is to show the inclusion ``$\subseteq$'', i.e., 
$$\text{a generic element of }
Q\backslash Q\dot{\gamma}\bar{B}
\cdot u_0 \text{ is contained in } \Pi_{f_\gamma}.$$
    \item[(2)] 
The second step is to compare the dimensions
$$\dim \Pi_{f_\gamma}\leq \dim Q\backslash Q\dot{\gamma}\bar{B}.$$
\end{itemize}

Recall that 
$$Q\backslash G
\stackrel{\sim}\longrightarrow
\operatorname{Gr}_q(\mathbb{C}^n),\qquad Qg\longmapsto \operatorname{span}(\text{last $q$ rows of $g$}).$$
Assume that the last $q$ rows of a generic element of $\dot{\gamma}\bar{B}\cdot u_0$ is
\begin{equation}\label{eq:genericqrows}
\left[\begin{array}{ccc|cccc}
\mid & \cdots & \mid &
\mid & \cdots & \cdots & \mid\\
\vv_1 & \cdots & \vv_p & \vv_{p+1} & \cdots & \cdots & \vv_n\\
\mid & \cdots & \mid &
\mid & \cdots & \cdots & \mid\\
\end{array}\right].    
\end{equation}
Let us describe \eqref{eq:genericqrows} explicitly as follows. Recall that $\sigma_{\gamma}$ is obtained by assigning $1,\ldots,p$ to $\clan{+}$'s and left-ends of $\gamma$ from left to right, then  $p+1,\ldots,p+m$ to uncolored $\clan{,}$'s of $\gamma$ from left to right, and finally $p+m+1,\ldots,p+m+q$ to $\clan{-}$'s and right-ends of $\gamma$.
By Definition \ref{def:dotgamma}, for $1\leq k\leq q$, the $(p+m+k)$-th row of $\dot{\gamma}$ is
$$
\def\O{\color{lightgray}0}
\def\I{1}
\begin{aligned}\vphantom{\dfrac12}
\left[\begin{array}{ccccc|ccccccc}
\cdots &\O&
\raisebox{0pc}[0pc][0pc]{$\stackrel{i}{\vphantom{\frac12}\I}$}
& \O&\cdots & 
\cdots &\cdots &\O&
\raisebox{0pc}[0pc][0pc]{$\stackrel{j}{\vphantom{\frac12}\I}$}
& \O&\cdots 
\end{array}\right] & \quad 
\begin{matrix}
\text{if there is an $(i,j)$-matching,}\\ 
\text{where $j=\sigma_\gamma^{-1}(p+m+k)$};\hfill
\end{matrix}
\\\vphantom{\dfrac12}
\left[\begin{array}{ccccc|ccccccc}
\cdots &\O&\O
& \O&\cdots & 
\cdots  &\O&
\raisebox{0pc}[0pc][0pc]{$\stackrel{i}{\vphantom{\frac12}\I}$}
& \O&\cdots&\cdots 
\end{array}\right] & \quad 
\begin{matrix}
\text{if the $i$-th node is  $\clan{-}$,}\\ 
\text{where $i=\sigma_\gamma^{-1}(p+m+k)$}. \hfill  
\end{matrix}
\end{aligned}$$
Then for $1\leq k\leq q$, the $(p+m+k)$-th row of a generic element of $\dot{\gamma}\bar{B}\cdot u_0$ is  
$$
\def\O{\color{lightgray}0}
\def\I{*}
\def\X{*}
\begin{aligned}\vphantom{\dfrac12}
\left[\begin{array}{ccccc|ccccccc}
\cdots &\X&
\raisebox{0pc}[0pc][0pc]{$\stackrel{\!\!\!\!\!\!\!\!\!p+1-i\!\!\!\!\!\!\!\!}{\vphantom{\frac12}\I}$}
& \O&\cdots & 
\cdots &\cdots &\O&
\raisebox{0pc}[0pc][0pc]{$\stackrel{j}{\vphantom{\frac12}\I}$}
& * &\cdots 
\end{array}\right] & \quad 
\begin{matrix}
\text{if there is an $(i,j)$-matching,}\\ 
\text{where $j=\sigma_\gamma^{-1}(p+m+k)$};\hfill
\end{matrix}
\\\vphantom{\dfrac12}
\left[\begin{array}{ccccc|ccccccc}
\cdots &\O&\O
& \O&\cdots & 
\cdots  &\O&
\raisebox{0pc}[0pc][0pc]{$\stackrel{i}{\vphantom{\frac12}\I}$}
& \X &\cdots&\cdots 
\end{array}\right] & \quad 
\begin{matrix}
\text{if the $i$-th node is  $\clan{-}$,}\\ 
\text{where $i=\sigma_\gamma^{-1}(p+m+k)$}.\hfill
\end{matrix}
\end{aligned}$$
Here all $*$'s are independent. 
Moreover, the lower right $q\times (m+q)$ block of $\dot{\gamma}\bar{B}\cdot u_0$ is a full rank echelon form. 
For example, for the $(4,3,4)$-preclan in \eqref{eq:Egvgamma=e}, we have 
\begin{equation*}
\def\m#1{\makebox[1.2pc]{$#1$}}
\begin{aligned}
\gamma =&\,  \clan{586+,.-,..,}\\
\sigma_\gamma = & \,\m{1}\m{2}\m{3}\m{4}\m{5}\m{8}\m{9}\m{6}\m{10}\m{11}\m{7}.\\
\end{aligned}
\end{equation*}
The last $q=4$ rows of $\dot{\gamma}$ are
$$\def\O{{\color{lightgray}0}}
\def\I{1}
\def\ind#1#2{\raisebox{0pc}[0pc][0pc]{%
    $\stackrel{#1}{\vphantom{\frac12}#2}$\vphantom{#2}}}
\begin{matrix}\\[-1ex]
\left[\begin{array}{cccc|ccccccc}
\ind{1}{\I}&\ind{2}{\O}&\ind{3}{\O}&\ind{4}{\O}&\ind{5}{\O}&\ind{6}{\I}&\ind{7}{\O}&\ind{8}{\O}&\ind{9}{\O}&\ind{10}{\O}&\ind{11}{\O}\\
\O&\O&\O&\O&\O&\O&\I&\O&\O&\O&\O\\
\O&\O&\I&\O&\O&\O&\O&\O&\I&\O&\O\\
\O&\I&\O&\O&\O&\O&\O&\O&\O&\I&\O\\
\end{array}\right].
\end{matrix}$$
Then the last $q$ rows of $\dot{\gamma}\bar{B}$ are
$$
\def\O{{\color{lightgray}0}}
\def\I{1}\def\X{*}
\def\ind#1#2{\raisebox{0pc}[0pc][0pc]{%
    $\stackrel{#1}{\vphantom{\frac12}#2}$\vphantom{#2}}}
\begin{matrix}\\[-1ex]
\left[\begin{array}{cccc|ccccccc}
\ind{1}{\X}&\ind{2}\X&\ind{3}{\X}&\ind{4}{\X}&\ind{5}{\O}&\ind{6}{\X}&\ind{7}{\X}&\ind{8}{\X}&\ind{9}{\X}&\ind{10}{\X}&\ind{11}{\X}\\
\O&\O&\O&\O&\O&\O&\X&\X&\X&\X&\X\\
\O&\O&\X&\X&\O&\O&\O&\O&\X&\X&\X\\
\O&\X&\X&\X&\O&\O&\O&\O&\O&\X&\X\\
\end{array}\right].
\end{matrix}$$
Thus the last $q$ rows of a generic element of $\dot{\gamma}\bar{B}\cdot u_0$ are 
\begin{equation}\label{eq:runningPosEg}
\def\O{{\color{lightgray}0}}
\def\I{1}\def\X{*}
\def\ind#1#2{\raisebox{0pc}[0pc][0pc]{%
    $\stackrel{#1}{\vphantom{\frac12}#2}$\vphantom{#2}}}
\begin{matrix}\\[-1ex]
\left[\begin{array}{cccc|ccccccc}
\ind{1}{\X}&\ind{2}\X&\ind{3}{\X}&\ind{4}{\X}&\ind{5}{\O}&\ind{6}{\X}&\ind{7}{\X}&\ind{8}{\X}&\ind{9}{\X}&\ind{10}{\X}&\ind{11}{\X}\\
\O&\O&\O&\O&\O&\O&\X&\X&\X&\X&\X\\
\X&\X&\O&\O&\O&\O&\O&\O&\X&\X&\X\\
\X&\X&\X&\O&\O&\O&\O&\O&\O&\X&\X\\
\end{array}\right].
\end{matrix}
\end{equation}

\subsection{The inclusion}

Let us associate a bounded affine permutation $f$  to \eqref{eq:genericqrows} by 
$$f(i):=\min\{j:\vv_i\in \operatorname{span}(\vv_{i+1},\ldots,\vv_j)\}.$$
We need to show $\Pi_f^\circ\subseteq \Pi_{f_\gamma}$, i.e., $f\geq f_\gamma$ under the affine Bruhat order.

For each column $\vv_i$ in \eqref{eq:genericqrows}, let
$Z(\vv_i)$ denote the set of nonzero row indices of $\vv_i$, and let $z(\vv_i)$ be the cardinality of $Z(\vv_i)$. For $1\le i\le p$, the $i$-th column of $\dot{\gamma}\bar{B}$ contains the maximal number of generic elements among the  first $i$ columns of $\dot{\gamma}\bar{B}$. For  $p+1\le i\le n$, the $i$-th column of $\dot{\gamma}\bar{B}$ contains the maximal number of generic elements among the $(p+1)$-st to $i$-th columns of $\dot{\gamma}\bar{B}$. After multiplying $u_0$ on the right of $\dot{\gamma}\bar{B}$, we find that 
\begin{align}
&Z(\vv_1)\supseteq Z(\vv_2)\supseteq\cdots\supseteq Z(\vv_{p}),\label{eq:containvi}\\[1ex]
&z(\vv_{j})-z(\vv_{j+1})\in \{0,1\},\  \text{ for any } 1\le j\le p-1,\label{eq:differ1}
\end{align}
and
\begin{align}
& Z(\vv_{p+1})\subseteq Z(\vv_{p+2})\subseteq\cdots\subseteq Z(\vv_{n}),\label{eq:containvii} \\[1ex]
&z(\vv_{j+1})-z(\vv_{j})\in \{0,1\},\  \text{ for any } p+1\le j\le n-1. \label{eq:differ2}
\end{align}
Moreover, $[\vv_{p+1}\,\ldots\,\vv_n]$ is an echelon form.

\paragraph{\bf Case 1}
The $(p+1-i)$-th node is $\clan{+}$ for some  $1\leq i\leq p$. 
By definition, $f_\gamma(i)=i+k$, where $k=\texttt{\#}\{\text{matchings covering this $\clan{+}$}\}$. 
 Note that $k$ is the number of $1$'s before the $(p+1-i)$-th column in the last $q$ rows of $\dot{\gamma}$.
Thus the $(p+1-i)$-th column of $\dot{\gamma}\bar{B}$ contains $k$ generic elements, i.e., $k=z(\vv_i)$.  By \eqref{eq:Lemmafgammaisw-d}, we have $i+k\leq p$. 
Therefore,  by \eqref{eq:containvi}, the $k+1$ vectors $\vv_i,\ldots,\vv_{i+k}$ are linearly dependent. 
Moreover, up to column operations and a permutation of rows, $[\vv_i\,\,\cdots\,\,\vv_{i+k-1}]$ contains a $k\times k$ upper triangular matrix.
Thus the $k$ vectors $\vv_i,\ldots,\vv_{i+k-1}$ are generically linearly independent. 
This implies $f(i)=i+k$. 

For example, in the above running example \eqref{eq:runningPosEg}, we have $f(1)=4$ since
$$\def\O{\color{lightgray}0}\def\X{*}
\def\ind#1#2{\raisebox{0pc}[0pc][0pc]{%
    $\stackrel{#1}{\vphantom{\frac12}#2}$\vphantom{#2}}}
\begin{matrix}\\[-1ex]
\left[\begin{array}{c}
\ind{1}{\X}\\\O\\\X\\\,\X\,
\end{array}\right]
\notin
\operatorname{span}
\left[\begin{array}{cc}
\ind{2}{\X}&\ind{3}\X\\
\O&\O\\
\X&\O\\
\X&\X\\
\end{array}\right],\qquad 
\left[\begin{array}{c}
\ind{1}{\X}\\\O\\\X\\\X
\end{array}\right]
\in
\operatorname{span}
\left[\begin{array}{ccc}
\ind{2}{\X}&\ind{3}{\X}&\ind{4}{\X}\\
\O&\O&\O\\
\X&\O&\O\\
\X&\X&\O\\
\end{array}\right].
\end{matrix}$$

\paragraph{\bf Case 2}
The $(p+1-i)$-th node is matched with the $j$-th node for some $1\leq i\leq p<j$. By definition, $f_\gamma(i)=j$. 
Then $\sigma_\gamma(j)$ is the row index of  $\dot{\gamma}$ with two $1$'s at the $(p+1-i)$-th and $j$-th column. Let $k=\sigma_\gamma(j)-p-m$.  
Since the $k$-th row of 
$\vv_{i+1},\ldots,\vv_{j-1}$ are all $0$'s, we have $f(i)\geq j$. 
We claim that 
$$\vv_i\in \operatorname{span}(\vv_{i+1},\ldots,\vv_{j}).$$
Since $z(\vv_p)\in \{0,1\}$ and $Z(\vv_{i})\setminus Z(\vv_{i+1})=\{k\}$, by \eqref{eq:containvi} and \eqref{eq:differ1}, it is easy to see  
$$\mathbf{e}_h \in \operatorname{span}(\vv_{i+1},\ldots,\vv_p), \text{ for any } h\in Z(\vv_{i})\ \text{ and }\ h\neq k . $$
Since $(p+1-i,j)$ is a matching, by the definition of $\dot{\gamma}$, we have $k\in Z(\vv_j)$. Then by \eqref{eq:containvii} and \eqref{eq:differ2},  
$$\mathbf{e}_{k}\in \operatorname{span}(\vv_{p+1},\dots,\vv_j).$$
This proves $f(i)=j$. 

For example, in the above running example \eqref{eq:runningPosEg}, we have $f(2)=9$ since
$$\def\O{{\color{lightgray}0}}
\def\I{1}\def\X{*}
\def\ind#1#2{\raisebox{0pc}[0pc][0pc]{%
    $\stackrel{#1}{\vphantom{\frac12}#2}$\vphantom{#2}}}
\begin{matrix}\\[-1ex]
\left[\begin{array}{c}
\ind{2}{\X}\\\O\\\X\\\X
\end{array}\right]
\notin
\operatorname{span}
\left[\begin{array}{cc|ccccc}
\ind{3}{\X}&\ind{4}{\X}&\ind{5}{\O}&\ind{6}{\X}&\ind{7}{\X}&\ind{8}{\X}\\
\O&\O&\O&\O&\X&\X\\
\O&\O&\O&\O&\O&\O\\
\X&\O&\O&\O&\O&\O\\
\end{array}\right],\qquad 
\left[\begin{array}{c}
\ind{2}{\X}\\\O\\\X\\\X
\end{array}\right]
\in
\operatorname{span}
\left[\begin{array}{cc|ccccc}
\ind{3}{\X}&\ind{4}{\X}&\ind{5}{\O}&\ind{6}{\X}&\ind{7}{\X}&\ind{8}{\X}&\ind{9}{\X}\\
\O&\O&\O&\O&\X&\X&\X\\
\O&\O&\O&\O&\O&\O&\X\\
\X&\O&\O&\O&\O&\O&\O\\
\end{array}\right].
\end{matrix}$$

\paragraph{\bf Case 3}
Let $p+1\leq i\leq p+m$. 
If $\vv_i=0$, then $i=j$ and obviously $f(i)=i$. 
Let us assume $\vv_i\neq 0$. 
By definition, $f_\gamma(i)=j$ where $j$ is the index for the $(i-p)$-th uncolored node $\clan{,}$. 
Since $j-i$ is the number of right-ends and $\clan{-}$'s before this uncolored node $\clan{,}$, each corresponds to a $1$ in the last $q$ rows of $\dot{\gamma}$, so $Z(\vv_j)=[1,j-i]$.

Note that $Z(\vv_{p+1}),\ldots,Z(\vv_j)$ are all contained in $[1,j-i]$. 
Therefore, the $j-i+1$ vectors $\vv_i,\ldots,\vv_j$ are linearly dependent. 
By \eqref{eq:differ2}, up to column operations, $[\vv_i\,\,\cdots\,\,\vv_j]$ contains a $(j-i)\times (j-i)$ upper triangular matrix. 
Thus the $j-i$ vectors $\vv_i,\ldots,\vv_{j-1}$ are generically linearly independent. 
This proves $f(i)=j$. 

For example, in the above running example \eqref{eq:runningPosEg}, we have $f(6)=8$ since
$$\def\O{\color{lightgray}0}\def\X{*}
\def\ind#1#2{\raisebox{0pc}[0pc][0pc]{%
    $\stackrel{#1}{\vphantom{\frac12}#2}$\vphantom{#2}}}
\begin{matrix}\\[-1ex]
\left[\begin{array}{c}
\ind{6}{\X}\\\O\\\O\\\O
\end{array}\right]
\notin
\operatorname{span}
\left[\begin{array}{cc}
\ind{7}{\X}\\
\X\\
\O\\
\O\\
\end{array}\right],\qquad 
\left[\begin{array}{c}
\ind{6}{\X}\\\O\\\O\\\O
\end{array}\right]
\in
\operatorname{span}
\left[\begin{array}{cc}
\ind{7}{\X}&\ind{8}{\X}\\
\X&\X\\
\O&\O\\
\O&\O\\
\end{array}\right].
\end{matrix}$$

\paragraph{\bf Case 4} We have shown that for $1\leq i\leq n-q$, we have $f(i)=f_\gamma(i)\leq n.$
By Definition \ref{fgamma} of $f_\gamma$, we have
\begin{align*}
\{f(n-q+1),\ldots,f(n)\}
& = \{n+k:k\in [n]\setminus \{f(1),\ldots,f(n-q)\}\}\\
& = \{n+k:k\in [n]\setminus \{f_\gamma(1),\ldots,f_\gamma(n-q)\}\}\\
& =[n+1,n+q_1]\cup \{n+\text{indices of $\clan{-}$'s}\}\\
& = \{f_\gamma(n-q+1)<\cdots<f_\gamma(n)\}. 
\end{align*}
In particular, we have a reduced decomposition $f=f_\gamma\cdot v_0$ for $v_0\in S_n$ permuting only elements in $\{n-q+1,\ldots,n\}$.
As a result, $f\geq f_\gamma$ under the affine Bruhat order. 

\begin{remark}
We actually have $f=f_\gamma$, which follows from the dimension comparison in the next subsection. 
\end{remark}

\subsection{Dimension comparison}
Assume $S\dot{\gamma}B/B\cap N\cdot B/B$ is nonempty.
We need to show 
$$\dim \Pi_{f_\gamma} \leq \dim Q\backslash Q\dot{\gamma}\bar{B}\cdot u_0 
= \dim Q\backslash Q\dot{\gamma}\bar{B}.$$
By transversality in the proof of Lemma \ref{lemma:diag2}, 
Lemma \ref{lemma:diag3} and dimension formula  Theorem \ref{thm:dim=length+}, we have
$$
\begin{aligned}
\operatorname{codim} Q\backslash Q\dot{\gamma}\bar{B} & = 
\operatorname{codim} S\dot{\gamma}B/B
= 
\binom{n}{2}-\binom{p}{2}-\binom{q}{2}-\binom{m}{2}-\ell(\gamma).
\end{aligned}$$
Note that 
$$\operatorname{codim} \Pi_{f_\gamma} = \ell(f_\gamma).$$
Thus it is enough to show the following  statement. 
For any $(p,m,q)$-preclan $\gamma$ with $v_\gamma=e$, we have
\begin{equation}\label{eq:combStep2}
\ell(f_\gamma)\geq \binom{n}{2}-\binom{p}{2}-\binom{q}{2}-\binom{m}{2}-\ell(\gamma). 
\end{equation}

Let us first consider the case $\gamma=\gamma_0$ is the $(p,m,q)$-preclan of maximal length, i.e. 
$$\gamma_0 = \text{a rainbow $(p,q)$-clan followed by $m$ uncolored $\clan{,}$'s}. $$
It is easy to check that $f_{\gamma_0}(i)=i+q$, and both sides of \eqref{eq:combStep2} are $0$. 

Note that if $v_\gamma =e$, then $v_{s_k*\gamma}=e$ for any $1\leq k\leq n-1$. 
By Lemma \ref{lem:preclanisconn}, it suffices to show 
$$
s_k*\gamma> \gamma\Longrightarrow 
\ell(f_{s_k*\gamma})<\ell(f_{\gamma}). 
$$
This follows by examining one-by-one the cases  in  
\eqref{eq:clanweak1}, 
\eqref{eq:clanweak2}, 
\eqref{eq:clanweak3} such that $v_\gamma=e$. There are 7 cases.  
For  simplicity,  denote 
$f=f_\gamma$ and $f'=f_{s_k\ast\gamma}$. 

Case 1. 
$\def\m#1{\makebox[1.2pc]{$#1$}}
\begin{array}{lll}
     \clan{\dots+-\dots}&\mapsto &\clan{\dots1.\dots}\\
    \m{}\m{k} & &\m{}\m{k}
\end{array}.$
Since $v_{\gamma}=e$, the two nodes $\clan{+}\clan{-}$ are labeled $p$ and $p+1$, and hence $k=p$ in this case.
Recall that $q_1$ denotes the number of matchings of $\gamma$. Then $\gamma'$ has $q_1+1$ matchings. By the definitions of $f$ and $f'$, node $p$ gives $f(1)=1+q_1$ and $f'(1)=p+1$.   Since the $(p+1)$-st node of $\gamma$ is the first $\clan{-}$, we have $f(b)=n+p+1$, where $b=p+m+q_1+1$. On the other hand, $f'(b)$ is determined by the last right endpoint of $\gamma'$, thus $f'(b)=n+q_1+1\le n+p$. Therefore, $f(i)=f'(i)$ except $i\in \{1,b\}\bmod n$, more precisely, $f'$ is obtained from $f$ by swapping $f(1)=q_1+1$ and $f(b-n)=p+1$.
Thus $\ell(f)>\ell(f')$.

Case 2. $\def\m#1{\makebox[1.2pc]{$#1$}}
\begin{array}{lll}
     \clan{\dots+2\dots.}&\mapsto&
\clan{\dots3+\dots.}  \\
  \m{}\m{k}\m{}\m{}\m{j}   & & \m{}\m{k}\m{}\m{}\m{j}
\end{array}.$
In this case $k<p$, and assume that $(k+1,j)$ is a matching of $\gamma$. Let $k'=p+1-k$, $c=\#\{\text{matchings covering the $k$-th node of $\gamma$}\}$.
So by definition of $f_{\gamma}$, we have 
$f(k'-1)=j, f(k')=k'+c.$
After swapping nodes $k$ and $k+1$, we have
$f'(k'-1)=k'-1+c+1=k'+c, f'(k')=j.$
That is, we have $f'=fs_{k'-1}$ with $f(k')\leq p<f(k'-1)$. 

Case 3. $\def\m#1{\makebox[1.2pc]{$#1$}}
\begin{array}{lll}
     \clan{2\dots.-\dots}&\mapsto&\clan{3\dots-.\dots}\\
    \m{}\m{}\m{k} && \m{}\m{}\m{k}
\end{array}.$
In this case $k>p$, let $i=f^{-1}(k), j=f^{-1}(n+k+1)$, where $1\le i,j\le n$.
The swapping of nodes $k$ and $k+1$ gives 
$f'(i)=k+1$ and $f'(j)=n+k$,
thus we have $f'=s_kf$ with $f^{-1}(k+1)\leq 0<f^{-1}(k)$. 

Case 4. $\def\m#1{\makebox[1.2pc]{$#1$}}\begin{array}{lll}
\clan{34\dots.\dots.}&\mapsto&\clan{52\dots.\dots.}  \\
     \m{k}& &\m{k} 
\end{array}.$
In this case $k<p$, let $k'=p+1-k$. We have $f'=fs_{k'-1}$ with $f(k'-1)>f(k')$.

Case 5. $\def\m#1{\makebox[1.2pc]{$#1$}}
\begin{array}{lll}
\clan{4\dots3\dots..}&\mapsto&\clan{5\dots2\dots..}  \\
  \m{}\m{}\m{}\m{}\m{k} & &\m{}\m{}\m{}\m{}\m{k}    
\end{array}.
$
In this case $k>p$, we have $f'=s_{k}f$ with $f^{-1}(k+1)<f^{-1}(k)$.

Case 6. $\def\m#1{\makebox[1.2pc]{$#1$}}
\begin{array}{lll}
   \clan{\dots,-\dots}&\mapsto&\clan{\dots-,\dots} \\
     \m{}\m{k} &&\m{}\m{k}
\end{array}.
$
In this case $k>p$, we have $f'=s_{k}f$ with $f^{-1}(k+1)\leq 0<f^{-1}(k)$.

Case 7. $\def\m#1{\makebox[1.2pc]{$#1$}}
\begin{array}{lll}
\clan{3\dots,.}&
\mapsto&
\clan{2\dots.,}\\
   \m{}\m{}\m{k} & &\m{}\m{}\m{k}
\end{array}.
$
In this case $k>p$, we have $f'=s_{k}f$ with $f^{-1}(k+1)\leq p<f^{-1}(k)$.

\section{Schubert Expansion}\label{Schubert Expansion}

In this section, we complete the proof of the theorems in the Introduction. 

By \cite{KLS,LLS18}, the equivariant fundamental class of positroid varieties is represented by double affine Stanley symmetric polynomials. 
Thus we have 
$$[\Pi_f]_T|_{Q\cdot 1} = 
\widetilde{F}_f|_{x_1=t_{n-q+1},\ldots,x_q=t_n}
=\widetilde{F}_{f}(t_{n-q+1},\ldots,t_n;t_1,t_2,\ldots,t_n).$$
Combining Proposition \ref{prop:eqSchubertexp}, 
Theorem \ref{thm:globalization} and Theorem \ref{thm:Bbarandpositroid}, we get the following theorem. 

\begin{theorem}\label{thm:Sorbitexpansion}
For any $(p,m,q)$-preclan $\gamma$, assume that 
$$\Upsilon_\gamma(x;t) = \sum_{w\in S_n}
d_{w,\gamma}(t)\cdot \mathfrak{S}_w(x;t).$$
Then we have
\begin{equation*}\label{eq:affcoeff}
d_{w,\gamma}(t)=
\begin{cases}
\widetilde{F}_{f_{w*\gamma}}(-t_{n-q+1},\ldots,-t_n;-t_p,\ldots,-t_1,-t_{p+1},\ldots,-t_n),\\\qquad  \text{if } \ell(w*\gamma)=\ell(w)+\ell(\gamma) \text{ and } v_{w*\gamma}=e,\\[1ex]
0, \quad \text{otherwise}.
\end{cases}
\end{equation*}
\end{theorem}

Combining Corollary \ref{coro:RichUpsilon} and Theorem \ref{prop:RichApp}, we are able to compute the double Schubert expansion of
$
\mathfrak{S}_v(x;t)
\mathfrak{S}_u(x;\cev{t})$
for $v,u$ satisfying any of the conditions of Theorem \ref{thm:equivuandv}. 
However, we need the expansion of 
$\mathfrak{S}_v(x;t)
\mathfrak{S}_u(x;t)$, so there  are still some work  to do. 

\subsection{Triple Schubert Structure Constants}
Now we are able to prove our main theorem. 

\begin{theorem}\label{th:triplecoeffi}
For a Richardson $(p,m,q)$-preclan $\gamma$, we have 
$$
\mathfrak{S}_{v_\gamma}(x;t)
\mathfrak{S}_{u_\gamma}(x;y)
=\sum_{w\in S_n} c_{v_\gamma,u_\gamma}^w(t;y)\cdot \mathfrak{S}_w(x;t),$$
where
$$c_{v_\gamma,u_\gamma}^w(t;y)= 
\begin{cases}
\widetilde{F}_{f_{w*\gamma}}(-y_q,\ldots,-y_1;-t_p,\ldots,-t_1,-t_{p+1},\ldots,-t_n), \\\qquad \text{if } \ell(w*\gamma)=\ell(w)+\ell(\gamma) \ \text{ and } v_{w*\gamma}=e, \\[1ex]
0, \quad \text{otherwise}. 
\end{cases}$$
\end{theorem}
\begin{proof}
The main trick  is to view the $y$-variables as a part of $t$-variables. 
Let us pick $N\gg0$.  Denote
$$T=(t_1,\ldots,t_{N}) = (t_1,t_2,\ldots,t_{n},y_{N-n},\ldots,y_2,y_1),$$
that is,  view $y_1,\ldots,y_{N-n}$ as $t_{N},\ldots,t_{n+1}$. 
Let $\gamma'$ be the $(p,M,q)$-preclan $(M=N-n+m)$ obtained by adding $N-n$ many uncolored nodes $\clan{,}$'s at the end of $\gamma$. 
Now, by Corollary \ref{coro:RichUpsilon} and Theorem \ref{thm:Sorbitexpansion}, we have 
$$
\begin{aligned}
\mathfrak{S}_{v_\gamma}(x;t)\cdot \mathfrak{S}_{u_\gamma}(x;y) 
& 
= \mathfrak{S}_{v_\gamma}(x;T)\cdot \mathfrak{S}_{u_\gamma}(x;\cev{T}) 
= \Upsilon_{\gamma'}(x;T)= \sum_{w\in S_N}d_{w,\gamma'}(T)\cdot \mathfrak{S}_w(x;T). 
\end{aligned}$$
Similar to the proof of Corollary \ref{coro:RichUpsilon0}, the sum actually only involves $w\in S_n$. 
In particular, $\mathfrak{S}_w(x;T)$ only involves $t_1,\ldots,t_n$ and thus it equals $\mathfrak{S}_w(x;t)$. 
Since we view $t_{N-q+1}=y_{q},\ldots,t_N=y_{1}$, by Theorem \ref{thm:Sorbitexpansion}, we have 
\begin{equation}\label{eq:affs}
d_{w,\gamma'}(T) = \begin{cases}
\widetilde{F}_{f_{w*\gamma'}}(-y_q,\ldots,-y_1;-t_p,\ldots,-t_1,-t_{p+1},\ldots,-t_n,-y_{N-n},\ldots,-y_1),\\
\qquad\quad \text{if } \ell(w*\gamma)=\ell(w)+\ell(\gamma)\text{ and } v_{w*\gamma}=e, \\[1ex]
0, \qquad \text{otherwise}. 
\end{cases}    
\end{equation}

It remains to show the double affine Stanley symmetric polynomial
$\widetilde{F}_{f_{w*\gamma'}}$ in \eqref{eq:affs}  only involves $t_1,\ldots,t_n$. 
As pointed out in the proof of Corollary \ref{coro:RichUpsilon0}, adding uncolored nodes $\clan{,}$'s commutes with $\gamma\mapsto s_k*\gamma$ for $1\leq k\leq n-1$. 
Now the assertion follows from the next lemma. 
\end{proof}

\begin{lemma}\label{lemma7.3}
Let $\gamma$ be a $(p,m,q)$-preclan, and $\gamma'$ be the $(p,m+1,q)$-preclan obtained by adding an uncolored node $\clan{,}$ at the end of $\gamma$.
Then $\widetilde{F}_{f_{\gamma'}}(x;t) = \widetilde{F}_{f_\gamma}(x;t)$, which does not involve $t_{n+1}$. 
\end{lemma}

\begin{proof}
By Definition \ref{fgamma}, we have 
$$f_{\gamma'}(i)=\begin{cases}
f_\gamma(i), & 1\leq i\leq p+m,\\
n+1, & i= p+m+1,\\
f_\gamma(i-1)+1, & p+m+1<i\leq n+1.
\end{cases}$$
The lemma follows from the following  weight-preserving bijection
$$
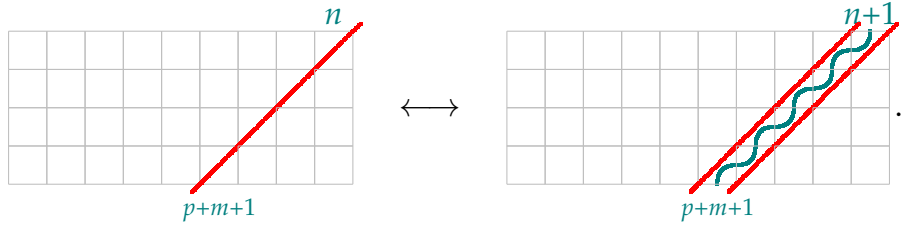
\def\diag{%
\begin{picture}(0,1)%
   \color{red}
    \linethickness{0.08\unitlength}
    \qbezier(0.2,1.2)(-1,0)(-1.2,-.2)
\end{picture}
}
\BPD{
\M{}\M{}\M{}\M{}\M{}\M{}\M{}\M{}\M{n}\\
\O\O\O\O\O\O\O\O\O\diag\\
\O\O\O\O\O\O\O\O\diag\O\\
\O\O\O\O\O\O\O\diag\O\O\\
\O\O\O\O\O\O\diag\O\O\O\\
\M{}\M{}\M{}\M{}\M{}\M{\scriptstyle p\text{+}m\text{+}1}}
\quad \longleftrightarrow \quad
\BPD{
\M{}\M{}\M{}\M{}\M{}\M{}\M{}\M{}\M{}\M{n\text{+}1}\\
\O\O\O\O\O\O\O\O\F\diag\J\diag\\
\O\O\O\O\O\O\O\F\diag\J\diag\O\\
\O\O\O\O\O\O\F\diag\J\diag\O\O\\
\O\O\O\O\O\F\diag\J\diag\O\O\O\\
\M{}\M{}\M{}\M{}\M{}\M{\scriptstyle p\text{+}m\text{+}1}}.
$$
Since $n<f_\gamma(n-q+1)<\cdots<f_\gamma(n)$ and $n+1=f'_\gamma(n-q+1)<\cdots<f'_\gamma(n+1)$, the tiles on and under the red lines are all $\BPD{\B}$. 
\end{proof} 

\begin{remark}
We have the following remarks. 
\begin{itemize}
    \item[(1)] From the proof of Theorem \ref{th:triplecoeffi}, the argument actually implies the existence of a triple version $\Upsilon_\gamma(x;t;y)$. They are characterized by the following two properties analogous to Corollary \ref{coro:RichUpsilon} and Theorem \ref{thm:dkonUpsilon}:
    $$
    \begin{aligned}
    \Upsilon_\gamma(x;t;y) & = 
    \mathfrak{S}_{v_\gamma}(x;t)\cdot
    \mathfrak{S}_{u_\gamma}(x;y)
    \qquad \text{when $\gamma$ is Richardson}\\
    \partial_k 
    \Upsilon_\gamma(x;t;y)
    & =\begin{cases}
    \Upsilon_{s_k*\gamma}(x;t;y), & \ell(s_k*\gamma)=\ell(\gamma)+1,\\
    0,& \text{otherwise}. 
    \end{cases}
    \end{aligned}
    $$
    Moreover if we switch $\clan{+}$ and $\clan{-}$, then we will need to switch $t$ and $y$. We leave the details to readers. Here we compute $\Upsilon_\gamma(x;t;y)$ for all $(1,1,1)$-preclans, to compare with Example \ref{eg:Upsilon}. 

    \def\term#1#2#3#4{
#1 \qquad 
\makebox[2pc]{$#2$}\qquad 
\makebox[2pc]{$#3$}\qquad 
\makebox[0.4\linewidth][l]{$#4$}\qquad 
}
$$\term{\text{preclan $\gamma$}}{v_\gamma}{u_\gamma}{\Upsilon_\gamma(x;t;y)}$$
$$\term{\clan{,+-}}{213}{231}{(x_1-t_1)\cdot (x_1-y_1)(x_2-y_1)}$$
$$\term{\clan{,-+}}{231}{213}{(x_1-t_1)(x_2-t_1)\cdot (x_1-y_1)}$$
$$\term{\clan{,1.}}{213}{213}{(x_1-t_1)\cdot (x_1-y_1)}$$
$$\term{\clan{+,-}}{123}{231}{(x_1-y_1)(x_2-y_1)}$$
$$\term{\clan{-,+}}{231}{123}{(x_1-t_1)(x_2-t_1)}$$
$$\term{\clan{+-,}}{123}{213}{(x_1-y_1)}$$
$$\term{\clan{-+,}}{213}{123}{(x_1-t_1)}$$
$$\term{\clan{2,.}}{123}{123}{x_1+x_2-t_1-y_1}$$
$$\term{\clan{1.,}}{123}{123}{1}$$

    \item[(2)] It was conjectured by Samuel \cite{Sam24} and proved by Gao and the third author \cite{Gaox} that
    $$c_{v,u}^w(t;y)\in \mathbb{N}[t_i-y_j]_{i,j\geq 0}. $$
Since 
$\widetilde{F}_f(x;t)\in \mathbb{N}[x_i-t_j]$, 
Theorem \ref{th:triplecoeffi} verifies this positivity when $u,v$ are inverse Grassmannian.

Following \cite{FGX,Sam24}, this also implies a positive formula for skew divided difference operators of MacDonald \cite{MacDonald} acting on Schubert polynomials
$$\partial_{w/u_\gamma}\mathfrak{S}_{v_\gamma}(x) = c_{v_\gamma,u_\gamma}^w(x;0)
\in \mathbb{N}[x_i].$$
The positivity is known as Kirillov conjecture \cite{Kirillov} and proved by Gao and the third author \cite{Gaox}.

\end{itemize}

\end{remark}

\begin{example}\label{eg:eg:triplep<q}
Let $p=2,m=1,q=3$. Consider $v=314526$ and $u=412356$.
The corresponding preclan is 
$$\gamma_{v,u} = \clan{,4--+.}
\qquad
\BPD[2pc]{\M{}\M{-t_2}\M{-t_1}\M{-t_3}\M{-t_4}\M{-t_5}\M{-t_6}\\
\M{-y_3}\O\O\O\O\O\O\\\M{-y_2}\O\O\O\O\O\O\\\M{-y_1}\O\O\O\O\O\O\\\M{}\M{}\M{}\M{}\M{}\M{}
}.
$$
The Hasse diagram over $\gamma_{v,u}$ is shown in Figure \ref{fig:HasseEg} (left). 
We have 
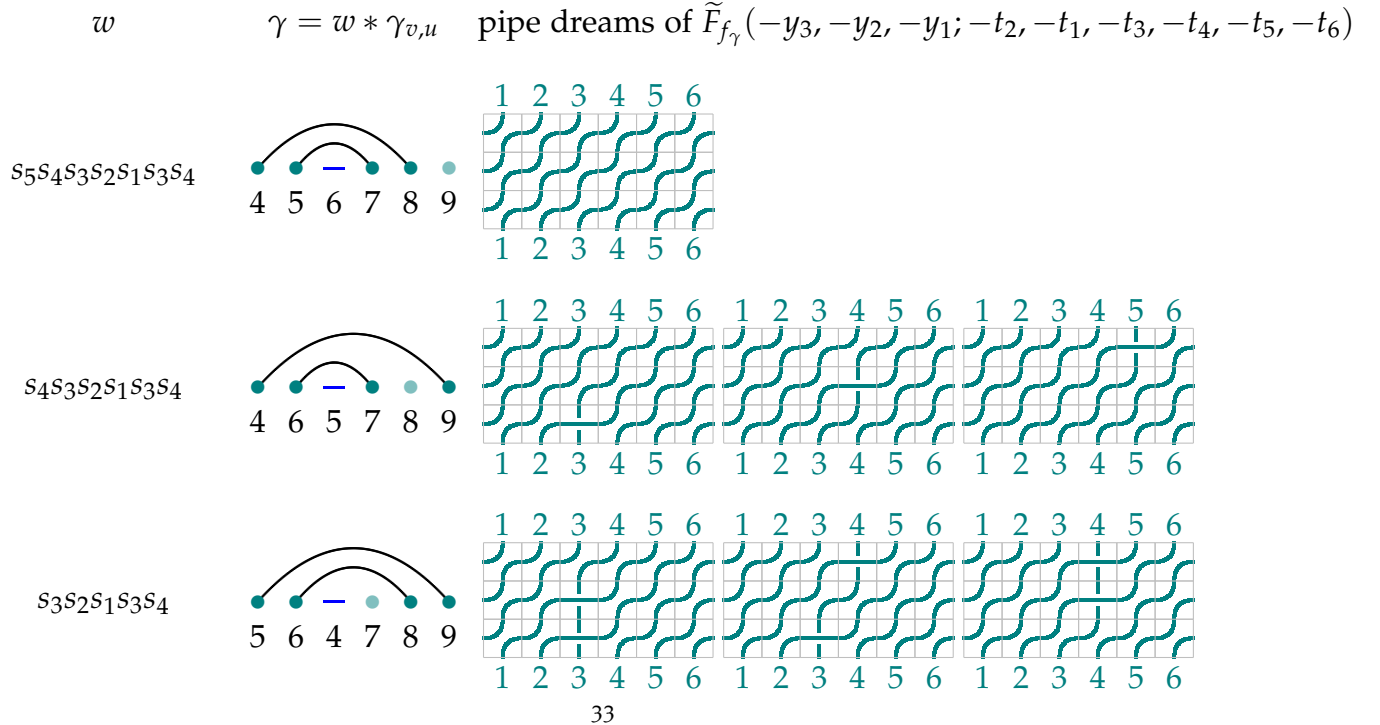
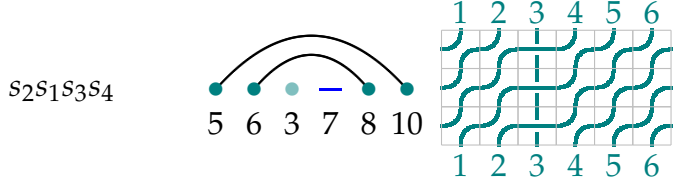
\def\term#1#2#3{
\makebox[0.2\linewidth]{$#1$}
\makebox[0.2\linewidth]{\def\m##1{\makebox[1.2pc]{$##1$}}%
$\begin{array}{l}#2\end{array}$}
\makebox[0.6\linewidth][l]{$#3$}
}
$$\term{w}{\gamma=w*\gamma_{v,u}}{\text{pipe dreams of $\widetilde{F}_{f_{\gamma}}(-y_3,-y_2,-y_1;-t_2,-t_1,-t_3,-t_4,-t_5,-t_6)$}}$$
$$\term{s_5s_4s_3s_2s_1s_3s_4}{\clan{42-..,}\\
\m{4}\m{5}\m{6}\m{7}\m{8}\m{9}}{
\BPD{\M{1}\M{2}\M{3}\M{4}\M{5}\M{6}\\
\B\B\B\B\B\B\\\B\B\B\B\B\B\\\B\B\B\B\B\B\\
\M{1}\M{2}\M{3}\M{4}\M{5}\M{6}}
}$$
$$\term{s_4s_3s_2s_1s_3s_4}{\clan{52-.,.}\\
\m{4}\m{6}\m{5}\m{7}\m{8}\m{9}}{
\BPD{\M{1}\M{2}\M{3}\M{4}\M{5}\M{6}\\\B\B\B\B\B\B\\\B\B\B\B\B\B\\\B\B\X\B\B\B\\\M{1}\M{2}\M{3}\M{4}\M{5}\M{6}}
\BPD{\M{1}\M{2}\M{3}\M{4}\M{5}\M{6}\\\B\B\B\B\B\B\\\B\B\B\X\B\B\\\B\B\B\B\B\B\\\M{1}\M{2}\M{3}\M{4}\M{5}\M{6}}
\BPD{\M{1}\M{2}\M{3}\M{4}\M{5}\M{6}\\\B\B\B\B\X\B\\\B\B\B\B\B\B\\\B\B\B\B\B\B\\\M{1}\M{2}\M{3}\M{4}\M{5}\M{6}}
}$$
$$\term{s_3s_2s_1s_3s_4}{\clan{53-,..}\\
\m{5}\m{6}\m{4}\m{7}\m{8}\m{9}}{
\BPD{\M{1}\M{2}\M{3}\M{4}\M{5}\M{6}\\\B\B\B\B\B\B\\\B\B\X\B\B\B\\\B\B\X\B\B\B\\\M{1}\M{2}\M{3}\M{4}\M{5}\M{6}}
\BPD{\M{1}\M{2}\M{3}\M{4}\M{5}\M{6}\\\B\B\B\X\B\B\\\B\B\B\B\B\B\\\B\B\X\B\B\B\\\M{1}\M{2}\M{3}\M{4}\M{5}\M{6}}
\BPD{\M{1}\M{2}\M{3}\M{4}\M{5}\M{6}\\\B\B\B\X\B\B\\\B\B\B\X\B\B\\\B\B\B\B\B\B\\\M{1}\M{2}\M{3}\M{4}\M{5}\M{6}}
}$$
$$\term{s_2s_1s_3s_4}{\clan{53,-..}\\
\m{5}\m{6}\m{3}\m{7}\m{8}\m{10}}{
\BPD{\M{1}\M{2}\M{3}\M{4}\M{5}\M{6}\\\B\B\X\B\B\B\\\B\B\X\B\B\B\\\B\B\X\B\B\B\\\M{1}\M{2}\M{3}\M{4}\M{5}\M{6}}
}$$
Thus 
\begin{align*}
\mathfrak{S}_{v}(x;t)
\mathfrak{S}_{u}(x;y)
& = \mathfrak{S}_{613425}(x;t)+[(t_3-y_1)+(t_4-y_2)+(t_5-y_3)]\mathfrak{S}_{51342}(x;t)\\
&\quad+[(t_3-y_1)(t_3-y_2)+(t_3-y_1)(t_4-y_3)+(t_4-y_2)(t_4-y_3)]\mathfrak{S}_{41352}(x;t)\\
&\quad+(t_3-y_1)(t_3-y_2)(t_3-y_3)\mathfrak{S}_{31452}(x;t).
\end{align*}
\end{example}

\begin{example}\label{eg:triplep>q}
Let $p=3,m=1,q=2$. Consider $v=412356$ and $u=314526$. The corresponding preclan $\gamma_{v,u}$ is 
$$\gamma_{v,u}=\clan{,4++-.}
\qquad
\BPD[2pc]{
\M{}\M{-t_3}\M{-t_2}\M{-t_1}\M{-t_4}\M{-t_5}\M{-t_6}\\
\M{-y_2}\O\O\O\O\O\O\\
\M{-y_1}\O\O\O\O\O\O\\
\M{}\M{}\M{}\M{}\M{}\M{}\M{}
}.
$$
The Hasse diagram over $\gamma_{v,u}$ is shown in Figure \ref{fig:HasseEg} (right), the same as the previous one but with the role of $\clan{+}$ and $\clan{-}$ switched. 
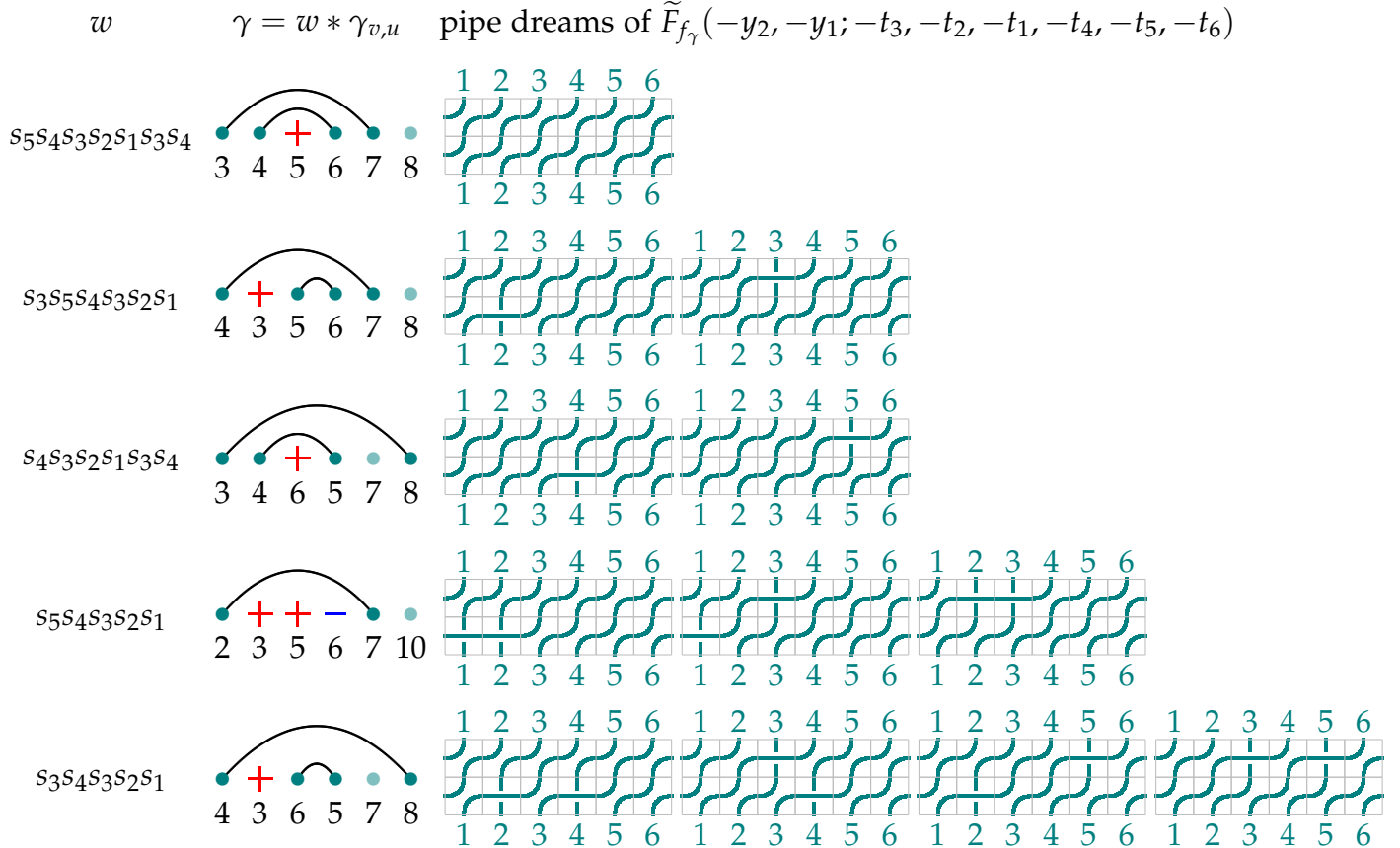
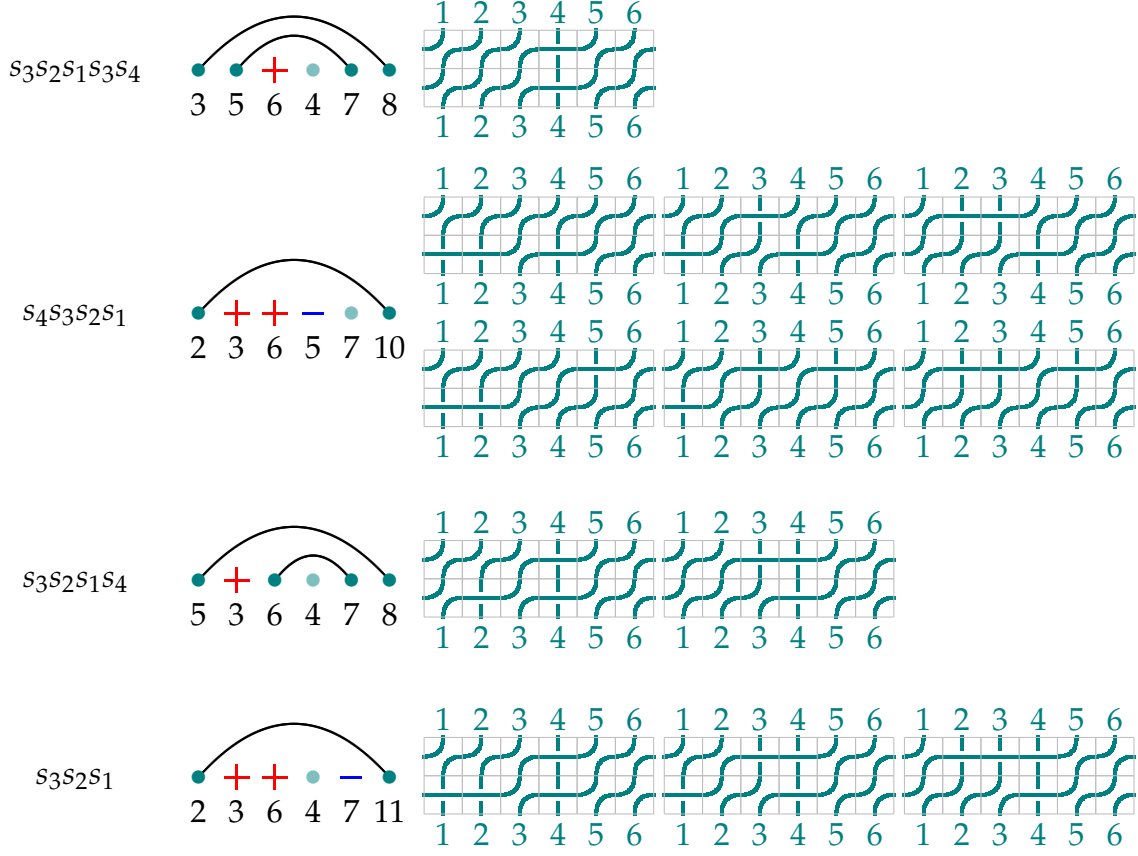
\def\term#1#2#3{
\makebox[0.15\linewidth]{$#1$}
\makebox[0.2\linewidth]{\def\m##1{\makebox[1.2pc]{$##1$}}%
$\begin{array}{l}#2\end{array}$}
\makebox[0.65\linewidth][l]{$\begin{matrix}#3\end{matrix}$}
}
$$\term{w}{\gamma=w*\gamma_{v,u}}{\text{pipe dreams of $\widetilde{F}_{f_\gamma}(-y_2,-y_1;-t_3,-t_2,-t_1,-t_4,-t_5,-t_6)$}}$$
$$\term{s_5s_4s_3s_2s_1s_3s_4}{\clan{42+..,}\\
\m{3}\m{4}\m{5}\m{6}\m{7}\m{8}}{
\BPD{
\M{1}\M{2}\M{3}\M{4}\M{5}\M{6}\\
\B\B\B\B\B\B\\\B\B\B\B\B\B\\
\M{1}\M{2}\M{3}\M{4}\M{5}\M{6}\\
}}$$
$$\term{s_3s_5s_4s_3s_2s_1}{\clan{4+1..,}\\
\m{4}\m{3}\m{5}\m{6}\m{7}\m{8}}{
\BPD{\M{1}\M{2}\M{3}\M{4}\M{5}\M{6}\\
\B\B\B\B\B\B\\\B\X\B\B\B\B\\
\M{1}\M{2}\M{3}\M{4}\M{5}\M{6}}
\BPD{\M{1}\M{2}\M{3}\M{4}\M{5}\M{6}\\
\B\B\X\B\B\B\\\B\B\B\B\B\B\\
\M{1}\M{2}\M{3}\M{4}\M{5}\M{6}}
}$$
$$\term{s_4s_3s_2s_1s_3s_4}{\clan{52+.,.}\\
\m{3}\m{4}\m{6}\m{5}\m{7}\m{8}}{
\BPD{\M{1}\M{2}\M{3}\M{4}\M{5}\M{6}\\
\B\B\B\B\B\B\\\B\B\B\X\B\B\\
\M{1}\M{2}\M{3}\M{4}\M{5}\M{6}}
\BPD{\M{1}\M{2}\M{3}\M{4}\M{5}\M{6}\\
\B\B\B\B\X\B\\\B\B\B\B\B\B\\
\M{1}\M{2}\M{3}\M{4}\M{5}\M{6}}
}$$
$$\term{s_5s_4s_3s_2s_1}{\clan{4++-.,}\\
\m{2}\m{3}\m{5}\m{6}\m{7}\m{10}}{
\BPD{\M{1}\M{2}\M{3}\M{4}\M{5}\M{6}\\
\B\B\B\B\B\B\\\X\X\B\B\B\B\\
\M{1}\M{2}\M{3}\M{4}\M{5}\M{6}}
\BPD{\M{1}\M{2}\M{3}\M{4}\M{5}\M{6}\\
\B\B\X\B\B\B\\\X\B\B\B\B\B\\
\M{1}\M{2}\M{3}\M{4}\M{5}\M{6}}
\BPD{\M{1}\M{2}\M{3}\M{4}\M{5}\M{6}\\
\B\X\X\B\B\B\\\B\B\B\B\B\B\\
\M{1}\M{2}\M{3}\M{4}\M{5}\M{6}}
}$$
$$\term{s_3s_4s_3s_2s_1}{\clan{5+1.,.}\\
\m{4}\m{3}\m{6}\m{5}\m{7}\m{8}}{
\BPD{\M{1}\M{2}\M{3}\M{4}\M{5}\M{6}\\
\B\B\B\B\B\B\\\B\X\B\X\B\B\\
\M{1}\M{2}\M{3}\M{4}\M{5}\M{6}}
\BPD{\M{1}\M{2}\M{3}\M{4}\M{5}\M{6}\\
\B\B\X\B\B\B\\\B\B\B\X\B\B\\
\M{1}\M{2}\M{3}\M{4}\M{5}\M{6}}
\BPD{\M{1}\M{2}\M{3}\M{4}\M{5}\M{6}\\
\B\B\B\B\X\B\\\B\X\B\B\B\B\\
\M{1}\M{2}\M{3}\M{4}\M{5}\M{6}}
\BPD{\M{1}\M{2}\M{3}\M{4}\M{5}\M{6}\\
\B\B\X\B\X\B\\\B\B\B\B\B\B\\
\M{1}\M{2}\M{3}\M{4}\M{5}\M{6}}
}$$
$$\term{s_3s_2s_1s_3s_4}{\clan{53+,..}\\
\m{3}\m{5}\m{6}\m{4}\m{7}\m{8}}{
\BPD{\M{1}\M{2}\M{3}\M{4}\M{5}\M{6}\\
\B\B\B\X\B\B\\\B\B\B\X\B\B\\
\M{1}\M{2}\M{3}\M{4}\M{5}\M{6}}
}$$
$$\term{s_4s_3s_2s_1}{\clan{5++-,.}\\
\m{2}\m{3}\m{6}\m{5}\m{7}\m{10}}{
\BPD{\M{1}\M{2}\M{3}\M{4}\M{5}\M{6}\\
\B\B\B\B\B\B\\\X\X\B\X\B\B\\
\M{1}\M{2}\M{3}\M{4}\M{5}\M{6}}
\BPD{\M{1}\M{2}\M{3}\M{4}\M{5}\M{6}\\
\B\B\X\B\B\B\\\X\B\B\X\B\B\\
\M{1}\M{2}\M{3}\M{4}\M{5}\M{6}}
\BPD{\M{1}\M{2}\M{3}\M{4}\M{5}\M{6}\\
\B\X\X\B\B\B\\\B\B\B\X\B\B\\
\M{1}\M{2}\M{3}\M{4}\M{5}\M{6}}\\
\BPD{\M{1}\M{2}\M{3}\M{4}\M{5}\M{6}\\
\B\B\B\B\X\B\\\X\X\B\B\B\B\\
\M{1}\M{2}\M{3}\M{4}\M{5}\M{6}}
\BPD{\M{1}\M{2}\M{3}\M{4}\M{5}\M{6}\\
\B\B\X\B\X\B\\\X\B\B\B\B\B\\
\M{1}\M{2}\M{3}\M{4}\M{5}\M{6}}
\BPD{\M{1}\M{2}\M{3}\M{4}\M{5}\M{6}\\
\B\X\X\B\X\B\\\B\B\B\B\B\B\\
\M{1}\M{2}\M{3}\M{4}\M{5}\M{6}}
}$$

$$\term{s_3s_2s_1s_4}{\clan{5+2,..}\\
\m{5}\m{3}\m{6}\m{4}\m{7}\m{8}}{
\BPD{\M{1}\M{2}\M{3}\M{4}\M{5}\M{6}\\
\B\B\B\X\B\B\\\B\X\B\X\B\B\\
\M{1}\M{2}\M{3}\M{4}\M{5}\M{6}}
\BPD{\M{1}\M{2}\M{3}\M{4}\M{5}\M{6}\\
\B\B\X\X\B\B\\\B\B\B\X\B\B\\
\M{1}\M{2}\M{3}\M{4}\M{5}\M{6}}
}$$

$$\term{s_3s_2s_1}{\clan{5++,-.}\\
\m{2}\m{3}\m{6}\m{4}\m{7}\m{11}}{
\BPD{\M{1}\M{2}\M{3}\M{4}\M{5}\M{6}\\
\B\B\B\X\B\B\\\X\X\B\X\B\B\\
\M{1}\M{2}\M{3}\M{4}\M{5}\M{6}}
\BPD{\M{1}\M{2}\M{3}\M{4}\M{5}\M{6}\\
\B\B\X\X\B\B\\\X\B\B\X\B\B\\
\M{1}\M{2}\M{3}\M{4}\M{5}\M{6}}
\BPD{\M{1}\M{2}\M{3}\M{4}\M{5}\M{6}\\
\B\X\X\X\B\B\\\B\B\B\X\B\B\\
\M{1}\M{2}\M{3}\M{4}\M{5}\M{6}}
}$$
Thus 
\begin{align*}
\mathfrak{S}_{v}(x;t)\cdot
\mathfrak{S}_{u}(x;y)
 =\, & \mathfrak{S}_{613425}(x;t)+ [(t_2-y_1)+(t_1-y_2)]\mathfrak{S}_{612435}(x;t)\\
 & +[(t_4-y_1)+(t_5-y_2)]\mathfrak{S}_{51342}(x;t)\\
 & + [(t_2-y_1)(t_3-y_1)+(t_1-y_2)(t_3-y_1)+(t_1-y_2)(t_2-y_2)]\mathfrak{S}_{612345}(x;t)\\
 & +[(t_2-y_1)(t_4-y_1)+(t_1-y_2)(t_4-y_1)\\&\quad+(t_2-y_1)(t_5-y_2)+(t_1-y_2)(t_5-y_2)]\mathfrak{S}_{51243}(x;t)\\
 &+(t_4-y_1)(t_4-y_2)\mathfrak{S}_{41352}(x;t)\\
 & + [(t_2-y_1)(t_3-y_1)(t_4-y_1)+(t_1-y_2)(t_3-y_1)(t_4-y_1)\\&\quad+(t_1-y_2)(t_2-y_2)(t_4-y_1)+(t_2-y_1)(t_3-y_1)(t_5-y_2)\\&\quad+(t_1-y_2)(t_3-y_1)(t_5-y_2)+(t_1-y_2)(t_2-y_2)(t_5-y_2)]\mathfrak{S}_{51234}(x;t)\\
 & +[(t_2-y_1)(t_4-y_1)(t_4-y_2)+(t_1-y_2)(t_4-y_1)(t_4-y_2)]\mathfrak{S}_{41253}(x;t)\\
 &+[(t_2-y_1)(t_3-y_1)(t_4-y_1)(t_4-y_2)+(t_1-y_2)(t_3-y_1)(t_4-y_1)(t_4-y_2)\\&\quad+(t_1-y_2)(t_2-y_2)(t_4-y_1)(t_4-y_2)]\mathfrak{S}_{4123}(x;t).
\end{align*}
\end{example}

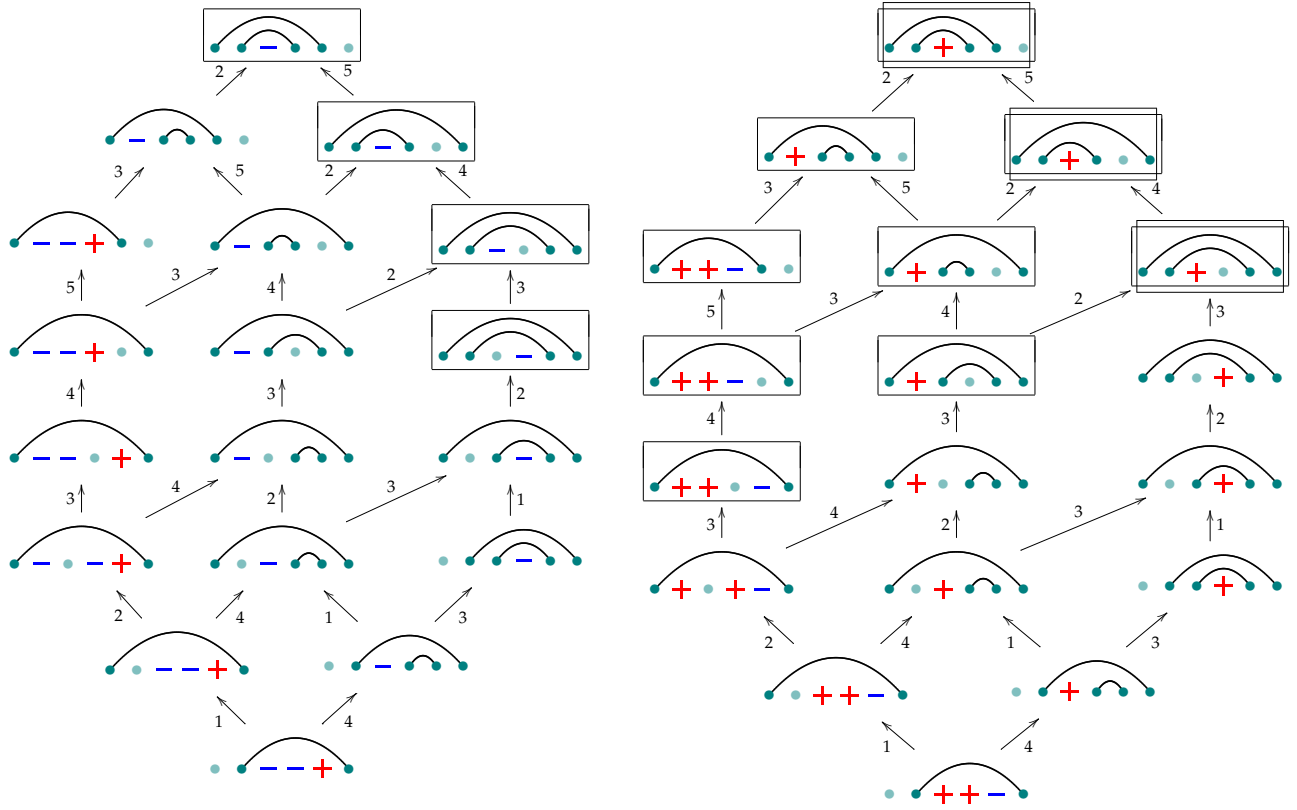
\begin{figure}[ht]
$$
\def\myclan#1{{\hspace{-1pc}
    \begin{matrix}\\[-1pc]
    \clan{#1}\vphantom{\dfrac12}\\\end{matrix}
    \hspace{-1pc}}}
\def\fmyclan#1{{\hspace{-.5pc}\begin{array}{|c|}\hline\\[-1pc]
    \hspace{-.5pc}\rule{0pc}{1.8pc}
    \clan{#1}\hspace{-.5pc}
    \\\hline\end{array}\hspace{-.5pc}}}
\def\ffmyclan#1{\fbox{$\hspace{-.5pc}\begin{array}{|c|}\hline\\[-1pc]
    \hspace{-.5pc}\rule{0pc}{1.8pc}
    \clan{#1}\hspace{-.5pc}
    \\\hline\end{array}\hspace{-.5pc}$}}
\begin{matrix}
\scalebox{0.7}{%
    \xymatrix@R=1.2em@C=-1.5em{
&& {\fmyclan{42-..,}}\\
&{\myclan{4-1..,}}\ar[ur]^2 &&
{\fmyclan{52-.,.}}\ar[ul]_5\\
{\myclan{4--+.,}}\ar[ur]^3&&
{\myclan{5-1.,.}}\ar[ul]_5\ar[ur]^2&& 
{\fmyclan{53-,..}}\ar[ul]_4\\
{\myclan{5--+,.}}\ar[u]^5\ar[urr]^3&&
{\myclan{5-2,..}}\ar[u]^4\ar[urr]^2&& 
{\fmyclan{53,-..}}\ar[u]_3\\
{\myclan{5--,+.}}\ar[u]^4&&
{\myclan{5-,1..}}\ar[u]^3&& 
{\myclan{5,2-..}}\ar[u]_2\\
{\myclan{5-,-+.}}\ar[u]^3\ar[urr]^4&&
{\myclan{5,-1..}}\ar[u]^2\ar[urr]^3&& 
{\myclan{,42-..}}\ar[u]_1\\
& {\myclan{5,--+.}}\ar[ul]^2\ar[ur]_4 &&
{\myclan{,4-1..}}\ar[ul]^1\ar[ur]_3\\
&&{\myclan{,4--+.}}\ar[ul]^1\ar[ur]_4
}}
\qquad 
\scalebox{0.7}{%
    \xymatrix@R=1.2em@C=-1.5em{
&& {\ffmyclan{42+..,}}\\
&{\fmyclan{4+1..,}}\ar[ur]^2 &&
{\ffmyclan{52+.,.}}\ar[ul]_5\\
{\fmyclan{4++-.,}}\ar[ur]^3&&
{\fmyclan{5+1.,.}}\ar[ul]_5\ar[ur]^2&& 
{\ffmyclan{53+,..}}\ar[ul]_4\\
{\fmyclan{5++-,.}}\ar[u]^5\ar[urr]^3&&
{\fmyclan{5+2,..}}\ar[u]^4\ar[urr]^2&& 
{\myclan{53,+..}}\ar[u]_3\\
{\fmyclan{5++,-.}}\ar[u]^4&&
{\myclan{5+,1..}}\ar[u]^3&& 
{\myclan{5,2+..}}\ar[u]_2\\
{\myclan{5+,+-.}}\ar[u]^3\ar[urr]^4&&
{\myclan{5,+1..}}\ar[u]^2\ar[urr]^3&& 
{\myclan{,42+..}}\ar[u]_1\\
& {\myclan{5,++-.}}\ar[ul]^2\ar[ur]_4 &&
{\myclan{,4+1..}}\ar[ul]^1\ar[ur]_3\\
&&{\myclan{,4++-.}}\ar[ul]^1\ar[ur]_4
}}
\end{matrix}
$$
    \caption{Hasse diagram over $\gamma_{v,u}$}
    \label{fig:HasseEg}
\end{figure}

\subsection{Double Schubert Structure Constants}
\label{sec:permutational}
Let us assume $p\geq q$. 
\begin{definition}\label{def:perm}
A $(p,m,q)$-preclan $\gamma$ is called {\it permutational}, if 
\begin{itemize}
    \item the first $q$ nodes are left-ends; 
    \item the next ($p-q$) nodes are $\clan{+}$.
\end{itemize}
In particular, $v_\gamma=e$ and there is no $\clan{-}$ in $\gamma$.    
For a permutational $(p,m,q)$-preclan $\gamma$, we define a permutation $\eta_{\gamma}\in S_{q+m}$ by 
\begin{equation}
\eta_\gamma(i)=f_\gamma(i+p-q)-p,\quad \text{for $1\leq i\leq q+m$.}
\end{equation}
\end{definition}

If $\gamma$ is permutational,  then $\{f_\gamma(p-q+1),\ldots,f_\gamma(p+m)\}$ is the set of positions of all right-ends and $\clan{,}$'s, thus $\eta_{\gamma}\in S_{q+m}$ is a well-defined permutation. 
Since $\eta_\gamma(q+1)<\cdots<\eta_\gamma(q+m)$,  we find that $\mathfrak{S}_{\eta_\gamma}(x_1,\ldots,x_{q+m};t_1,\ldots,t_{q+m})$ involves only $x_1,\ldots,x_q$.  
For example, 
\begin{equation}\label{permutationalpre}
\def\m#1{\makebox[1.2pc]{$#1$}}
\gamma = \clan{695++,..,,.,}
\qquad 
\begin{aligned}
\eta_\gamma =&\,\m{}\m{}\m{3}\m{6}\m{2}\m{1}\m{4}\m{5}\m{7}\\
f_{\gamma}=&\, \m{4}\m{5}\m{8}\m{11}\m{7}\m{6}\m{9}\m{10}\m{12}\m{13}\m{14}\m{15}. 
\end{aligned}
\end{equation}

\begin{lemma}
For a $(p,m,q)$-preclan $\gamma$, we have 
$$\widetilde{F}_{f_\gamma}(-t_q,\ldots,-t_1;-t_p,\ldots,-t_1,-t_{p+1},\ldots,-t_n)
=\begin{cases}
\mathfrak{S}_{\eta_\gamma}(-t_q,\ldots,-t_1;-t_{p+1},\ldots,-t_n),\\
\qquad \quad \text{if $\gamma$ is permutational;}\\[1ex]
0, \qquad \text{otherwise}. 
\end{cases}
$$
\end{lemma}

\begin{proof}
Let $\operatorname{PD}(f_\gamma)$ denote the set of pipe dreams of $f_\gamma$.
We have 
$$\widetilde{F}_{f_\gamma}(-t_q,\ldots,-t_1;-t_p,\ldots,-t_1,-t_{p+1},\ldots,-t_n) = \sum_{\pi\in \operatorname{PD}(f_\gamma)}
\operatorname{wt}(\pi),
$$
where a $\BPD{\X}$ at position $(i,j)$  contributes $-t_{a}+t_{b}$ to $\operatorname{wt}(\pi)$ with 
$$
a=q+1-i
,\qquad 
b = \text{the $j$-th component of $(p,\ldots,1,{p+1}\ldots,n)$}. $$
Let us decompose a pipe dream of $f_\gamma$ into the following regions. 
$$
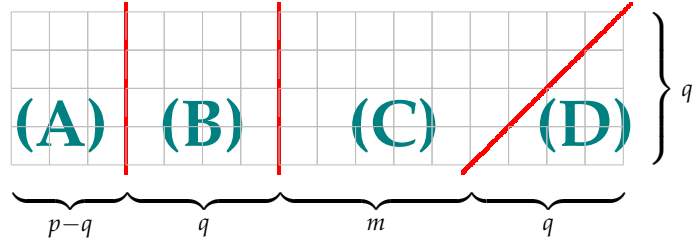
\setlength{\unitlength}{1.2pc}
\def\diag{%
\begin{picture}(0,1)%
   \color{red}
    \linethickness{0.08\unitlength}
    \qbezier(0.2,1.2)(-1,0)(-1.2,-.2)
\end{picture}}
\def\|{%
\begin{picture}(0,1)%
   \color{red}
    \linethickness{0.08\unitlength}
    \qbezier(0.0,1.2)(0,0)(0.0,-.2)
\end{picture}}
\def\A#1{\raisebox{-0.5\unitlength}[0pc][0pc]{%
    \makebox[0pc][c]{\huge\bf#1}}}
\begin{array}{c}
\left.\BPD{
\O\O\O\|\O\O\O\O\|\O\O\O\O\O\O\O\O\O\diag\\
\O\O\O\|\O\O\O\O\|\O\O\O\O\O\O\O\O\diag\O\\
\O\A{\,\,(A)}\O\O\|\O\O\A{(B)}\O\O\|\O\O\O\A{(C)}\O\O\O\O\diag\O\A{(D)}\O\\
\O\O\O\|\O\O\O\O\|\O\O\O\O\O\O\diag\O\O\O\\
}\ \ \right\}{\scriptstyle q}
\\
\,\underbrace{\rule{3\unitlength}{0pc}}_{p-q}\!
\underbrace{\rule{4\unitlength}{0pc}}_{q}\!
\underbrace{\rule{5\unitlength}{0pc}}_{m}\!
\underbrace{\rule{4\unitlength}{0pc}}_{q}\qquad
\end{array}
$$

Let $\operatorname{PD}'(f_\gamma)$ be the set of pipe dreams of $f_\gamma$ with  region $\textbf{(B)}$ consisting of only $\BPD{\B}$. 
Firstly, we claim that 
$$\sum_{\pi\in \operatorname{PD}(f_\gamma)}
\operatorname{wt}(\pi) = \sum_{\pi\in \operatorname{PD}'(f_\gamma)}
\operatorname{wt}(\pi).$$
We prove the claim by constructing a sign-reversing involution $\theta$ on the set $\operatorname{PD}(f_\gamma)\setminus \operatorname{PD}'(f_\gamma)$ as follows. Let $\pi\in \operatorname{PD}(f_\gamma)\setminus \operatorname{PD}'(f_\gamma)$. If there is a $\BPD{\X}$ tile  on the diagonal of region $\textbf{(B)}$, then $\operatorname{wt}(\pi)=0$,  set $\theta(\pi)=\pi$. Otherwise,
pick a $\BPD{\X}$ at position $(i,j)$ with $|i-j|$ smallest (if there are more than one such $\BPD{\X}$, pick the  one  with $j+i$ maximal). Such a position $(i,j)$ is unique, since it is impossible to have the following  "symmetric" double intersections:
$$\BPD{
\B\B\B\X\\
\B\B\B\B\\
\B\B\B\B\\
\X\B\B\B}.$$
We define $\theta(\pi)$ to be the pipe dream obtained from $\pi$ by reflecting the $\BPD{\X}$ tile at $(i,j)$ to position $(j,i)$ along the diagonal of region $\textbf{(B)}$, that is, $\theta$ is the following local move
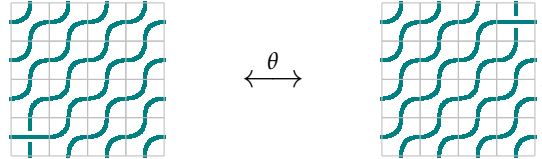
\[
\BPD{\B\B\B\B\\
\B\B\B\B\\
\B\B\B\B\\
\X\B\B\B\\}\qquad \stackrel{\theta}\longleftrightarrow\qquad 
\BPD{
\B\B\B\X\\
\B\B\B\B\\
\B\B\B\B\\
\B\B\B\B}.\]
Since the weight of a $\BPD{\X}$ tile at position $(i,j)$ of region $\textbf{(B)}$ is the negative of 
the weight of a $\BPD{\X}$ tile at position $(j,i)$, we have 
$\operatorname{wt}(\pi)+\operatorname{wt}(\theta(\pi))=0$. Thus the claim holds.

Next, we show that if region $\textbf{(B)}$ consists of only $\BPD{\B}$, then the regions $\textbf{(A)},\textbf{(D)}$ also consist of only  $\BPD{\B}$. Let $\operatorname{PD}''(f_\gamma)$ denote the set of pipe dreams of $f_\gamma$ with $\BPD{\X}$ only appearing in region $\textbf{(C)}$. We aim to show that  $\operatorname{PD}'(f_\gamma)=\operatorname{PD}''(f_\gamma)$ as sets of pipe dreams. 
Since region $\textbf{(B)}$ consists of only $\BPD{\B}$, 
we have $f_\gamma(i)\leq p$ for $1\leq i\leq p-q$ and $f_\gamma(i)>p$ for $p-q+1\leq i\leq p$. By Definition \ref{fgamma}, the first $p$ nodes of $\gamma$ consist of $q$ left-ends followed by $(p-q)$ $\clan{+}$'s. Moreover, it is easy to see that there are no $\clan{-}$ in $\gamma$. This implies that $\gamma$ is permutational. Therefore,
$$f_\gamma(i-q)=i,\quad 1\le i\le p.$$
Thus regions $\textbf{(D)}$ and $\textbf{(A)}$ consist of  only $\BPD{\B}$. 

Since $\gamma$ is permutational, the pipes exiting  from the upper boundary from columns $p+1,\ldots,n$ are exactly the pipes entering from the lower bound from columns $p-q+1,\ldots,p+m$. 
Thus, the pipe dreams in $\operatorname{PD}''(f_\gamma)$ are exactly pipe dreams of $\eta_\gamma$ where we omit the $\BPD{\B}$'s in the last $m$ rows since $\eta_{\gamma}(q+1)<\cdots<\eta_{\gamma}(q+m)$.
By the definitions of 
$\eta_\gamma$ and $f_\gamma$, we have 
$$\sum_{\pi\in \operatorname{PD}(f_\gamma)}
\operatorname{wt}(\pi) = \sum_{\pi\in \operatorname{PD}''(f_\gamma)}
\operatorname{wt}(\pi)=\mathfrak{S}_{\eta_\gamma}(-t_q,\ldots,-t_1;-t_{p+1},\ldots,-t_n).$$

\end{proof}

\begin{example}
    Let $\gamma$ be the permutational preclan in \eqref{permutationalpre}, the following is a pipe dream of $f_\gamma$ in $\operatorname{PD}''(f_\gamma)$. Restricting to region \textbf{(C)} and adding $m$ rows consisting entirely of $\BPD{\B}$'s, we obtain a pipe dream of $\eta_\gamma$.     $$\setlength{\unitlength}{1.2pc}
\def\diag
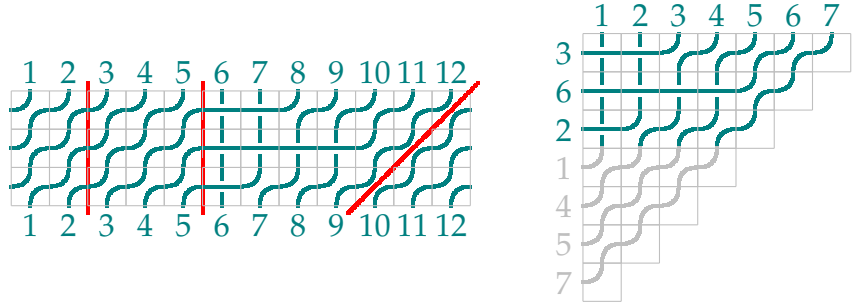
{%
\begin{picture}(0,1)%
   \color{red}
    \linethickness{0.08\unitlength}
    \qbezier(0.2,1.2)(-1,0)(-1.2,-.2)
\end{picture}}
\def\|{%
\begin{picture}(0,1)%
   \color{red}
    \linethickness{0.08\unitlength}
    \qbezier(0.0,1.2)(0,0)(0.0,-.2)
\end{picture}}
    \BPD{\M{1}\M{2}\M{3}\M{4}\M{5}\M{6}\M{7}\M{8}\M{9}\M{10}\M{11}\M{12}\\
    \B\B\|\B\B\B\|\X\X\B\B\B\B\B\diag\\
        \B\B\|\B\B\B\|\X\X\X\X\B\B\diag\B\\
        \B\B\|\B\B\B\|\X\B\B\B\B\diag\B\B\\
\M{1}\M{2}\M{3}\M{4}\M{5}\M{6}\M{7}\M{8}\M{9}\M{10}\M{11}\M{12}
    }\qquad
\BPD{\M{}\M{1}\M{2}\M{3}\M{4}\M{5}\M{6}\M{7}\\\M{3}\X\X\B\B\B\B\J\\\M{6}\X\X\X\X\B\J\\\M{2}\X\B\B\B\J\\{\color{lightgray}\M{1}\B\B\B\J}\\{\color{lightgray}\M{4}\B\B\J}\\{\color{lightgray}\M{5}\B\J}\\{\color{lightgray}\M{7}\J}}
    $$
\end{example}

Now we obtain the main result of this paper. 

\begin{theorem}\label{thm:maindouble}
For a Richardson $(p,m,q)$-preclan $\gamma$, we have 
$$
\mathfrak{S}_{v_\gamma}(x;t)\cdot
\mathfrak{S}_{u_\gamma}(x;t)
=\sum_{w\in S_n} c_{v_\gamma,u_\gamma}^w(t)\cdot \mathfrak{S}_w(x;t),$$
where 
$$c_{v_\gamma,u_\gamma}^w(t)
=\begin{cases}
\mathfrak{S}_{\eta_{w*\gamma}}(-t_q,\ldots,-t_1;-t_{p+1},\ldots,-t_n),\\
\qquad \quad \text{if $\ell(w*\gamma)=\ell(w)+\ell(\gamma)$ and $w*\gamma$ is permutational};\\[1ex]
0, \qquad \text{otherwise}. 
\end{cases}
$$
\end{theorem}

\begin{example}
Let $p=3,m=1,q=2$. 
Consider $v=412356$ and $u=314526$. 
The permutational preclan over $\gamma$ was doubly framed in Figure \ref{fig:HasseEg} (right). 
\def\term#1#2#3#4{
\makebox[0.15\linewidth]{$#1$}
\makebox[0.2\linewidth]{\def\m##1{\makebox[1.2pc]{$##1$}}%
$\begin{array}{l}#2\end{array}$}
\makebox[0.3\linewidth][l]{$\begin{matrix}#3\end{matrix}$}
}
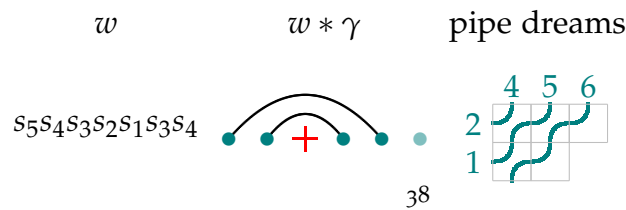
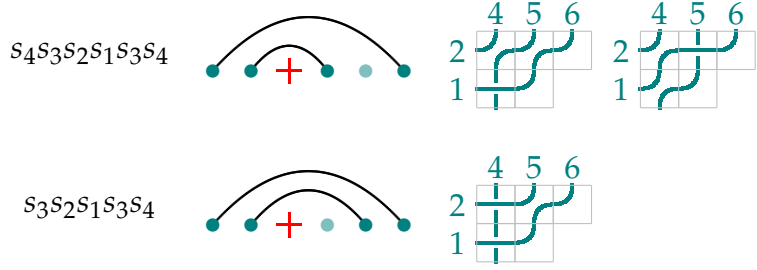
$$\term{w}{w*\gamma}{\text{pipe dreams}}{\text{bumpless pipe dreams}}$$
$$\term{s_5s_4s_3s_2s_1s_3s_4}{\clan{42+..,}\\
}{
\BPD{\M{}\M{4}\M{5}\M{6}\\
\M{2}\B\B\J\\\M{1}\B\J\\}}{
\BPD{\F\H\H\M{2}\\\I\F\H\M{1}\\\M{4}\M{5}\M{6}}
}$$
$$\term{s_4s_3s_2s_1s_3s_4}{\clan{52+.,.}\\
}{
\BPD{\M{}\M{4}\M{5}\M{6}\\
\M{2}\B\B\J\\\M{1}\X\J}
\BPD{\M{}\M{4}\M{5}\M{6}\\
\M{2}\B\X\J\\\M{1}\B\J}
}{
\BPD{\F\H\H\M{2}\\\I\O\F\M{1}\\\M{4}\M{5}\M{6}}
\BPD{\O\F\H\M{2}\\\F\J\F\M{1}\\\M{4}\M{5}\M{6}}
}$$
$$\term{s_3s_2s_1s_3s_4}{\clan{53+,..}\\
}{
\BPD{\M{}\M{4}\M{5}\M{6}\\
\M{2}\X\B\J\\\M{1}\X\J\\}
}{
\BPD{\O\F\H\M{2}\\\O\I\F\M{1}\\\M{4}\M{5}\M{6}}
}$$
Thus 
$$\begin{aligned}
\mathfrak{S}_{v}(x;t)\cdot
\mathfrak{S}_{u}(x;t)
& = \mathfrak{S}_{613425}(x;t)+(t_4+t_5-t_1-t_2)\mathfrak{S}_{51342}(x;t)\\
&\quad +(t_4-t_1)(t_4-t_2)\mathfrak{S}_{41352}(x;t).
\end{aligned}$$
\end{example}

\subsection{Application to simple coefficients}

Recall that the Stanley symmetric function for a permutation $w\in S_\infty$ is the stable limit 
$$
F_w(x) = 
\lim_{k\to \infty} \mathfrak{S}_{1^k\times w}(x),$$
where $1^k\times w=1\cdots k(w(1)+k)(w(2)+k)\cdots$. In particular, 
$$F_w(x_1,\ldots,x_q) = F_w(x_1,\ldots,x_q,0,\ldots)
= \mathfrak{S}_{1^q\times w}(x_1,\ldots,x_q,0,\ldots).
$$
Recall that $g_\gamma\in S_n$ is defined by $g_\gamma(i)=f_\gamma(i-q),\ \text{for $1\le i\le n$}.$

\begin{lemma}\label{lem:affinestanley}
For  a $(p,m,q)$-preclan $\gamma$ with $v_\gamma=e$, we have 
$$\widetilde{F}_{f_\gamma}(x_1,\ldots,x_q)
= F_{g_\gamma}(x_1,\ldots,x_q). $$
\end{lemma}
\begin{proof} 
By Lemma \ref{ggamma}, $g_\gamma$
is an ordinary permutation in $S_n$.
Let $\pi$ be a pipe dream for the affine Stanley symmetric function $f_\gamma$, see Figure \ref{fig:pipedreamofg} for an illustration. Since the last $q$ values of $f_\gamma$ are increasing and larger than $n$, the region $\textbf{(B)}$ of $\pi$ consists of only $\BPD{\B}$. Thus the region $\textbf{(B')}$ also consists of only $\BPD{\B}$.
By appending the region $\textbf{(C)}$ consisting of only $\BPD{\B}$ at the bottom of region $\textbf{(A)}$, we find that region $\textbf{(A)}\cup\textbf{(C)}$ forms a pipe dream of $g_\gamma$. That is,
$$\widetilde{F}_{f_\gamma}(x_1,\ldots,x_q)
= \mathfrak{S}_{g_\gamma}(x_1,\ldots,x_q,0,\ldots,0). $$

It is easy to see $\mathfrak{S}_{g_\gamma}(x_1,\ldots,x_n)=\mathfrak{S}_{1\times g_\gamma}(0,x_1,\ldots,x_n)$.
Since the first  $q+1$ values of $1\times g_\gamma$ are increasing, we have $\mathfrak{S}_{1\times g_\gamma}(0,x_1,\ldots,x_q,0,\ldots,0)$ is symmetric. Therefore,
\begin{align*}
\mathfrak{S}_{g_\gamma}(x_1,\ldots,x_q,0,\ldots,0)&=\mathfrak{S}_{1\times g_\gamma}(0,x_1,\ldots,x_q,0,\ldots,0)\\
&=\mathfrak{S}_{1\times g_\gamma}(x_1,\ldots,x_q,0,0,\ldots,0)\\
&=\cdots\\
&=\mathfrak{S}_{1^q\times g_\gamma}(x_1,\ldots,x_q,0,0,\ldots,0)\\
&=F_{g_\gamma}(x_1,\ldots,x_q). 
\qedhere
\end{align*}

\begin{figure}
    \centering
    $$
    \setlength{\unitlength}{1.2pc}
\def\diag{%
\begin{picture}(0,1)%
   \color{red}
    \linethickness{0.08\unitlength}
    \qbezier(0.2,1.2)(-1,0)(-1.2,-.2)
\end{picture}}
\def\|{%
\begin{picture}(0,1)%
   \color{red}
    \linethickness{0.08\unitlength}
    \qbezier(0.0,1.2)(0,0)(0.0,-.2)
\end{picture}}
\def\_{%
\begin{picture}(0,1)%
   \color{red}
    \linethickness{0.08\unitlength}
    \qbezier(-12.5,0.0)(0,0)(0,0)
\end{picture}}
\def\A#1{\raisebox{-0.5\unitlength}[0pc][0pc]{%
    \makebox[0pc][c]{\huge\bf#1}}}
\begin{array}{l}
\left.\BPD{
\M{}\M{}\M{}\O\|\O\O\O\O\O\O\O\O\O\O\O\O\diag\\
\M{}\M{}\O\O\|\O\O\O\O\A{(A)}\O\O\O\O\O\O\O\diag\O\\
\M{}\O\A{(B')}\O\O\|\O\O\O\O\O\O\O\O\O\O\diag\O\A{(B)}\O\\
\O\O\O\O\|\O\O\O\O\O\O\O\O\O\diag\O\O\O\_}\right\}{\,\,\scriptstyle q}\\
\left.\BPD{
\M{}\M{}\M{}\M{}\O\O\O\O\O\O\O\O\\
\M{}\M{}\M{}\M{}\O\O\O\O\O\O\O\\
\M{}\M{}\M{}\M{}\O\O\A{(C)}\O\O\O\O\\
\M{}\M{}\M{}\M{}\O\O\O\O\O\\
\M{}\M{}\M{}\M{}\O\O\O\O\\
\M{}\M{}\M{}\M{}\O\O\O\\
\M{}\M{}\M{}\M{}\O\O\\
\M{}\M{}\M{}\M{}\O\\
}\right.\phantom{\,\,\scriptstyle q}
\\
\rule{4\unitlength}{0pc}
\end{array}$$
\vspace{-.3cm}
\caption{The pipe dream of $g_{\gamma}$}
\label{fig:pipedreamofg}
\end{figure}
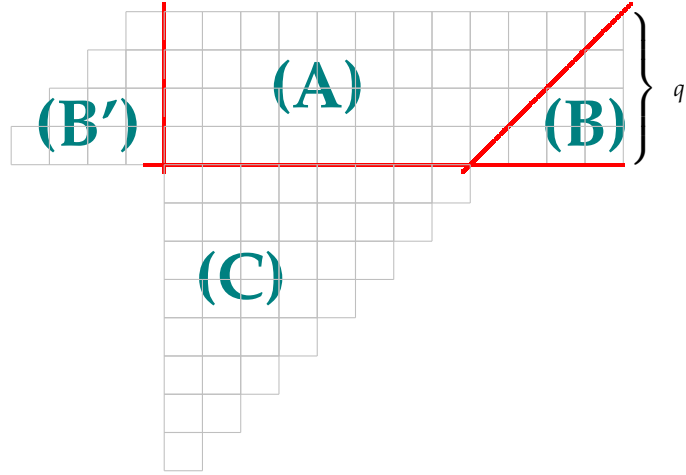
\end{proof}

\def\m#1{\makebox[1.2pc]{$#1$}}
For example, let $\gamma$ be as in \eqref{eq:Egvgamma=e}, where $f_\gamma=\m{4}\m{9}\m{10}\m{6}\m{5}\m{8}\m{11}\m{12}\m{13}\m{14}\m{18}, g_\gamma=\m{1}\m{2}\m{3}\m{7}\m{4}\m{9}\m{10}\m{6}\m{5}\m{8}\m{11}$.
$$
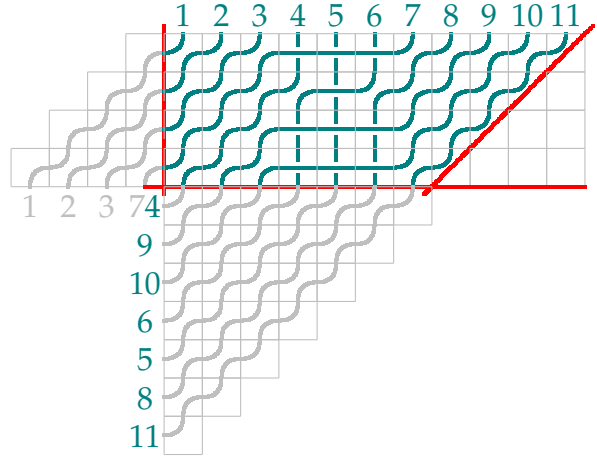
\setlength{\unitlength}{1.2pc}
\def\diag{%
\begin{picture}(0,1)%
   \color{red}
    \linethickness{0.08\unitlength}
    \qbezier(0.2,1.2)(-1,0)(-1.2,-.2)
\end{picture}}
\def\|{%
\begin{picture}(0,1)%
   \color{red}
    \linethickness{0.08\unitlength}
    \qbezier(0.0,1.2)(0,0)(0.0,-.2)
\end{picture}}
\def\_{%
\begin{picture}(0,1)%
   \color{red}
    \linethickness{0.08\unitlength}
    \qbezier(-11.5,0.0)(0,0)(0,0)
\end{picture}}
\def\A#1{\raisebox{-0.5\unitlength}[0pc][0pc]{%
    \makebox[0pc][c]{\huge\bf#1}}}
\begin{array}{l}
\BPD{
\M{}\M{}\M{}\M{}\M{1}\M{2}\M{3}\M{4}\M{5}\M{6}\M{7}\M{8}\M{9}\M{10}\M{11}\\
\M{}\M{}\M{}{\color{lightgray}\F}\|\B\B\B\X\X\X\B\B\B\B\J\diag\\
\M{}\M{}{\color{lightgray}\F\B}\|\B\B\B\B\X\B\B\B\B\J\diag\O\\
\M{}{\color{lightgray}\F\B\B}\|\B\B\B\X\X\X\B\B\J\diag\O\O\\
{\color{lightgray}\F\B\B\B}\|\B\B\B\X\X\X\B\J\diag\O\O\O\_\\
{\color{lightgray}\M{1}\M{2}\M{3}}\M{{\color{lightgray}7}4}{\color{lightgray}\B\B\B\B\B\B\J}\M{}\M{}\M{}\M{}\\
\M{}\M{}\M{}\M{9}{\color{lightgray}\B\B\B\B\B\J}\\
\M{}\M{}\M{}\M{10}{\color{lightgray}\B\B\B\B\J}\\
\M{}\M{}\M{}\M{6}{\color{lightgray}\B\B\B\J}\\
\M{}\M{}\M{}\M{5}{\color{lightgray}\B\B\J}\\
\M{}\M{}\M{}\M{8}{\color{lightgray}\B\J}\\
\M{}\M{}\M{}\M{11}{\color{lightgray}\J}\\
}
\end{array}
$$

For a partition $\lambda=(\lambda_1,\ldots,\lambda_k)$, we can associate  a $k$-Grassmannian permutation  $w_{\lambda}$ to $\lambda$  by 
$$ w_\lambda(i)=i+\lambda_{k+1-i},\quad \text{for}\ 1\leq i\leq k,$$ and let the remaining entries $w_{k+1}<\cdots<w_n$ to be the elements of $[n]\backslash\{w_1,\ldots,w_k\}$ written in increasing order. 
A permutation $w=w(1)w(2)\cdots w(n)\in S_n$ is called \emph{321-avoiding} if there does not exist a sequence
$i<j<k$ such that  $w(i) > w(j) > w(k).$

\begin{lemma}[{\cite[Section 2]{BJS93}}]
\label{skewshape}
Let $w\in S_n$. Then $w$ is  $321$-avoiding if and only if $w$ can be written as $w=w_{\lambda/\mu}:=w_{\lambda}w^{-1}_{\mu}$ with $\ell(w)=\ell(w_\lambda)-\ell(w_\mu^{-1})$ for some partitions $\mu\subseteq\lambda$. 
\end{lemma}

For a 321-avoiding permutation $w\in S_n$, an explicit way of constructing $\lambda$ and $\mu$ is as follows. 
Denote $I(w)=\{i:w(i)>i\}$ and $I^c(w)=\{i:w(i)\leq i\}$ and $J(w)=\{w(i):i\in I(w)\}$. By \cite{BJS93}, after deleting the empty rows in $I^c(w)$ and empty columns in $J(w)$ from its Rothe diagram $D(w)$ and then flipping each column, we obtain a skew shape Young diagram, denoted as $\lambda/\mu$, where $\lambda$ is the minimal partition. 

For example, $w=312465$ is $321$-avoiding with $I(w)=\{1,5\}$, $I^c(w)=\{2,3,4,6\}$ and $J(w)=\{3,6\}$. After deleting rows in $I^c(w)$ and columns in $J(w)$ from $D(w)$ and flipping each column, we have $\lambda=(4,2)$ and $\mu=(3,0)$. Thus $312465=w_{(4,2)}w^{-1}_{(3,0)}=361425(152346)^{-1}$.

Note that if $u$ is $321$-avoiding, then $u^{-1}$ is  also $321$-avoiding. By Lemma \ref{skewshape}, we can also write $u^{-1}=w_{\lambda}w_{\mu}^{-1}$ for some partitions $\mu\subseteq \lambda$. 
Then we have 
\begin{equation}\label{uu}
u=w_{\mu}w^{-1}_{\lambda}=:w_{\mu}\tilde{u}, 
\end{equation}
 where $\ell(\tilde{u})=\ell(u)+\ell(w_{\mu})$.
For example, let $u^{-1}=231465$, which is  321-avoiding, then we can write $u^{-1}=w_{(3,1,1)}w^{-1}_{(2,0,0)}$. Thus
$u=w_{(2,0,0)}w^{-1}_{(3,1,1)}$ and $\tilde{u}=w^{-1}_{(3,1,1)}=(236145)^{-1}=412563$.

Recall that the Edelman--Greene coefficient $EG^w_{\mu}$ is defined by
 $$F_{w}(x)=\sum_{\mu}EG^w_{\mu}\cdot s_{\mu}(x),$$
 which is known to be the number of reduced word tableaux of $w$ with shape $\mu$ \cite{EG}; for more combinatorial formulas, see \cite{LLS18,FGS}.

\begin{theorem}\label{thm:egcoeff}
Let $v$ be an inverse Grassmannian permutation and $u$ be a $321$-avoiding permutation. Let $\tilde{u}=w^{-1}_{\lambda}$ as in \eqref{uu} such that $u=w_{\mu}\cdot\tilde{u}$. Then we have 
$$c_{v,u}^w=\begin{cases}
EG^{g_{\gamma}}_{\mu}, & 
\text{if $\gamma=w*\gamma_{v,\tilde{u}}$ 
satisfies $\ell(\gamma)=\ell(w)+\ell(\gamma_{v,\tilde{u}})$ and 
$v_\gamma=e$}; \\[1ex]
0,& \text{otherwise}.  
\end{cases}$$
\end{theorem}
\begin{proof}
By setting $t=0$ and replacing $y$ by $-y$ in Theorem \ref{th:triplecoeffi}, we have 
$$\mathfrak{S}_v(x)\cdot
\mathfrak{S}_{\tilde{u}}(x;-y)=
\sum_{w\in S_\infty} c_{v,\tilde{u}}^w(-y)\cdot \mathfrak{S}_w(x),$$
where, by letting $\gamma=w*\gamma_{v,\tilde{u}}$,
$$c_{v,\tilde{u}}^w(-y)
=F_{g_{\gamma}}(y)
=\sum_{\mu}EG^{g_{\gamma}}_{\mu}\cdot
s_{\mu}(y_1,\ldots,y_q).$$
On the other hand, since $u=w_{\mu}\tilde{u}$, by the Cauchy formula of the double Schubert polynomials  \cite{MacDonald}, we have
$$\mathfrak{S}_{\tilde{u}}(x;-y)=\sum_{\text{reduced }\tilde{u}=w^{-1}_\mu\cdot u}
\mathfrak{S}_{u}(x)\cdot
s_{\mu}(y_1,\ldots,y_q).$$
The theorem follows by comparing the coefficients of $s_\mu(y_1,\ldots,y_q)$.
\end{proof}

\begin{example}
Following our running Example \ref{eg:triplep>q} where $v=412356$ and $u=314526$, let $t=0$ and replace $y_i$ by $-y_i$, we have
\begin{align*}
    \mathfrak{S}_{v}(x)\cdot
\mathfrak{S}_{u}(x;-y)
 =\, & \mathfrak{S}_{613425}(x)+ s_{(1)}(y_1,y_2)\mathfrak{S}_{612435}(x) +s_{(1)}(y_1,y_2)\mathfrak{S}_{51342}(x)\\ 
 & + s_{(2)}(y_1,y_2)\mathfrak{S}_{612345}(x)
 +[s_{(2)}(y_1,y_2)+s_{(1,1)}(y_1,y_2)]\mathfrak{S}_{51243}(x)\\
 & +s_{(1,1)}(y_1,y_2)\mathfrak{S}_{41352}(x) + [s_{(3)}(y_1,y_2)+s_{(2,1)}(y_1,y_2)]\mathfrak{S}_{51234}(x)\\
 & +s_{(2,1)}(y_1,y_2)\mathfrak{S}_{41253}(x)+s_{(3,1)}(y_1,y_2)\mathfrak{S}_{4123}(x).
\end{align*}
On the other hand, we have 
\begin{align*}
    \mathfrak{S}_{u}(x;-y)=&\,\mathfrak{S}_{31452}(x)+s_{(1)}(y_1,y_2)\mathfrak{S}_{21453}(x)+s_{(1,1)}(y_1,y_2)\mathfrak{S}_{12453}(x)
    +s_{(2)}(y_1,y_2)\mathfrak{S}_{21354}(x)\\
    &+s_{(2,1)}(y_1,y_2)\mathfrak{S}_{12354}(x)+s_{(3)}(y_1,y_2)\mathfrak{S}_{21345}(x)+s_{(3,1)}(y_1,y_2)\mathfrak{S}_{12345}(x).
\end{align*}
Comparing the coefficients, we get 
\begin{align*}
 \mathfrak{S}_{412356}(x)\mathfrak{S}_{31452}(x)&=\mathfrak{S}_{613425}(x)\\
 \mathfrak{S}_{412356}(x)\mathfrak{S}_{21453}(x)&=\mathfrak{S}_{612435}(x)+\mathfrak{S}_{51342}(x)\\
 \mathfrak{S}_{412356}(x)\mathfrak{S}_{12453}(x)&=\mathfrak{S}_{51243}(x)+\mathfrak{S}_{41352}(x)\\
\mathfrak{S}_{412356}(x)\mathfrak{S}_{21354}(x)&=\mathfrak{S}_{612345}(x)+\mathfrak{S}_{51243}(x)\\
    \mathfrak{S}_{412356}(x)\mathfrak{S}_{12354}(x)&=\mathfrak{S}_{41253}(x)+\mathfrak{S}_{51234}(x)\\
\mathfrak{S}_{412356}(x)\mathfrak{S}_{21345}(x)&=\mathfrak{S}_{51234}(x)\\
\mathfrak{S}_{412356}(x)\mathfrak{S}_{12345}(x)&=\mathfrak{S}_{412356}(x).
\end{align*} 
\end{example}

\appendix

\section{Quiver Representations}
\label{sec:Sorbit}

\def\ddim{\operatorname{\mathbf{dim}}}
\def\vv{\mathbf{v}}
\def\Hom{\operatorname{Hom}}
\def\D#1#2#3#4{\begin{matrix}
{\!\!\xymatrix@!=0pc{
\ar@{}[d]|<<{\phantom{A}}="uu"
\ar@{}[d]|>>>{\phantom{A}}="dd"
\ar"uu";[r]^{#3}
\ar"dd";[dr]_{#4}
&\makebox[0pc][l]{{$#1$}}\\
&\makebox[0pc][l]{{$#2$}}}}\\
\end{matrix}}

\subsection{Quiver representations}
A quiver is nothing but a directed graph. 
Formally, a quiver is $Q=(Q_0,Q_1,s,t)$,  
where $Q_0$ is the set of vertices, $Q_1$ is the set of arrows, and $s,t:Q_1\to Q_0$ are the functions indicating the source and target of an arrow. 
A representation of a quiver $Q$ is a family of finite-dimensional vector spaces $\{V_i\}_{i\in Q_0}$ indexed by vertices equipped with a family of linear maps $\{f_{h}:V_{s(h)}\to V_{t(h)}\}_{h\in Q_1}$ indexed by arrows. 
The category of quiver representations forms an abelian category with morphisms given by commutative diagrams. 
In particular, we can define the notion of isomorphism and direct sums. 
We say a representation is indecomposable if it cannot be written as a direct sum of two nonzero representations. 
For a quiver representation $V=(\{V_i\},\{f_{h}\})$, we define its dimension vector 
$$\ddim V = (\dim V_i)_{i\in Q_0}\in \mathbb{N}^{Q_0}. $$
Recall the following classical theorem  of Gabriel \cite{Gabriel}. 

\begin{theorem}[Gabriel]\label{Gabriel}
If the underlying graph of $Q$ is an ADE Dynkin diagram, then up to isomorphism, there are finitely many indecomposable representations. 
They are bijective to positive roots in the corresponding root system via $\ddim$ if we identify the standard basis of $\mathbb{Z}^{Q_0}$ with the simple roots. 
\end{theorem}

\begin{example}\label{eg:QuiverA}
Let $Q$ be the following quiver 
$$
\stackrel{1}\circ 
\longrightarrow 
\stackrel{2}\circ 
\longrightarrow 
\stackrel{3}\circ 
\longrightarrow \cdots 
\longrightarrow 
\stackrel{\!\!\!\!\!\!N-1\!\!\!\!\!\!}\circ.
$$
The underlying graph is a Dynkin diagram of type $A_{N-1}$. 
The indecomposable representations of $Q$ are 
$$\cdots\longrightarrow 
0
\longrightarrow 
\stackrel{i}{\vphantom{\tfrac12}
\,\mathbb{C}\,}
\stackrel{1}\longrightarrow 
\cdots 
\stackrel{1}\longrightarrow 
\stackrel{j-1}{\vphantom{\tfrac12}
\,\mathbb{C}\,}
\longrightarrow 0\longrightarrow 
\cdots$$
for some $1\leq i<j\leq N$. 
\end{example}

\begin{example}\label{eg:QuiverD}
Let us consider the following quiver $Q$
$$\begin{matrix}
\xymatrix@R=0pc{
\ar@{}[d]|{\displaystyle
\,\mathop{\circ}^1_{}\,}="1"& 
\ar@{}[d]|{\displaystyle
\,\mathop{\circ}^2_{}\,}="2"& 
\ar@{}[d]|{\displaystyle
\,\cdots\,}="dot"& 
\ar@{}[d]|{\displaystyle
\,\mathop{\circ}^n_{}\,}="n"& 
\stackrel{+}\circ
\ar"1";"2"\ar"2";"dot"\ar"dot";"n"
\ar"n";[]\ar"n";[d]
\\
&&&&
\underset{-}\circ
}
\end{matrix}.$$
The underlying graph is a Dynkin diagram of type $D_{n+2}$. 
The indecomposable representations are classified as follows. 
\begin{itemize}
    \item 
    For any connected full subgraph $Q'$ of $Q$, we have an indecomposable representation 
    $(\{V_i\},\{f_{h}\})$ 
    of $Q$ with 
$$V_i=
\begin{cases}
\mathbb{C}, & i\in Q_0',\\
0, & \text{otherwise},\\
\end{cases}\qquad 
f_{h}=
\begin{cases}
[\mathbb{C}\stackrel{1}\to \mathbb{C}], 
& h\in Q_1',\\
0& \text{otherwise}. 
\end{cases}$$

    \item For $1\leq i< j\leq n$, we have an indecomposable representation
\begin{equation}\label{eq:repD2}
\cdots\longrightarrow 0 \longrightarrow 
\stackrel{i}{\vphantom{\tfrac12}
\,\mathbb{C}\,} \stackrel{1}\longrightarrow \cdots \stackrel{1}\longrightarrow
\stackrel{\!\!\!\!\!\!j-1\!\!\!\!\!\!}{\vphantom{\tfrac12}
\,\mathbb{C}\,} 
\stackrel{\left[\begin{subarray}{c}1\\1
\end{subarray}\right]}\longrightarrow \mathbb{C}^2
\stackrel{\left[\begin{subarray}{c}1\,\,0\\0\,\,1
\end{subarray}\right]}\longrightarrow \cdots 
\stackrel{\left[\begin{subarray}{c}1\,\,0\\0\,\,1
\end{subarray}\right]}\longrightarrow  \mathbb{C}^2
\D{\mathbb{C}}{\mathbb{C}.}{[1\,\,0]}{[0\,\,1]}
\end{equation}
\end{itemize}
\end{example}

For a vector $\vv\in\mathbb{N}^{Q_0}$, following Lusztig \cite[Chapter 9]{Lus}, we define 
$$E_{\vv}= E_{\vv}(Q)= \bigoplus_{(i\to j)\in Q_1}
\Hom(\mathbb{C}^{\vv_{i}},\mathbb{C}^{\vv_{j}}),\qquad 
G_{\vv} =G_\vv(Q)= \prod_{i\in Q_0}GL(\mathbb{C}^{\vv_i}). $$
For each element $x\in E_\vv$, we can associate a representation of $Q$ with dimension vector $\vv$. Using this construction, it is tautological to identify
$$\{\text{$G_\vv$-orbits of $E_\vv$}\}
=\left\{
\begin{matrix}
\text{isomorphism classes}\\
\text{of representations of $Q$}\\
\text{with dimension vector $\vv$}
\end{matrix}\right\}.$$
By Krull--Schmidt Theorem, every representation can be written uniquely as a direct sum of indecomposable representations up to isomorphism. 
Thus, the Gabriel's Theorem \ref{Gabriel} and the above identification imply that the $G_\vv$-orbits on $E_\vv$ are bijective to partitions of $\vv$ into positive roots, and in particular they are finite.

For our purpose, we will also need two operations on quivers. 
Let $i\in Q_0$ be a vertex. 
We define the deletion $Q\backslash i$ to be the quiver obtained by removing the vertex $i$ and all arrows incident to $i$. 
We define the contraction $Q/i$ to be the quiver obtained from the deletion $Q\backslash i$ by adding a new arrow $j\to k$ if there is a path $j\to i\to k$ in $Q$. 
For a vector $\vv\in \mathbb{N}^{Q_0}$, we have two maps 
$$
\operatorname{del}: E_\vv(Q) \longrightarrow E_{\bar{\vv}}({Q\backslash i}),\qquad 
\operatorname{con}: E_\vv(Q) \longrightarrow E_{\bar{\vv}}({Q/i}),
$$
where $\bar{\vv}\in \mathbb{N}^{Q_0\setminus \{i\}}$. 
The first map is the natural projection, and the second map is by defining $f_{j\to k}=f_{i\to k}\circ f_{j\to i}$ for the new arrow $(j\to k)\in Q_1$ for path $j\to i \to k$. 
It induces a map 
$$\begin{aligned}
\operatorname{del}: & 
\{\text{$G_\vv(Q)$-orbits of $E_\vv(Q)$}\}
\longrightarrow 
\{\text{$G_{\bar{\vv}}(Q\backslash i)$-orbits of $E_{\bar{\vv}}(Q\backslash i)$}\},\\
\operatorname{con}: & 
\{\text{$G_\vv(Q)$-orbits of $E_\vv(Q)$}\}
\longrightarrow 
\{\text{$G_{\bar{\vv}}(Q/ i)$-orbits of $E_{\bar{\vv}}(Q/ i)$}\}.
\end{aligned}$$
Viewing both sides as isomorphism classes of representations of quivers, the map corresponds to a similar operation on representations. 

\begin{example}\label{eg:deletion}
Let $Q$ be the quiver in Example \ref{eg:QuiverD}. 
Let $\bar{Q}=Q\setminus -$, i.e. 
$$\stackrel{1}\circ 
\longrightarrow 
\stackrel{2}\circ 
\longrightarrow 
\stackrel{3}\circ 
\longrightarrow \cdots 
\longrightarrow 
\stackrel{n}\circ
\longrightarrow 
\stackrel{+}\circ.$$
Under the deletion operation, the representation \eqref{eq:repD2} is mapped to 
\begin{equation}\label{eq:deleteofDrep2}
\begin{aligned}
&
\quad \,
\cdots \longrightarrow 0
\longrightarrow \mathbb{C}
\stackrel{1}\longrightarrow\cdots
\stackrel{1}\longrightarrow
\mathbb{C}
\stackrel{\left[\begin{subarray}{c}1\\1
\end{subarray}\right]}\longrightarrow \mathbb{C}^2
\stackrel{\left[\begin{subarray}{c}1\,\,0\\0\,\,1
\end{subarray}\right]}\longrightarrow \cdots 
\stackrel{\left[\begin{subarray}{c}1\,\,0\\0\,\,1
\end{subarray}\right]}\longrightarrow  \mathbb{C}^2
\stackrel{[1\,\,0]}\longrightarrow
\mathbb{C}\\
& \cong 
\cdots \longrightarrow 0
\longrightarrow \mathbb{C}
\stackrel{1}\longrightarrow \cdots
\stackrel{1}\longrightarrow \mathbb{C}
\stackrel{1}\longrightarrow \,\mathbb{C}\,
\stackrel{1}\longrightarrow \cdots
\stackrel{1}\longrightarrow \,\mathbb{C}\,
\stackrel{1}\longrightarrow \mathbb{C}
\\
& \oplus\,
\cdots \longrightarrow 0
\longrightarrow \,0\,
\longrightarrow \cdots
\longrightarrow \,0\,
\longrightarrow \,\mathbb{C}\,
\stackrel{1}\longrightarrow \cdots
\stackrel{1}\longrightarrow \,\mathbb{C}\,
\longrightarrow 0.
\end{aligned}
\end{equation}
\end{example}

\begin{example}\label{eg:contraction}
Let $Q$ be the quiver in Example \ref{eg:QuiverD}. 
For $1\leq k\leq n-1$, consider $\bar{Q}=Q/k$, i.e. 
$$\begin{matrix}
\xymatrix@R=0pc{
\ar@{}[d]|{\displaystyle
\,\mathop{\circ}^1_{}\,}="1"& 
\ar@{}[d]|{\displaystyle
\,\mathop{\circ}^2_{}\,}="2"& 
\ar@{}[d]|{\displaystyle
\,\cdots\,}="dot"& 
\ar@{}[d]|{\displaystyle
\,\mathop{\circ}^{k-1}_{}\,}="i-1"&&
\ar@{}[d]|{\displaystyle
\,\mathop{\circ}^{k+1}_{}\,}="i+1"& 
\ar@{}[d]|{\displaystyle
\,\cdots\,}="dot2"& 
\ar@{}[d]|{\displaystyle
\,\mathop{\circ}^n_{}\,}="n"& 
\stackrel{+}\circ
\ar"1";"2"\ar"2";"dot"\ar"dot";"i-1"\ar"i-1";"i+1"\ar"i+1";"dot2"\ar"dot2";"n"
\ar"n";[]\ar"n";[d]
\\
&&&&&&&&
\underset{-}\circ
}
\end{matrix}.$$
When $(i,j)=(k,k+1)$, 
under the contraction operation, the representation \eqref{eq:repD2} is mapped to 
$$
\begin{aligned}
&\quad \cdots\longrightarrow 0 \longrightarrow 
\stackrel{\!\!\!\!\!\!k-1\!\!\!\!\!\!}{\vphantom{\dfrac12}
\,0\,} 
-\!\!\!-\!\!\!-\!\!\!-\!\!\!\longrightarrow \stackrel{\!\!\!\!\!\!k+1\!\!\!\!\!\!}{\vphantom{\dfrac12}
\,\mathbb{C}^2\,} 
\stackrel{\left[\begin{subarray}{c}1\,\,0\\0\,\,1
\end{subarray}\right]}\longrightarrow \cdots 
\stackrel{\left[\begin{subarray}{c}1\,\,0\\0\,\,1
\end{subarray}\right]}\longrightarrow  \mathbb{C}^2
\D{\mathbb{C}}{\mathbb{C}}{[1\,\,0]}{[0\,\,1]}\\
& = 
\cdots\longrightarrow 0 \longrightarrow 
\stackrel{\!\!\!\!\!\!k-1\!\!\!\!\!\!}{\vphantom{\dfrac12}
\,0\,} 
-\!\!\!-\!\!\!-\!\!\!-\!\!\!\longrightarrow \stackrel{\!\!\!\!\!\!k+1\!\!\!\!\!\!}{\vphantom{\dfrac12}
\,\,\mathbb{C}\,\,} 
\stackrel{1}\longrightarrow \cdots 
\stackrel{1}\longrightarrow  \,\mathbb{C}\,
\D{\mathbb{C}}{0}{1}{}\\
& \oplus 
\cdots\longrightarrow 0 \longrightarrow 
\stackrel{\!\!\!\!\!\!k-1\!\!\!\!\!\!}{\vphantom{\dfrac12}
\,0\,} 
-\!\!\!-\!\!\!-\!\!\!-\!\!\!\longrightarrow \stackrel{\!\!\!\!\!\!k+1\!\!\!\!\!\!}{\vphantom{\dfrac12}
\,\,\mathbb{C}\,\,} 
\stackrel{1}\longrightarrow \cdots 
\stackrel{1}\longrightarrow  \,\mathbb{C}\,
\D{0}{\mathbb{C}.}{}{1}
\end{aligned}
$$
For other $(i,j)$, we will obtain an indecomposable representation of $\bar{Q}$. 
\end{example}

\subsection{Finiteness}\label{sec:finiteofS}
The purpose of this subsection is to show $S$ is spherical, and prove the $S$-orbits of flag variety $G/B$ are parametrized by preclans. 

Let us take $Q$ to be the quiver in Example \ref{eg:QuiverD}. 
Let us consider the dimension vector $\vv$ with 
$$\vv_+=p,\qquad \vv_-=q,\qquad\text{other }\vv_i=i.$$

Recall that we are using $G=GL(\mathbb{C}^n)$. 
We can naturally identify
$$G/B=\{0=V_0\subset V_1\subset \cdots \subset V_{n-1}\subset V_n=\mathbb{C}^n:\dim V_i=i\},$$
where $1\cdot B$ corresponds to the flag $V_\bullet$ with $V_i=\operatorname{span}(e_1,\ldots,e_i)$. 
Similarly, we can naturally identify $G/S$ with the space of pairs $(K_+,K_-)$ of subspaces of $\mathbb{C}^n$ such that  
$$\dim \mathbb{C}^n/K_+=p,\qquad 
\dim \mathbb{C}^n/K_-=q,$$
and the natural map $\mathbb{C}^n\to \mathbb{C}^n/K_+\oplus \mathbb{C}^n/K_-$ is surjective. 
The element $1\cdot S$ corresponds to the pair $(K_+,K_-)$ such that 
\begin{equation}\label{eq:stdKpm}
K_+=\operatorname{span}(e_{p+1},\ldots,e_{n}),\qquad 
K_-=\operatorname{span}(e_{1},\ldots,e_{n-q}).
\end{equation}

Let $E_\vv^\circ\subset E_\vv$ be the $G_\vv$-invariant open subset of $(f_{ij})_{(i\to j)\in Q_1}$ such that 
\begin{itemize}
    \item the linear map 
    $f_{k,k+1}:\mathbb{C}^{\vv_k}\to \mathbb{C}^{\vv_{k+1}}$ is injective for $k=1,\ldots,n-1$; 
    \item the linear map 
    $f_{n,+}\oplus f_{n,-}:\mathbb{C}^{\vv_n}\to \mathbb{C}^{\vv_+}\oplus \mathbb{C}^{\vv_-}$ is surjective.
\end{itemize} 
Let 
$G_\vv'=\prod_{i\in Q_0\setminus \{n\}}GL(\mathbb{C}^{\vv_i})$ be the subgroup of $G_\vv$. 
Note that $G_\vv=G_\vv'\times GL(\mathbb{C}^n)$. 
The following proposition is direct. 

\begin{prop}\label{prop:identifyG/S}
The action of $G_\vv'$ on $E_\vv^\circ$ is free, and we have an isomorphism of $G$-varieties
$$E_\vv^\circ/G_\vv'\stackrel{\sim}\longrightarrow G/S\times G/B$$
sending $(f_{ij})$ to $(K_+,K_-,V_\bullet)$ with 
$$K_\pm=\ker[\mathbb{C}^{\vv_n}
\stackrel{f_{n,\pm}}\longrightarrow \mathbb{C}^{\vv_\pm}],\qquad 
V_i=\operatorname{im}[
\mathbb{C}^{\vv_i}
\stackrel{f_{i,i+1}}\longrightarrow
\cdots 
\stackrel{f_{n-1,n}}\longrightarrow
\mathbb{C}^{\vv_{n}}].$$
\end{prop}

By the discussion above, we have the following 
$$\begin{aligned}
\{\text{$S$-orbits of $G/B$}\}
& \cong \{\text{$G$-orbits of $G/S\times G/B$}\}\\
& \cong \{\text{$G$-orbits of $E_\vv^\circ/G_\vv'$}\}\\
& \cong \{\text{$G_\vv$-orbits of $E_\vv^\circ$}\}
\hookrightarrow 
\{\text{$G_\vv$-orbits of $E_\vv$}\}.
\end{aligned}$$
As a result, the subgroup $S$ is spherical. 

Let us enumerate the $S$-orbits explicitly. 
By the definition, the $G_\vv$-orbits of $E_\vv^\circ$ correspond to representations $V=(\{V_i\},\{f_{ij}\})$ of $Q$ with $\ddim V=\vv$ and 
\begin{itemize}
    \item the linear map 
    $f_{k,k+1}:V_k\to V_{k+1}$ is injective for $k=1,\ldots,n-1$; 
    \item the linear map 
    $f_{n,+}\oplus f_{n,-}:V_n\to V_+\oplus V_-$ is surjective.
\end{itemize} 
Equivalently, it is a direct sum of indecomposable representations of the same properties. 
Using the classification in Example \ref{eg:QuiverD}, they are 
\eqref{eq:repD2} and 
\begin{gather}
\label{eq:repDcomma}
\cdots
\longrightarrow 0 
\longrightarrow 0 
\longrightarrow 
\stackrel{i}{\vphantom{\tfrac12}
\,\mathbb{C}\,} \stackrel{1}\longrightarrow \cdots \stackrel{1}\longrightarrow
\mathbb{C}
\D{0}{0}{}{}
\\
\label{eq:repD+}
\cdots
\longrightarrow 0 
\longrightarrow 0 
\longrightarrow 
\stackrel{i}{\vphantom{\tfrac12}
\,\mathbb{C}\,} \stackrel{1}\longrightarrow \cdots \stackrel{1}\longrightarrow
\mathbb{C}
\D{\mathbb{C}}{0}{1}{}
\\
\label{eq:repD-}
\cdots
\longrightarrow 0 
\longrightarrow 0 
\longrightarrow 
\stackrel{i}{\vphantom{\tfrac12}
\,\mathbb{C}\,} \stackrel{1}\longrightarrow \cdots \stackrel{1}\longrightarrow
\mathbb{C}
\D{0}{\mathbb{C}}{}{1}
\end{gather}
for $1\leq i\leq n$. 
We can define a preclan as follows:
\begin{itemize}
    \item 
    We create a match between the $i$-th node and the $j$-th node if \eqref{eq:repD2} appears as a summand of $V$. 
    \item 
    We color the $i$-th node by $\clan{,}$ (resp., $\clan{+}$, $\clan{-}$) if \eqref{eq:repDcomma} 
    (resp., \eqref{eq:repD+}, \eqref{eq:repD-}) appears as a summand of $V$. 
\end{itemize}
We leave to readers to check that this assignment is well-defined and it induces a bijection 
$$\{\text{$S$-orbits of $G/B$}\}\cong \{\text{$(p,m,q)$-preclans}\}. 
$$

\begin{example}
Let us consider the clan $\gamma$ in 
\eqref{eq:preclaneg}. 
The corresponding representation is 
$$
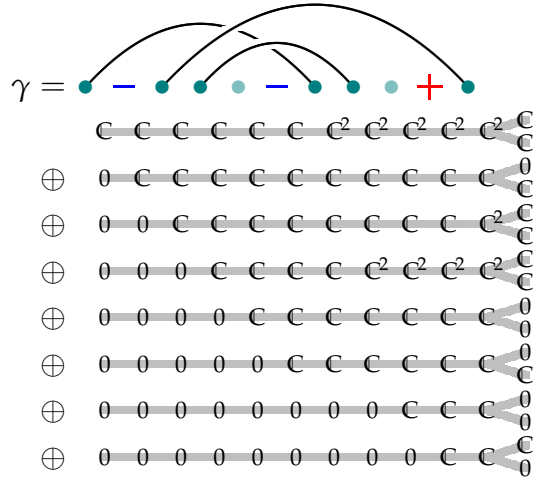
\def\Q#1#2#3{%
\raisebox{-0.3pc}{\makebox[0pc][l]{
    \setlength{\unitlength}{1.2pc}%
    \hspace{0.6pc}%
    \begin{picture}(0,0)
    \color{lightgray}
    \linethickness{0.2\unitlength}
    \qbezier(0.5,0.5)(10.5,0.5)(10.5,0.5)
    \qbezier(10.5,0.5)(11.5,0.2)(11.5,0.2)
    \qbezier(10.5,0.5)(11.5,0.8)(11.5,0.8)
    \color{black}
    \put(0,0.3){\makebox[0pc][l]{#1}}
    \put(11,0.6){\makebox[0pc][l]{#2}}
    \put(11,-.0){\makebox[0pc][l]{#3}}
    \end{picture}}}}
\def\C{\makebox[1.2pc]{$\scriptstyle\mathbb{C}$}}
\def\O{\makebox[1.2pc]{$\scriptstyle0$}}
\def\CC{\makebox[1.2pc]{$\scriptstyle\mathbb{C}^{\!2\!\!\!}$}}
\begin{aligned}
\gamma=&\clan{6-84,-..,+.}\\
&\Q{\C\C\C\C\C\C\CC\CC\CC\CC\CC}{\C}{\C}\\
\oplus&\Q{\O\C\C\C\C\C\C\C\C\C\C}{\O}{\C}\\
\oplus&\Q{\O\O\C\C\C\C\C\C\C\C\CC}{\C}{\C}\\
\oplus&\Q{\O\O\O\C\C\C\C\CC\CC\CC\CC}{\C}{\C}\\
\oplus&\Q{\O\O\O\O\C\C\C\C\C\C\C}{\O}{\O}\\
\oplus&\Q{\O\O\O\O\O\C\C\C\C\C\C}{\O}{\C}\\
\oplus&\Q{\O\O\O\O\O\O\O\O\C\C\C}{\O}{\O}\\
\oplus&\Q{\O\O\O\O\O\O\O\O\O\C\C}{\C}{\O}\\
\end{aligned}
$$
\end{example}

\begin{prop}\label{prop:an-ele-S}
For a $(p,m,q)$-preclan $\gamma$, 
the corresponding $S$-orbit is given by $S\dot{\gamma}B/B$, 
where $\dot{\gamma}$ is constructed in Definition \ref{def:dotgamma}. 
\end{prop}
\begin{proof}
Recall that
$\dot{\gamma}=[\dot{\gamma}_1\,\,\dot{\gamma}_2\,\,\cdots\,\,\dot{\gamma}_n]\in G=GL_n$ for column vectors $\dot{\gamma}_1,\ldots,\dot{\gamma}_n$. For $1\le i\le n$, let $V_i=\operatorname{span}(\dot{\gamma}_1,\ldots,\dot{\gamma}_i)$ and $K_\pm$ be given by \eqref{eq:stdKpm}. 
It suffices to show the preclan associated with the representation of $Q$
$$V_1
\stackrel{\subset}\longrightarrow 
V_2
\stackrel{\subset}\longrightarrow 
\cdots 
\stackrel{\subset}\longrightarrow 
\mathbb{C}^n
\D{\mathbb{C}^n/K_+}{\mathbb{C}^n/K_-,}{}{}
$$
is $\gamma$ itself. 
Assume that there is an $(i,j)$-matching. 
Then $\dot{\gamma}_i=\mathbf{e}_a+\mathbf{e}_b$ and $\dot{\gamma}_j=\mathbf{e}_b$ for $1\leq a\leq p\leq n-q<b\leq n$. 
We can construct a summand isomorphic to \eqref{eq:repD2}
$$\cdots 
\stackrel{=}\longrightarrow
0
\stackrel{\subset}\longrightarrow
\stackrel{i}{\vphantom{\tfrac12}\,C\,}
\stackrel{=}\longrightarrow
\cdots 
\stackrel{=}\longrightarrow
\stackrel{\!\!\!\!\!\!j-1\!\!\!\!\!\!}{\vphantom{\tfrac12}
\,C\,}
\stackrel{\subset}\longrightarrow
D
\stackrel{=}\longrightarrow
\cdots 
\stackrel{=}\longrightarrow
D
\D{D/B}{D/A}{}{}
$$
where $A=\operatorname{span}(\mathbf{e}_a)$, $B=\operatorname{span}(\mathbf{e}_b)$, 
$C=\operatorname{span}(\mathbf{e}_a+\mathbf{e}_b)$ and 
$D=\operatorname{span}(\mathbf{e}_a,\mathbf{e}_b)$. 
Similarly, for each unmatched node $\clan{,}$, 
$\clan{+}$ or  
$\clan{-}$, there is a summand isomorphic to \eqref{eq:repDcomma} 
\eqref{eq:repD+} or \eqref{eq:repD-}, respectively. 
\end{proof}

\subsection{Weak order}\label{sec:weakordApp}
The main purpose of this subsection is to prove Theorem \ref{th:weakordApp} on the weak order of preclans \eqref{eq:clanweak1}, \eqref{eq:clanweak2}, \eqref{eq:clanweak3}.

Let $\check{Q}= Q/k$ be the contraction considered in Example \ref{eg:contraction}. 
Let us use a check $\check{*}$ to denote the corresponding object for the quiver $\check{Q}$ obtained by composing the path. 
Recall that we can naturally identify
$$G/P_k=\{0\subset V_1\subset \cdots \subset 
V_{k-1}\subset V_{k+1}\subset \cdots \subset V_{n-1}\subset \mathbb{C}^n:\dim V_i=i\}.$$
The following proposition is similar to Proposition \ref{prop:identifyG/S}.

\begin{prop}The action of $\check{G}_{\check{\vv}}'$ on $\check{E}_{\check{\vv}}^\circ$ is free, and we have an isomorphism of $G$-varieties
$$
\check{E}_{\check{\vv}}^\circ/\check{G}_{\check{\vv}}'\stackrel{\sim}\longrightarrow G/S\times G/P_k. $$
\end{prop}

Moreover, the contraction operation $\operatorname{con}:E_\vv^\circ\to \check{E}_{\check{\vv}}^\circ$ induces the following commutative diagram
$$
\xymatrix{
E_\vv^\circ/G_\vv'
\ar[r]^-{\sim}\ar[d]&
G/S\times G/B\ar[d]^{\operatorname{id}\times \pi_k}\\
\check{E}_{\check{\vv}}^\circ/\check{G}_{\check{\vv}}'\ar[r]^-{\sim} & G/S\times G/P_k.
}$$
It reduces to studying the fiber of the left vertical map, i.e., the contraction operation. 
Since the contraction operation commutes with direct sums, we can ignore the indecomposable summands $V$ with $(\ddim V)_{k+1}=(\ddim V)_{k-1}$. 
Since $(\ddim V)_{k+1}=(\ddim V)_{k-1}+2$, it reduces to classify the representations of $\check{Q}$ of the form 
$V_1\oplus V_2$ 
where each $V_i$ is indecomposable with $(\ddim V_i)_{k+1}=(\ddim V_i)_{k-1}+1$. 
Thus Theorem \ref{th:weakordApp} follows from a case-by-case check. 
We will only illustrate three cases to show the main idea.

\begin{example}
Consider the representation of $\check{Q}$
$$
\left(\cdots\longrightarrow 0 \longrightarrow 
\stackrel{\!\!\!\!\!\!k-1\!\!\!\!\!\!}{\vphantom{\dfrac12}
\,0\,} 
-\!\!\!-\!\!\!-\!\!\!-\!\!\!\longrightarrow \stackrel{\!\!\!\!\!\!k+1\!\!\!\!\!\!}{\vphantom{\dfrac12}
\,\mathbb{C}\,} 
\stackrel{1}\longrightarrow \cdots 
\stackrel{1}\longrightarrow  \mathbb{C}
\D{\mathbb{C}}{0}{1}{}\quad\right)^{\oplus 2}.
$$
The fiber is the space of $Q$-representations
$$
\begin{matrix}
\def\idd{\left[\begin{subarray}{c}1\,\,0\\0\,\,1
\end{subarray}\right]}
\cdots\longrightarrow 0 \longrightarrow 
\stackrel{\!\!\!\!\!\!k-1\!\!\!\!\!\!}{\vphantom{\dfrac12}\,0\,}
\longrightarrow 
\stackrel{k}{\vphantom{\dfrac12}
\,V\,}
\stackrel{\subset}\longrightarrow 
\stackrel{\!\!\!\!\!\!k+1\!\!\!\!\!\!}{\vphantom{\dfrac12}
\,\mathbb{C}^2\,}
\stackrel{\idd}\longrightarrow  
\cdots 
\stackrel{\idd}\longrightarrow  
\mathbb{C}^2
\D{\mathbb{C}^2}{0.}{\idd}{\vphantom{\idd}}\\[-4ex]
\end{matrix}$$
where $V$ is a one-dimensional subspace of $\mathbb{C}^2$.
For each $V$, the representation is   isomorphic to 
$$\begin{aligned}
&\quad \cdots\longrightarrow 0 \longrightarrow 
\stackrel{\!\!\!\!\!\!k-1\!\!\!\!\!\!}{\vphantom{\dfrac12}\,0\,}
\longrightarrow 
\stackrel{k}{\vphantom{\dfrac12}
\,\mathbb{C}\,}
\stackrel{1}\longrightarrow 
\stackrel{\!\!\!\!\!\!k+1\!\!\!\!\!\!}{\vphantom{\dfrac12}
\,\mathbb{C}\,}
\stackrel{1}\longrightarrow  
\cdots 
\stackrel{1}\longrightarrow  
\mathbb{C}
\D{\mathbb{C}}{0}{1}{}\\
& \oplus\cdots\longrightarrow 0 \longrightarrow 
\stackrel{\!\!\!\!\!\!k-1\!\!\!\!\!\!}{\vphantom{\dfrac12}\,0\,}
\longrightarrow 
\stackrel{k}{\vphantom{\dfrac12}
\,\,0\,\,}
\stackrel{}\longrightarrow 
\stackrel{\!\!\!\!\!\!k+1\!\!\!\!\!\!}{\vphantom{\dfrac12}
\,\mathbb{C}\,}
\stackrel{1}\longrightarrow  
\cdots 
\stackrel{1}\longrightarrow  
\mathbb{C}
\D{\mathbb{C}}{0.}{1}{} 
\end{aligned}$$
That is, the orbit corresponding to 
$$\gamma = \clan{\dots\dots++\dots\dots}$$
is of type (I). 
\end{example}

\begin{example}
Consider the representation of $\check{Q}$
$$\begin{aligned}
& \quad \cdots\longrightarrow 0 \longrightarrow 
\stackrel{\!\!\!\!\!\!k-1\!\!\!\!\!\!}{\vphantom{\dfrac12}
\,0\,} 
-\!\!\!-\!\!\!-\!\!\!-\!\!\!\longrightarrow \stackrel{\!\!\!\!\!\!k+1\!\!\!\!\!\!}{\vphantom{\dfrac12}
\,\mathbb{C}\,} 
\stackrel{1}\longrightarrow \cdots 
\stackrel{1}\longrightarrow  \mathbb{C}
\D{\mathbb{C}}{0}{1}{}\\
& \oplus
\cdots\longrightarrow 0 \longrightarrow 
\stackrel{\!\!\!\!\!\!k-1\!\!\!\!\!\!}{\vphantom{\dfrac12}
\,0\,} 
-\!\!\!-\!\!\!-\!\!\!-\!\!\!\longrightarrow \stackrel{\!\!\!\!\!\!k+1\!\!\!\!\!\!}{\vphantom{\dfrac12}
\,\mathbb{C}\,} 
\stackrel{1}\longrightarrow \cdots 
\stackrel{1}\longrightarrow  \mathbb{C}
\D{0}{\mathbb{C}}{}{1}
\end{aligned}$$
The fiber is the space of $Q$-representations 
$$\def\idd{\left[\begin{subarray}{c}1\,\,0\\0\,\,1
\end{subarray}\right]}
\cdots\longrightarrow 0 \longrightarrow 
\stackrel{\!\!\!\!\!\!k-1\!\!\!\!\!\!}{\vphantom{\dfrac12}\,0\,}
\longrightarrow 
\stackrel{k}{\vphantom{\dfrac12}
\,V\,}
\stackrel{\subset}\longrightarrow 
\stackrel{\!\!\!\!\!\!k+1\!\!\!\!\!\!}{\vphantom{\dfrac12}
\,\mathbb{C}^2\,}
\stackrel{\idd}\longrightarrow  
\cdots 
\stackrel{\idd}\longrightarrow  
\mathbb{C}^2
\D{\mathbb{C}}{\mathbb{C}}{[1\,\,0]}{[0\,\,1]}
$$
where $V$ is a one-dimensional subspace of $\mathbb{C}^2$. 
As a representation we can replace 
$$
\cdots\longrightarrow
\stackrel{\!\!\!\!\!\!k-1\!\!\!\!\!\!}{\vphantom{\dfrac12}\,0\,}
\longrightarrow 
\stackrel{k}{\vphantom{\dfrac12}
\,V\,}
\stackrel{\subset}\longrightarrow 
\stackrel{\!\!\!\!\!\!k+1\!\!\!\!\!\!}{\vphantom{\dfrac12}
\,\mathbb{C}^2\,}
\longrightarrow\cdots
\qquad \text{ by }\qquad 
\cdots\longrightarrow
\stackrel{\!\!\!\!\!\!k-1\!\!\!\!\!\!}{\vphantom{\dfrac12}\,0\,}
\longrightarrow 
\stackrel{k}{\vphantom{\dfrac12}
\,\mathbb{C}\,}
\stackrel{\left[\begin{subarray}{c}a\\b
\end{subarray}\right]}\longrightarrow 
\stackrel{\!\!\!\!\!\!k+1\!\!\!\!\!\!}{\vphantom{\dfrac12}
\,\mathbb{C}^2\,}
\longrightarrow\cdots$$
if $V=\operatorname{span}(\left[\begin{subarray}{c}a\\b \end{subarray}\right]\neq 0)$. 
When $V = \mathbb{C}\oplus 0$, i.e. $b=0$, the representation is isomorphic to 
$$\begin{aligned}
&\quad \cdots\longrightarrow 0 \longrightarrow 
\stackrel{\!\!\!\!\!\!k-1\!\!\!\!\!\!}{\vphantom{\dfrac12}\,0\,}
\longrightarrow 
\stackrel{k}{\vphantom{\dfrac12}
\,\mathbb{C}\,}
\stackrel{1}\longrightarrow 
\stackrel{\!\!\!\!\!\!k+1\!\!\!\!\!\!}{\vphantom{\dfrac12}
\,\mathbb{C}\,}
\stackrel{1}\longrightarrow  
\cdots 
\stackrel{1}\longrightarrow  
\mathbb{C}
\D{\mathbb{C}}{0}{1}{}\\
& \oplus\cdots\longrightarrow 0 \longrightarrow 
\stackrel{\!\!\!\!\!\!k-1\!\!\!\!\!\!}{\vphantom{\dfrac12}\,0\,}
\longrightarrow 
\stackrel{k}{\vphantom{\dfrac12}
\,\,0\,\,}
\stackrel{}\longrightarrow 
\stackrel{\!\!\!\!\!\!k+1\!\!\!\!\!\!}{\vphantom{\dfrac12}
\,\mathbb{C}\,}
\stackrel{1}\longrightarrow  
\cdots 
\stackrel{1}\longrightarrow  
\mathbb{C}
\D{0}{\mathbb{C}.}{}{1} 
\end{aligned}$$
When $V=0\oplus \mathbb{C}$, i.e., $a=0$, we have similar decomposition with the role of $\clan{+}$ and $\clan{-}$ switched.
When $a,b\neq 0$, the representation is isomorphic to the representation with $a=b=1$, i.e. \eqref{eq:repD2} with $(i,j)=(k,k+1)$. 
In summary, we have
$$\begin{array}{c@{\qquad}c@{\qquad}c@{\qquad}c}
V & V=\mathbb{C}\oplus 0 & V=0\oplus \mathbb{C}& \text{otherwise}\\
\text{preclan} & \clan{\dots\dots+-\dots\dots} & 
\clan{\dots\dots-+\dots\dots} & 
\clan{\dots\dots1.\dots\dots} \\
\text{type} & \rm (IIIa) & \rm (IIIa) & \rm (IIIb)
\end{array}$$
\end{example}

\begin{example}
Consider the representation of $\check{Q}$
$$\begin{aligned}
& \quad \cdots\longrightarrow \,0\, \longrightarrow 
\stackrel{\!\!\!\!\!\!k-1\!\!\!\!\!\!}{\vphantom{\dfrac12}
\,\,0\,\,} 
-\!\!\!-\!\!\!-\!\!\!-\!\!\!\longrightarrow \stackrel{\!\!\!\!\!\!k+1\!\!\!\!\!\!}{\vphantom{\dfrac12}
\,\mathbb{C}\,} 
\stackrel{1}\longrightarrow \cdots 
\stackrel{1}\longrightarrow  \,\mathbb{C}\,
\D{0}{0}{}{}\\
& \oplus
\def\idd{\left[\begin{subarray}{c}1\,\,0\\0\,\,1
\end{subarray}\right]}
\cdots\longrightarrow \mathbb{C} 
\stackrel{1}\longrightarrow 
\stackrel{\!\!\!\!\!\!k-1\!\!\!\!\!\!}{\vphantom{\dfrac12}
\,\mathbb{C}\,} 
\stackrel{\left[\begin{subarray}{c}1\\1\end{subarray}\right]}{
-\!\!\!-\!\!\!-\!\!\!-\!\!\!\longrightarrow} \stackrel{\!\!\!\!\!\!k+1\!\!\!\!\!\!}{\vphantom{\dfrac12}
\,\mathbb{C}^2\,} 
\stackrel{\idd}\longrightarrow \cdots 
\stackrel{\idd}\longrightarrow  \mathbb{C}^2
\D{\mathbb{C}}{\mathbb{C}}{[1\,\,0]}{[0\,\,1]}
\end{aligned}$$
The fiber is the space of $Q$-representations 
$$
\def\idd{\left[\begin{subarray}{c}1\,\,0\,\,0\\0\,\,1\,\,0\\0\,\,0\,\,1\end{subarray}\right]}
\cdots\longrightarrow \mathbb{C} \stackrel{1}\longrightarrow 
\stackrel{\!\!\!\!\!\!k-1\!\!\!\!\!\!}{\vphantom{\dfrac12}\,\mathbb{C}\,}
\stackrel{\left[\begin{subarray}{c}0\\1\\1\end{subarray}\right]}{\longrightarrow}
\stackrel{k}{\vphantom{\dfrac12}
\,V\,}
\stackrel{\subset}\longrightarrow 
\stackrel{\!\!\!\!\!\!k+1\!\!\!\!\!\!}{\vphantom{\dfrac12}
\,\mathbb{C}^3\,}
\stackrel{\idd}\longrightarrow  
\cdots 
\stackrel{\idd}\longrightarrow  
\mathbb{C}^3
\D{\mathbb{C}}{\mathbb{C}}{[0\,\,1\,\,0]}{[0\,\,0\,\,1]}
$$
where $V$ is a two-dimensional subspace of $\mathbb{C}^3$ containing $\left[\begin{subarray}{c}0\\1\\1\end{subarray}\right]$. 
When   $V = \operatorname{span}(\left[\begin{subarray}{c}0\\1\\1\end{subarray}\right],\left[\begin{subarray}{c}1\\0\\0\end{subarray}\right])$, the representation is isomorphic to 
$$\begin{aligned}
& \quad \cdots\longrightarrow \,0\, \longrightarrow 
\stackrel{\!\!\!\!\!\!k-1\!\!\!\!\!\!}{\vphantom{\dfrac12}
\,\,0\,\,} 
\longrightarrow 
\stackrel{k}{\vphantom{\dfrac12}
\,\,\mathbb{C}\,\,} 
\longrightarrow 
\stackrel{\!\!\!\!\!\!k+1\!\!\!\!\!\!}{\vphantom{\dfrac12}
\,\mathbb{C}\,} 
\stackrel{1}\longrightarrow \cdots 
\stackrel{1}\longrightarrow  \,\mathbb{C}\,
\D{0}{0}{}{}\\
& \oplus
\def\idd{\left[\begin{subarray}{c}1\,\,0\\0\,\,1
\end{subarray}\right]}
\cdots\longrightarrow \mathbb{C} 
\stackrel{1}\longrightarrow 
\stackrel{\!\!\!\!\!\!k-1\!\!\!\!\!\!}{\vphantom{\dfrac12}
\,\mathbb{C}\,} 
\stackrel{1}\longrightarrow
\stackrel{k}{\vphantom{\dfrac12}
\,\mathbb{C}\,} 
\stackrel{\left[\begin{subarray}{c}1\\1\end{subarray}\right]}{
\longrightarrow} \stackrel{\!\!\!\!\!\!k+1\!\!\!\!\!\!}{\vphantom{\dfrac12}
\,\mathbb{C}^2\,} 
\stackrel{\idd}\longrightarrow \cdots 
\stackrel{\idd}\longrightarrow  \mathbb{C}^2
\D{\mathbb{C}}{\mathbb{C}.}{[1\,\,0]}{[0\,\,1]}
\end{aligned}$$
Otherwise when $\left[\begin{subarray}{c}1\\0\\0\end{subarray}\right]\notin V$, the representation is isomorphic to 
$$\begin{aligned}
& \quad \cdots\longrightarrow \,0\, \longrightarrow 
\stackrel{\!\!\!\!\!\!k-1\!\!\!\!\!\!}{\vphantom{\dfrac12}
\,\,\,0\,\,\,} 
\longrightarrow 
\stackrel{k}{\vphantom{\dfrac12}
\,\,\,0\,\,\,} 
\longrightarrow 
\stackrel{\!\!\!\!\!\!k+1\!\!\!\!\!\!}{\vphantom{\dfrac12}
\,\mathbb{C}\,} 
\stackrel{1}\longrightarrow \cdots 
\stackrel{1}\longrightarrow  \,\mathbb{C}\,
\D{0}{0}{}{}\\
& \oplus
\def\idd{\left[\begin{subarray}{c}1\,\,0\\0\,\,1
\end{subarray}\right]}
\cdots\longrightarrow \mathbb{C} 
\stackrel{1}\longrightarrow 
\stackrel{\!\!\!\!\!\!k-1\!\!\!\!\!\!}{\vphantom{\dfrac12}
\,\mathbb{C}\,} 
\stackrel{\left[\begin{subarray}{c}1\\1\end{subarray}\right]}\longrightarrow
\stackrel{k}{\vphantom{\dfrac12}
\,\mathbb{C}^2\,} 
\stackrel{\idd}{
\longrightarrow} \stackrel{\!\!\!\!\!\!k+1\!\!\!\!\!\!}{\vphantom{\dfrac12}
\,\mathbb{C}^2\,} 
\stackrel{\idd}\longrightarrow \cdots 
\stackrel{\idd}\longrightarrow  \mathbb{C}^2
\D{\mathbb{C}}{\mathbb{C}.}{[1\,\,0]}{[0\,\,1]}
\end{aligned}$$
In summary, we have
$$\begin{array}{c@{\qquad}c@{\qquad}c@{\qquad}c}
V & \operatorname{span}(\left[\begin{subarray}{c}0\\1\\1\end{subarray}\right],\left[\begin{subarray}{c}1\\0\\0\end{subarray}\right])
& \text{otherwise}\\
\text{preclan} & \clan{3\dots,.\dots\dots} & 
\clan{2\dots.,\dots\dots} \\
\text{type} & \rm (IIa) & \rm (IIb) 
\end{array}$$
\end{example}

\subsection{Richardson variety}\label{sec:RichApp}
Since $S\subseteq P$, we have a natural map
$$
\begin{aligned}
\{\text{$S$-orbits of $G/B$}\}
& \to 
\{\text{$P$-orbits of $G/B$}\}\\
& \cong \{\text{$p$-inverse Grassmannian in $S_n$}\}.\\
\end{aligned}$$
The purpose of this subsection is to study the map.
 
Let $\bar{Q}= Q\backslash -$ the deletion considered in Example \ref{eg:deletion}. 
Let us use a bar $\bar{*}$ to denote the corresponding object for the quiver $\bar{Q}$ obtained by omitting the component not in $\bar{Q}$. 
For example, let $\bar{E}_{\bar{\vv}}^\circ\subset \bar{E}_{\bar{\vv}}$ be the $\bar{G}_{\bar{\vv}}$-invariant open subset of $(f_{ij})_{(i\to j)\in Q_1}$ such that 
\begin{itemize}
    \item the linear map 
    $f_{k,k+1}:\mathbb{C}^{\vv_k}\to \mathbb{C}^{\vv_{k+1}}$ is injective for $k=1,\ldots,n-1$; 
    \item the linear map 
    $f_{n,+}:\mathbb{C}^{\vv_n}\to \mathbb{C}^{\vv_+}$ is surjective.
\end{itemize} 
We can naturally identify 
$$G/P =\{K_+\subseteq \mathbb{C}^n:\dim \mathbb{C}^n/K_+=p\}$$
with $1\cdot P$ corresponding to $K_+$ determined by \eqref{eq:stdKpm}. 
The following proposition is similar to Proposition \ref{prop:identifyG/S}.

\begin{prop}
The action of $\bar{G}_{\bar{\vv}}'$ on $\bar{E}_{\bar{\vv}}^\circ$ is free, and we have an isomorphism of $G$-varieties
$$
\bar{E}_{\bar{\vv}}^\circ/\bar{G}_{\bar{\vv}}'\stackrel{\sim}\longrightarrow G/P\times G/B. $$
\end{prop}

Similarly, we have the following 
$$\begin{aligned}
\{\text{$P$-orbits of $G/B$}\}
& \cong \{\text{$G$-orbits of $G/P\times G/B$}\}\\
& \cong \{\text{$G$-orbits of $\bar{E}_{\bar{\vv}}^\circ/\bar{G}_{\bar\vv}'$}\}\\
& \cong \{\text{$\bar{G}_{\bar{\vv}}$-orbits of $\bar{E}_{\bar{\vv}}^\circ$}\}
\hookrightarrow 
\{\text{$\bar{G}_{\bar{\vv}}$-orbits of $\bar{E}_{\bar{\vv}}$}\}.
\end{aligned}$$
By  definition, the $\bar{G}_{\bar{\vv}}$-orbits of $\bar{E}_{\bar{\vv}}^\circ$ correspond to representations $V=(\{V_i\},\{f_{ij}\})$ of $\bar{Q}$ with $\ddim V=\bar{\vv}$ and 
\begin{itemize}
    \item the linear map 
    $f_{k,k+1}:V_k\to V_{k+1}$ is injective for $k=1,\ldots,n-1$; 
    \item the linear map 
    $f_{n,+}:V_n\to V_+$ is surjective.
\end{itemize} 
Equivalently, it is a direct sum of indecomposable representations of the same property. 
By Example \eqref{eg:QuiverA}, they are 
\begin{gather}
\label{eq:repA0}
\cdots
\longrightarrow 0 
\longrightarrow 0 
\longrightarrow 
\stackrel{i}{\vphantom{\tfrac12}
\,\mathbb{C}\,} \stackrel{1}\longrightarrow \cdots \stackrel{1}\longrightarrow
\stackrel{n}{\vphantom{\tfrac12}
\,\mathbb{C}\,} 
\longrightarrow 
\stackrel{+}{\vphantom{\tfrac12}
\,0\,}\\
\label{eq:repA1}
\cdots
\longrightarrow 0 
\longrightarrow 0 
\longrightarrow 
\stackrel{i}{\vphantom{\tfrac12}
\,\mathbb{C}\,} \stackrel{1}\longrightarrow \cdots \stackrel{1}\longrightarrow
\stackrel{n}{\vphantom{\tfrac12}
\,\mathbb{C}\,} 
\longrightarrow 
\stackrel{+}{\vphantom{\tfrac12}
\,\mathbb{C}\,} 
\end{gather}
for $1\leq i\leq n$. 
We can define a $p$-subset $A\subset [n]$ of the indices of $i\in [n]$ such that \eqref{eq:repA1} appears as a summand of $V$. 
We have 
$$\{\text{$P$-orbits of $G/B$}\}
\cong \{\text{$p$-subsets of $[n]$}\}.$$
The proof of the following Proposition is similar to the proof of Proposition \ref{prop:an-ele-S}. 

\begin{prop}\label{prop:Porbis}
For a $p$-subset $A$, 
the corresponding $P$-orbit is given by $PvB/B$, 
where $v\in S_n$ is the $p$-inverse Grassmannian permutation such that  $v^{-1}([p])=A$. 
\end{prop}

The map $\operatorname{del}:E_\vv^\circ\to \bar{E}_{\bar{\vv}}^\circ$ induces the following commutative diagrams
$$\xymatrix{
E_\vv^\circ/G_\vv'
\ar[r]^-{\sim}\ar[d]&
G/S\times G/B\ar[d]\\
\bar{E}_{\bar{\vv}}^\circ/\bar{G}_{\bar{\vv}}'\ar[r]^-{\sim} & G/P\times G/B.
}\qquad 
\xymatrix{
\{\text{$S$-orbits of $G/B$}\}
\ar[r]^-{\sim}\ar[d]& 
\{\text{$(p,m,q)$-preclans}\}\ar[d]\\
\{\text{$P$-orbits of $G/B$}\}
\ar[r]^-{\sim}& 
\{\text{$p$-subsets of $[n]$}\}.}$$
Under the deletion operation, 
$$
\eqref{eq:repDcomma},\eqref{eq:repD-}\longmapsto 
\eqref{eq:repA0}, \qquad 
\eqref{eq:repD+} \longmapsto \eqref{eq:repA1},
\qquad 
\eqref{eq:repD2} \longmapsto 
\eqref{eq:deleteofDrep2}.
$$
Thus the right vertical map above is given by sending a preclan $\gamma$ to the $p$-subset $A$ of the indices of $\clan{+}$ and left-ends. We thus have the following. 

\begin{prop}\label{prop:StoP}
The $S$-orbit  $S\dot{\gamma}B/B$ is contained in $Pv_\gamma B/B$. 
\end{prop}

Let us deal with $Q$-orbits. 
We can naturally identify 
$$G/Q =\{K_-\subset \mathbb{C}^n:\dim \mathbb{C}^n/K_-=q\}$$
with $1\cdot Q$ corresponding to $K_-$ determined by \eqref{eq:stdKpm}. 
Similarly, we have
$$\{\text{$Q$-orbits of $G/B$}\}
\cong \{\text{$q$-subsets of $[n]$}\}.$$
Recall $w_0=n\cdots 21\in S_n$ is the longest element. 

\begin{prop}
For a $q$-subset $A$, 
the corresponding $Q$-orbit is given by $Qw_0uB/B$, where $u\in S_n$ is the $q$-inverse Grassmannian permutation such that $u^{-1}([q])=A$. 
\end{prop}

Note that it is necessary to insert a $w_0$ in front of $u$. 
This is because, distinct from the case $K_+$, the subspace $K_-$ is the span of the first $n-q$ standard basis vectors. 
We have 
$$\xymatrix{
\{\text{$S$-orbits of $G/B$}\}
\ar[r]^-{\sim}\ar[d]& 
\{\text{$(p,m,q)$-preclans}\}\ar[d]\\
\{\text{$Q$-orbits of $G/B$}\}
\ar[r]^-{\sim}& 
\{\text{$q$-subsets of $[n]$}\}.}$$
The right vertical map is given by sending a preclan $\gamma$ to the $q$-subset $A$ of the indices of $\clan{-}$ and left-ends. 

\begin{prop}\label{prop:StoQ}
The $S$-orbit  $S\dot{\gamma}B/B$ is contained in $Qw_0u_\gamma B/B$. 
\end{prop}

\end{document}